\documentclass[11pt]{article}
\usepackage{geometry}    %12pt
\usepackage[english]{babel}

\usepackage[utf8]{inputenc}
\usepackage{tikz,tikz-3dplot }\usetikzlibrary{patterns}
\usepackage{blindtext,datetime,nicefrac,subcaption,graphicx,wrapfig,caption,bbm,
xspace,amsmath,amsthm,amssymb,enumerate,color,esint,paralist,comment}

\usepackage[hidelinks,colorlinks=false,citecolor=false,
linkcolor=false,urlcolor=false]{hyperref}

\usepackage{appendix}
\numberwithin{equation}{section}

\usepackage[capitalise,nameinlink,noabbrev]{cleveref}
\newtheorem{theorem}{Theorem}[section]
\newtheorem*{theorem*}{Theorem}
\newtheorem{lemma}[theorem]{Lemma}

\newtheorem{corollary}[theorem]{Corollary}

\theoremstyle{definition}
\newtheorem{setting}[theorem]{Setting}

\newtheorem*{remark*}{Remark}

\newcommand{\1}{\mathbbm{1}}

\newtheorem*{examples*}{Examples}

\crefname{lemma}{Lemma}{Lemmas}
\crefname{proposition}{Proposition}{Propositions}
\crefname{theorem}{Theorem}{Theorems}
\crefname{setting}{Setting}{Settings}
\crefname{claim}{Claim}{Claims}
\crefname{enumi}{Item}{Items}
\creflabelformat{enumi}{(#2\textup{#1}#3)}
\crefname{equation}{}{}
\crefname{assumption}{Setting}{Settings}
\crefname{definition}{Definition}{Definitions}
\crefname{remark}{Remark}{Remarks}
\crefname{figure}{Fig.}{Figs.}

\crefname{chapter}{Chapter}{Chapters}
\Crefname{chapter}{chapter}{chapters}

\crefname{section}{Section}{Sections}
\Crefname{section}{section}{sections}

\crefname{subsection}{Section}{Sections}
\Crefname{subsection}{Section}{Sections}
\Crefname{theorem}{theorem}{theorems}
\Crefname{lemma}{lemma}{lemmas}
\crefname{step}{Step}{Steps}
\newcommand{\pconst}{M}

\newcommand{\defeq}{\leftarrow}
\newcommand{\rdown}{\underline{r}}
\newcommand{\rup}{\overline{r}}
\newcommand{\etan}[1]{E_{#1}^\tau }
\newcommand{\etanT}[1]{V_{#1}^\tau }
\newcommand{\mean}[2]{\left\langle#1\right\rangle_{#2}}
\newcommand{\enor}[1]{E_{#1}^\nu}
\newcommand{\enorT}[1]{V_{#1}^{\nu}}
\newcommand{\efull}[1]{E_{#1}^{\#}}

\newcommand{\constMac}{\mathfrak{m}}
\newcommand{\constPI}{c_{\mathrm{PI}}}
\newcommand{\disDiff}[2]{\mathsf{D}^{#1}_{#2}}
\newcommand{\Laplace}{\mathop{}\!\mathbin\bigtriangleup}
\newcommand{\ballO}[2]{\mathbbm{B}^d_r}
\newcommand{\ballC}[2]{\bar{\mathbbm{B}}^d_r}
\newcommand{\ballB}[2]{\partial\mathbbm{B}^d_r}

\newcommand{\unit}[2]{\mathbf{e}^{#1}_{#2}}

\newcommand{\fourier}[1]{\ensuremath{\mathcal{F}\left(#1\right)}}
\newcommand{\DM}{\mathcal{D}}
\newcommand{\NM}{\mathcal{N}}

 \newcommand{\ima}{\mathbf{i}}
\renewcommand{\tilde}{\widetilde}

\newcommand{\N}{\mathbbm{N}}
\newcommand{\E}{\mathbbm{E}}
\newcommand{\R}{\mathbbm{R}}
\newcommand{\K}{\mathbbm{K}}
\newcommand{\Z}{\mathbbm{Z}}
\newcommand{\I}{\mathbbm{I}}
\newcommand{\C}{\mathbbm{C}}\renewcommand{\P}{\mathbbm{P}}
\newcommand{\perfct}[2]{\mathcal{P}_{#1}\big(#2,\C\big)}

\renewcommand{\epsilon}{\varepsilon}

\newcommand{\disint}[2]{\ensuremath{[#1,#2]\cap\Z}}

\DeclareMathOperator{\var}{\mathsf{var}}
\DeclareMathOperator{\lvar}{\mathsf{locvar}}

\title{An $L^p$-comparison, $p\in (1,\infty)$, on the finite differences of a discrete harmonic function at the boundary of a discrete box}

\author{Tuan Anh Nguyen\\Faculty of Mathematics, University of Duisburg-Essen,\\
45117 Essen, Germany, e-mail: \url{tuan.nguyen@uni-due.de}
}

\begin{document}\maketitle
\begin{abstract}
It is well-known that 
for a harmonic function $u$ defined on the unit ball of the $d$-dimensional Euclidean space, $d\geq 2$,
the tangential and normal component of the gradient $\nabla u$
on the sphere are comparable by means of the $L^p$-norms,
$p\in(1,\infty)$, up to multiplicative constants that depend only on $d,p$. This paper formulates and proves a discrete analogue of this result for discrete harmonic functions defined on a  discrete box on the $d$-dimensional lattice
with multiplicative constants that do not depend on the size of the box.
\end{abstract}
\fbox{\parbox{.95\textwidth}{
\emph{Comment}: 3 figures}
}
\tableofcontents
\section{Introduction}
This paper formulates and proves a discrete analogue of a classical
result
in the continuum setting which states that the tangential and normal component
of the gradient of a harmonic function on the boundary of a domain are comparable by means of $L^p$-norms, $p\in (1,\infty)$. 
For convenience  we give a simplified version of this result in \cref{d1} below.
For complete formulations and proofs we refer the reader to Maergoiz~\cite{Mae73} (see, e.g., Theorems 1 and 2), Mikhlin~\cite{Mik65} (see § 44 p. 208), and Bella, Fehrman, and Otto~\cite{BFO18} (see Lemma 4).
This result can be viewed as a \emph{stability estimate for harmonic extensions} of given Dirichlet or Neumann boundary conditions. 
Thus, it
 plays an important role in the proof of a Liouville theorem for a class of elliptic equations with degenerate random coefficient fields
(see formulas (40) and (41)~in~\cite{BFO18}) where the so-called idea \emph{perturbing around the homogenized coefficients} is realized 
by  \emph{harmonic extensions} of given boundary conditions.
The discrete analogue that we want to show here can be applied to prove a Liouville theorem for the random conductance model under degenerate conditions, which is the discrete analogue of \cite{BFO18} (see the paragraph below Lemma 4 in \cite{BFO18} and the PhD thesis of the author \cite{Ngu17} (e.g., Section III.2.3 for an outline).
%Furthermore, by a private communication with 
%Eugenia Malinnikova we know that
%it can be used to prove a version of three balls theorem for other domains than balls.  
\begin{theorem}\label{d1}
For every $d\in \N$, $x\in \R^d$ denote by $\|x\|$ the Euclidean norm of $x$.
Denote by $\int d\sigma$ the usual surface integral.
For every $d\in\N$, $r\in (0,\infty)$ let $\ballO{d}{r},\ballC{d}{r}, \ballB{d}{r}$
be the sets given by
\begin{align}
\ballO{d}{r}=\{x\in\R^d\colon \|x\|_{\R^d}<r \},\quad
\ballC{d}{r}=\{x\in\R^d\colon \|x\|_{\R^d}\leq r \},\quad
\ballB{d}{r} =\{x\in\R^d\colon \|x\|_{\R^d}=r \}. 
\end{align}
For every 
$d\in \N\cap[2,\infty)$, 
$r\in (0,\infty)$, $u\in C^\infty(\ballB{d}{r})$ let
$\partial u\colon\ballB{d}{r}\to\R ^d $ be the gradient of $u$, 
let $\partial_\tau u \colon\ballB{d}{r}\to\R ^d$ be the tangential component of 
$\partial u$, and
let
$\partial_\nu u\colon\ballB{d}{r}\to\R ^d $ be the normal component of $\partial u$, i.e., it holds for all $x\in\ballB{d}{r}$  that $(\partial_\tau u)(x)$ is the orthogonal projection of $\partial u (x)$ onto the tangential space of $\ballB{d}{r}$ at $x$ and $(\partial u)(x)=(\partial_\tau u)(x)+(\partial_\nu u)(x)$. 
Then there exists a function
$C\colon ([2,\infty)\cap\N)\times(1,\infty)\to(0,\infty)$ such that
for all $d\in \mathbbm{N}\cap [2,\infty)$, $p\in(1,\infty)$, 
$r\in(0,\infty)$, $u\colon\ballC{d}{r} \to\mathbbm{R}$
with 
$
u\upharpoonright_{\ballB{d}{r}}\in C^\infty(\ballB{d}{r})$ and $
\forall\, x\in \ballB{d}{r} \colon  \sum_{i=1}^d(\partial_{ii}^2u)(x)=0
$
it holds that
\begin{align}
\frac{1}{C(d,p)}\int_{\ballB{d}{r}} \|\partial_\tau u\|^p\,d\sigma\leq 
\int_{\ballB{d}{r}} \|\partial_\nu u\|^p\,d\sigma\leq C(d,p)\int_{\ballB{d}{r}} \|\partial_\tau u\|^p\,d\sigma.\label{d1b}
\end{align}
\end{theorem}
\begin{figure}\center
\includegraphics[]{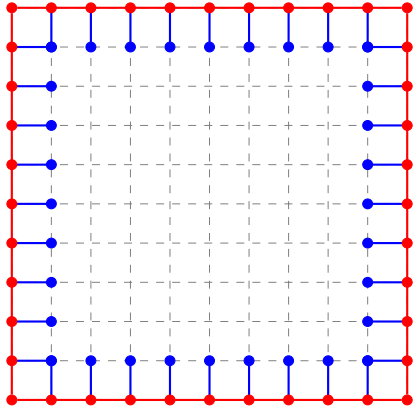}
\caption{Tangential and normal edges of a two dimensional box}\label{f01}
\end{figure}
In order to formulate the discrete analogue of \cref{d1} let us introduce our notation. For the rest of this paper we always use the notation given in \cref{d05} below.
\begin{setting}[Notation for the whole paper]\label{d05}
Let $\K\in\{\R,\C\}$.
For every $d\in \N$ let
$\unit{d}{i}$, $i\in\disint{1}{d}$, be the standard $d$-dimensional basis vectors and let 
$E_d$ be the set of all \emph{oriented nearest neighbour edges} (for short: \emph{edges}) on the $d$-dimensional lattice, i.e., the set given by
$E_d=\{(x,y)\colon (x,y\in \Z^d)\wedge 
(\sum_{i=1}^{d}|x_i-y_i|=1)
\}$.
For every $d\in\N$, $A\subset \Z^d$, 
$u\colon A\to\K$ let 
$
\Laplace u\colon\{x\in\Z^d\colon  \forall\, i\in \disint{1}{d}\colon (x+\unit{d}{i}\in A)\wedge( x-\unit{d}{i}\in A)\}\to\K
$
be the \emph{discrete Laplacian} of $u$, i.e., the function which satisfies  for all $x\in A$ with 
$\forall\, i\in \disint{1}{d}\colon (x+\unit{d}{i}\in A)\wedge( x-\unit{d}{i}\in A)$
that \begin{align}
(\Laplace u)(x)= \left[\sum_{i=1}^{d}u(x+\unit{d}{i})+u(x-\unit{d}{i})\right]-2du(x)
\end{align}
and let
$\nabla u \colon 
(A\times A)\cap E_d
\to\K$ be the function  which satisfies that for all $x,y\in A$ with $(x,y)\in E_d$ it holds that
$\nabla_{(x,y)}u=u(y)-u(x)$. For every $d\in \N\cap[2,\infty)$, $p\in [1,\infty]$,
 every finite set $A$, and every function
 $f\colon A\to\K$ let $\|f\|_{L^p(A)}\in [0,\infty)$ be the real number given by
\begin{align}
\|f\|_{L^p(A)}^p=
\begin{cases}
\sum_{x\in A}|f(x)|^p&\colon p\in [1,\infty)\\
\max_{x\in A}|f(x)|&\colon p=\infty.
\end{cases}\label{x01d}
\end{align}
\end{setting}
Note that in \cref{d05} above the arguments of $\nabla u$ are edges.
We will also introduce another notation for discrete derivatives which are functions of vertices
(see \cref{v03}). 
However, to formulate the main result let us temporarily use the notation in \cref{d02} below.
\begin{setting}\label{d02}
For every $d,N\in \N\cap[2,\infty)$ let
$\etan{d,N},\enor{d,N}\subseteq E_d$ be the sets of edges which satisfy that
\begin{align}
&\etan{d,N}=\left\{(x,y)\in E_d\colon x,y\in 
([0,N]^d)\setminus((0,N)^d)\right\},\\[3pt]
&\enor{d,N}=\left\{(x,y)\in E_d\colon x\in([0,N]^d)\setminus((0,N)^d)\quad\text{and}\quad y\in [1,N-1]^{d} \right\},
\end{align}
let 
$\etanT{d,N}\subseteq \Z^d$ be the set of vertices   given by
$\etanT{d,N}=\Z^d \cap ([0,N]^{d})\setminus ((0,N)^d)$,
let $\mathcal{D}_{d,N} $ be the set of 
 functions $v\colon \etanT{d,N}\to\R$,
let $\mathcal{N}_{d,N} $ be the set of  functions
$v\colon\enor{d,N}\to\R $ with $\sum_{e\in \enor{d,N}}v(e)=0$,
and
let $\mathbbm{Q}_{d,N}$ be the set of  functions
$u\colon (\disint{0}{N})^{d}\to\R$ which satisfy that
$
\forall\, x\in (\disint{1}{N-1})^d\colon   (\Laplace u)(x) =0.
$
\end{setting}
In \cref{d02} above for $d\in[2,\infty)\cap\Z$ we consider boxes on the \emph{$d$-dimensional lattice}
instead of balls on the $d$-dimensional Euclidean space.
\cref{f01} illustrates 
the sets $\etanT{d,N}$, $\etan{d,N}$,
$\enor{d,N}$ for $d=2$, $N=10$: $\etanT{d,N}$ contains all red points, $\etan{d,N}$ contains all red edges, and $\enor{d,N}$ contains all blue edges.
For fixed $d,N\in \N\cap[2,\infty)$ 
the set $\etanT{d,N}$
can be viewed as a \emph{discrete boundary} of the box $(\disint{0}{N})^d$,
a function $v\colon \etanT{d,N}\to\R$ is thus a \emph{discrete Dirichlet condition}, 
an edge $e\in \etan{d,N}$ can be viewed as a tangential vector, 
an edge $e\in \enor{d,N}$ can be viewed as a normal vector, and a function 
$v\colon\enor{d,N}\to\R $ with $\sum_{e\in \enor{d,N}}v(e)=0$ is thus a \emph{discrete Neumann condition}. Here, the vanishing mean is a necessary condition for the Neumann problem to have a solution, which also holds in the continuum setting.  Finally, the set 
$\mathbbm{Q}_{d,N}$ is the set of  functions which is defined on the box 
$(\disint{0}{N})^{d}$
and harmonic in the interior $(\disint{1}{N-1})^d$ of the box.

%Although it is not really necessary to distinguish  the orientations of edges in $\etan{d,N}$, it is important to think of edges in $\enor{d,N}$ as interior normal vectors.

%\cref{d06} below is the main result of this paper.  

\begin{theorem}[Main result]\label{d06}Assume \cref{d02}. Then there exists a function $C\colon ([2,\infty)\cap\N)\times(1,\infty)\to(0,\infty)$,
there exists a unique family of linear operators
\begin{align}
\left\{\Phi^{\mathcal{D}}_{d,N}\colon \mathcal{D}_{d,N}\to \mathbbm{Q}_{d,N}\colon d,N\in [2,\infty)\cap\N \right\} ,
\end{align}
and there exists a family of linear operators
\begin{align}
\left\{\Phi^{\mathcal{N}}_{d,N}\colon \mathcal{N}_{d,N}\to \mathbbm{Q}_{d,N}\colon d,N\in [2,\infty)\cap\N\right\} 
\end{align}
such that 
\begin{enumerate}[i)]
\item \label{d06a}
it holds 
for all $d,N\in [2,\infty)\cap\N$, $f\in \mathcal{D}_{d,N}$, $x\in \etanT{d,N}$ that
$(\Phi^{\mathcal{D}}_{d,N}f)(x) =f(x)$,
\item it holds 
for all $d,N\in [2,\infty)\cap\N$, $p\in (1,\infty)$, $f\in \mathcal{D}_{d,N}$  that 
\begin{align}
\left\|\nabla(\Phi^{\mathcal{D}}_{d,N}f)\right\|_{L^p(\enor{d,N})}
\leq C(d,p)\left\|\nabla f\right\|_{L^p(\etan{d,N})},
\label{a01}
\end{align}
 \item\label{d06b} it holds 
for all $d,N\in [2,\infty)\cap\N$, $e\in \enor{d,N}$, $g\in \mathcal{N}_{d,N}$ that 
$\nabla_e (\Phi^{\mathcal{N}}_{d,N}g)  =g(e)$,
and
\item it holds 
for all $d,N\in [2,\infty)\cap\N$, $p\in (1,\infty)$, $g\in \mathcal{N}_{d,N}$  that 
\begin{align}\label{a03}
\left\|\nabla(\Phi^{\mathcal{N}}_{d,N}g)\right\|_{L^p(\etan{d,N})}
\leq C(d,p)\|g\|_{L^p(\enor{d,N})}.
\end{align}
\end{enumerate}
\end{theorem}
\cref{d06a} in \cref{d06} above
implies that for every $d,N\in [2,\infty)\cap\N$, 
$f\in \mathcal{D}_{d,N}$
the function $\Phi^{\mathcal{D}}_{d,N}f$ is  the solution $u\colon (\disint{0}{N})^{d}\to\R$ to the \emph{discrete Dirichlet problem}
\begin{align}\label{d03}
\begin{cases}
\forall\, x\in (\disint{1}{N-1})^d\colon \quad (\Laplace u)(x)=0,\\
\forall\, x\in \etanT{d,N}\colon\quad  u(x)=f(x)
\end{cases}
\end{align}
and  the family $(\Phi^{\mathcal{D}}_{d,N})_{d,N\in [2,\infty)\cap\N}$ therefore exists uniquely as in the statement of \cref{d06}.
Next,
\cref{d06b} in \cref{d06} above
implies that for every $d,N\in [2,\infty)\cap\N$, 
$g\in \mathcal{N}_{d,N}$
 the function $\Phi^{\mathcal{N}}_{d,N}g$ is a solution
$u\colon (\disint{0}{N})^{d}\to\R$ to \emph{the discrete
Neumann  problem}
\begin{align}\label{d04}
\begin{cases}
\forall\, x\in (\disint{1}{N-1})^d\colon \quad (\Laplace u)(x)=0,\\
\forall\, e\in \enor{d,N}\colon\quad  \nabla_eu=g(e).
\end{cases}
\end{align}
Note that there is no full statement on the uniqueness
of the Neumann problem \eqref{d04}. More precisely,
the uniqueness of the Neumann problem \eqref{d04} only holds up to a constant  on $(\disint{1}{N-1})^d$, i.e., if 
$\forall\, e\in \enor{d,N}\colon g(e)=0$ and if
$u$ is a solution to  \eqref{d04}, then the restriction of $u$ on ${(\disint{1}{N-1})^d}$ is a constant function. 
In addition, note that for fixed $d,N\in [2,\infty)\cap\N$, $g\in \enor{d,N}$
there exists no real numbers $C\in (0,\infty)$ such that for \emph{every} solution
$u$
to \eqref{d04} it holds that
$\|\nabla u\|_{L^p(\etan{d,N})}\leq C\|g\|_{L^p(\enor{d,N})}$. Indeed, 
e.g.,  in the case $d=2$ we can freely change the value of $u$ at the four corners of the rectangle in \cref{f01} to make $\|\nabla u\|_{L^p(\etan{d,N})}$ arbitrary large without damaging the fact that $u$ is a solution to \eqref{d04}.
Consequently,  it is impossible to make any claims on the uniqueness of the family
$(\Phi^{\mathcal{N}}_{d,N}) _{d,N\in [2,\infty)\cap\N}$
in the statement of \cref{d06}.

Next, let us give a brief and rough explanation why \cref{d06} is useful for  the idea of using \emph{harmonic extensions} in the proof of the Liouville theorem in \cite{Ngu17}. Let $u$ be a function defined on the box in \cref{f01}. We keep the Dirichlet condition of $u$ at red points and replace the values of $u$ at other points by an extension that is harmonic in the interior of the box. This will clearly erase the Neumann condition of $u$. However, \cref{d06} claims that the new Neumann condition can still be bounded by the remaining Dirichlet condition.

Discrete Laplacian and discrete harmonic functions are interesting topics that date back to 1920s (see, e.g., the fundamental works by
Lewy, Friedrichs, and Courant~\cite{LFC28},
Heilbronn~\cite{Hei49}, Duffin~\cite{Duf53}). Discrete boundary problems have been widely studied in numerical analysis, e.g., to approximate the continuum solutions (see, e.g., the classical work by Stummel~\cite{Stu67} and for further references see, e.g.,
G\"urlebeck and Hommel \cite{Hom98}, \cite{GH01}, \cite{GH02},
who studied Dirichlet and Neumann boundary problems 
on general two-dimensional discretized domains
using difference potentials,
and the references therein). 

Although discrete and continuum objects often have many similar properties, it is not always trivial to adapt things from the continuum case to the discrete case and vice verse. To the best of the author's knowledge, there exists no result in the discrete case which deals with the bounds \eqref{a01} and \eqref{a03}, while $L^p$-comparisons, $p\in (1,\infty)$, between the tangential and non-tangential components 
of harmonic functions on Lipschitz and $C^1$-domains and related topics have been studied by several papers, e.g., in chronological order:
Mikhlin~\cite{Mik65}, Maergoiz~\cite{Mae73},
Calderon, Calderon, Fabes, Jodeit, and Rivi\`erie~\cite{CCF+78}, 
Fabes, Jodeit, and Rivi\`erie~\cite{FEBJR78}, 
Jerison and Kenig~\cite{JK81}, Verchota~\cite{Ver84}, Dahlberg and Kenig~\cite{DK87}, Mitrea and Mitrea~\cite{MM13}. 
The main issue in the discrete case is to show that the functions $C$ in \eqref{a01} and \eqref{a03} do not depend on the size $N$ of the discrete box while in the continuum case this is not an issue due to a simple scaling argument. In fact, for \eqref{d1b} we only need to consider $r=1$.

The proof of \cref{d06} that we represent here essentially mimics the proof of Lemma~4 in Bella, Fehrman, and Otto \cite{BFO18} who formulate and prove 
\cref{d1} with balls replaced by boxes
in the continuum case. We separate the proof into several steps and organize the paper as follows. 
\cref{se01} formulates and proves a discrete counterpart of inequality (88) in~\cite{BFO18}, which was shown by using the continuum Poison kernels. In order to adapt this idea to the discrete case we use a result in Lawler and Limic~\cite{LL10} to \emph{approximate the discrete Poison kernels by the continuum Poison kernels}.
Estimates by means of the \emph{Marcinkiewicz multiplier theorem}, 
e.g., inequalities (78), (79), (82), and (99) 
in \cite{BFO18}
are adapted in \cref{v03} which focuses on \emph{discrete harmonic functions on haft spaces with periodic boundary conditions}.
In order to avoid many tedious calculations with higher derivatives of the multipliers we apply Cauchy's integral formula. In addition, with some elementary arguments, \cref{d133c} provides a result of independent interest that the author has not
found in the literature. Finally,
\cref{se02} 
applies the results obtained in \cref{se01,v03} to prove the main result, \cref{d06}. 
As Bella, Fehrman, and Otto \cite{BFO18} we call estimate~\eqref{a01} 
\emph{the Dirichlet case} and estimate~\eqref{a03} \emph{the Neumann case} and prove them separately.
The main techniques here are basically to adapt two ideas
learnt from~\cite{BFO18} to the discrete case:
i) returning to the case of periodic boundary conditions by using \emph{even and odd reflections} and ii) reducing to the case of haft spaces. 
Concerning the idea of using reflections, Section IV.2.1 in the author's dissertation \cite{Ngu17} may provide a simple illustration with figures in the two-dimensional case that may help to understand the general case.
Another  interesting application of even and odd reflections  and
the discrete Marcinkiewicz multiplier theorem is
to prove $L^p$-estimates for discrete Poisson equations (see Section~2.5.2 in Jovanovi{\'c} and S{\"u}li~\cite{JS14}).

For convenience, throughout this paper, 
the arguments here are often compared with that in the continuum case in \cite{BFO18}. However, since there are several differences between the discrete case and the continuum case, this paper is organized so that the reader can easily start from scratch. 

Finally, the proof shows that the functions $C$ in \cref{d06} may depend exponentially on the dimension: this result, as finite difference method in general, may not be quite useful for high-dimensional applications (the so-called \emph{curse of dimensionality}).

\bigskip
Our notation will be defined clearly in the formulation of each result. In addition, remember that throughout this paper we always use the notation in \cref{d05} above and the usual conventions in \cref{d05b} below.
\begin{setting}[Conventions]\label{d05b}
Denote by $\ima$ the imaginary unit. 
Denote by $\Re(z)$ and $\Im(z)$ the real and imaginary part of $z\in\K$, respectively, where $\K\in \{\R,\C\}$.
Write $\N=\{1,2,\ldots\}$ and $\N_0=\N\cup\{0\}$. 
For  $d\in\N$, $x,y\in\K^d$, 
$i\in\disint{1}{d}$ denote by $x_i$ the $i$-th coordinate of $x$ (if no confusion can arise),
denote by $x\cdot y$ the standard scalar product of $x$ and $y$,
i.e., $ x\cdot y= \sum_{i=1}^{d}x_i\bar{y}_i$, and
denote by 
$|x|_\infty$ the maximum norm of $x$, i.e.,
$|x|_\infty=\max_{i=1}^d|x_i|$. For every set $A$ denote by $|A|$ the  cardinality of~$A$. Partial derivatives will be denoted by $\partial_i$, $\tfrac{\partial}{\partial t_i}$, $\tfrac{\partial}{\partial \xi_i}$.
When applying a result we often use a phrase like
'\cref{z15} with $d\defeq d-1$' that should be read as
'\cref{z15} applied with $d$ (in the notation of \cref{z15}) replaced by $d-1$ (in the current notation)' and we often omit a trivial replacement to lighten the notation, e.g., we rarely write, e.g., '\cref{s26} with $d\defeq d$'.

\end{setting}
%We use the following conventions: $\sup\emptyset=-\infty$ and $\inf\emptyset=\infty$, and the empty sum is zero.
\subsection*{Acknowledgement}This paper is based on a part of the author's dissertation \cite{Ngu17} written  under supervision of Jean-Dominique Deuschel at Technische Universit\"at Berlin. The author thanks Benjamin Fehrman and Felix Otto for useful discussions and for sending him the manuscript of \cite{BFO18}. The author gratefully acknowledges financial support of the DFG Research
Training Group (RTG 1845) ”Stochastic Analysis with Applications in Biology, Finance
and Physics” and the Berlin Mathematical School (BMS).
\section{Potential-theoretic results for harmonic functions on  haft spaces}\label{se01}
\subsection{Main result}
In this section we essentially prove
\cref{p01} below, which formulates a discrete analogue of 
inequality (88) in Bella, Fehrman, and Otto~\cite{BFO18}. We basically follow the proof in \cite{BFO18}. However, to make the argument more illustrative we introduce a  simple random walk in \cref{x01c}.
 \cref{x14,x15} are discrete counterparts of inequality (92) and (93) in \cite{BFO18}.
Combining \cref{x14,x15} with a Marcinkiewicz-type interpolation argument we obtain
\cref{s32}.
Approximating the discrete Poisson kernels by the continuum counterparts we obtain 
\cref{z02}. This and \cref{s32} imply \cref{p01}.
\begin{setting}\label{p01b}
For every $L\in\N $, $d\in [2,\infty)\cap \N$
let $\I_L$  be the set given by $\I_L=\disint{-L+1}{L} $ and let
$\mathbbm{H}_{d,L,\geq 0}$ be the set of all bounded functions
$u\colon \Z^{d-1}\times \N_0\to \mathbbm{R}$ with the properties that
\begin{enumerate}[i)]
\item 
it holds
for all $x\in \Z^{d-1}\times \N_0$, $  i\in \disint{1}{d-1}$ that
$u(x)=u(x+2L\unit{d}{i})$ and 
\item 
it holds for all $x\in \Z^{d-1}\times \N$
that $(\Laplace u)(x)=0$.
\end{enumerate}
\end{setting}
\begin{corollary}\label{p01}Assume \cref{p01b}.
Then there exists 
$C\colon ([2,\infty)\cap \N)\times (1,\infty)\times(0,\infty)\to (0,\infty)$ such that 
for all $d\in [2,\infty)\cap \N$, $p\in (1,\infty)$, 
$\rup\in (0,\infty)$, $L\in\N$, $N\in (0,L\rup]\cap\N$,
$u\in \mathbbm{H}_{d,L,\geq 0}$ it holds 
that
$
\|u\|_{L^p(\{0\}\times \I_L^{d-2}\times (\disint{1}{N}))}\leq  C(d,p,\rup)
\|u\|_{L^p(\I_L^{d-1}\times\{0\})}.
$
\end{corollary}
\subsection{Results which directly follow from the simple random walk representation}
Throughout this section we use the notation given in \cref{x01c} below. 
Due to the Riesz-Thorin interpolation argument 
for \cref{x13} we have to consider 
the function $u$ in \cref{x01c} as a complex-valued function. For other results we only need to replace $\K$ by $\R$.
\begin{setting}[Simple random walks]\label{x01c}
Let 
$d\in [2,\infty)\cap \N$ be fixed,
let $(\Omega,\mathcal{A},\P)$ be a probability space with expectation denoted by $\mathbbm{E}$,
let $X_n\colon \Omega\to \Z^d$, $n\in \N$, be independent random variables which satisfy for all 
$n\in\N$, $i \in\disint{1}{d}$
that  
$
  \P(X_n= \unit{d}{i})=\P(X_n= -\unit{d}{i})=\tfrac{1}{2d},$
and let $S_n\colon \Omega\to\Z^d$, $n\in\N_0$, 
$T\colon \Omega\to \N_0$ be the random variables which
satisfy for all $n\in\N$ that 
\begin{align}\label{y33}
  S_n = \sum_{j=1}^{n}X_j
\quad\text{and}\quad
T=\inf\left\{j\in\N_0:S_j\in \mathbbm{Z}^{d-1}\times \{0\} \right\}.
\end{align}
\end{setting}
\begin{lemma}\label{z01}Assume \cref{p01b,x01c} and let $L\in\N$, $u\in\mathbbm{H}_{d,L,\geq 0} $.
Then 
\begin{enumerate}[i)]
\item \label{z01a}
it holds for all $(x,y)\in \Z^{d-1}\times\N_0$ that
$u(x,y)
=\mathbbm{E}\!\left[u\left(S_T+(x,0)\right)\middle| S_0= (0,y)\right]$ and
\item \label{z01b}it holds for all $y\in \N_0$ that
$\sum_{x\in \I_L^{d-1}}u(x,y)=\sum_{x\in \I_L^{d-1}}u(x,0)$.
\end{enumerate}
\end{lemma}
The following proof relies on martingale theory. For an elementary proof see \cref{c01}.
\begin{proof}[Proof of \cref{z01}]
The assumption that
$\forall\, (x,y)\in \Z^{d-1}\times \N\colon (\Laplace u)(x,y)=0$ 
and the assumption that $u$ is bounded
demonstrate for all $ x\in \Z^{d-1}$  that $(u(S_n+(x,0)))_{n\in\N_0}$ is a bounded martingale. The optional stopping theorem 
proves that
\begin{align}\label{z01c}
u(x,y)=
\mathbbm{E}\!\left[u\left(S_0+(x,0)\right)\middle| S_0= (0,y)\right]=
\mathbbm{E}\!\left[u\left(S_T+(x,0)\right)\middle| S_0= (0,y)\right].
\end{align}
This shows \cref{z01a}. Furthermore, \eqref{z01c}, linearity, and periodicity imply
for all $y\in \N_0$ that 
\begin{align}
\sum_{x\in \I_L^{d-1}}u(x,y)
=\mathbbm{E}\!\!\left[\sum_{x\in \I_L^{d-1}}u\left(S_T+(x,0)\right)\middle| S_0= (0,y)\right]
=
\sum_{x\in \I_L^{d-1}}u(x,0).
\end{align}
The proof of \cref{z01} is thus completed.
\end{proof}
\begin{lemma}\label{x02}Assume \cref{p01b} and let $d\in [2,\infty)\cap \N$,
$L\in\N$, $u\in\mathbbm{H}_{d,L,\geq 0} $. Then it holds that
\begin{align}
\left\|u\right\|_{L^1(\I_L^{d-1}\times\{N\})}\leq \left\|u\right\|_{L^1(\I_L^{d-1}\times\{0\})}
\quad\text{and}\quad
\left\|u\right\|_{L^\infty(\I_L^{d-1}\times\{N\})}\leq \left\|u\right\|_{L^\infty(\I_L^{d-1}\times\{0\})}.
\end{align}
\end{lemma}
\begin{proof}[Proof of \cref{x02}]
Recall that we use the notation in \cref{x01c}.
The fact that 
$\P$-almost surely it holds that
$S_T\in\Z^ {d-1}\times\{0\}  $ and the assumption on periodicity, i.e., 
$\forall\, (x,y)\in \Z^{d-1}\times \N_0,\,  i\in \disint{1}{d-1}\colon u(x,y)=u(x+2L\unit{d-1}{i},y)$ imply that 
$\P$-almost surely it holds that
\begin{align}
\sum_{x\in\I_L^{d-1}} |u(S_T +(x,0))| = 
\sum_{x\in\I_L^{d-1}} |u(x,0)|
\quad\text{and}\quad
\max_{x\in\I_L^{d-1}} |u(S_T +(x,0))| = 
\max_{x\in\I_L^{d-1}} |u(x,0)|.\label{y35}
\end{align}
Furthermore,  \cref{z01}  shows  for all 
$(x,y)\in \Z^d\times \N_0$ that 
\begin{align}\label{y34}
\left|u(x,y)\right|
=\Big|\mathbbm{E}\left[u\left(S_T+(x,0)\right)\middle| S_0= (0,y)\right]\!\Big|
\leq \mathbbm{E}\left[\Big.\!\left|u\left(S_T+(x,0)\right)\right|\middle| S_0= (0,y)\right].
\end{align} 
This and \eqref{y35} establish that
\begin{align}\begin{split}   
\sum_{x\in\I_L^{d-1}} |u(x,N)|&\leq \sum_{x\in\I_L^{d-1}}
\mathbbm{E}\!\left[\Big. |u\left(S_T+(x,0)\right)| \middle| S_0= (0,N)\right]\\
&=
\mathbbm{E}\!\left[\sum_{x\in\I_L^{d-1}} |u\left(S_T+(x,0)\right)| \middle| S_0= (0,N)\right]
= \sum_{x\in\I_L^{d-1}} |u(x,0)|
\end{split}
\end{align}
and
\begin{align}\begin{split}   
\max_{x\in\I_L^{d-1}} |u(x,N)|&\leq \max_{x\in\I_L^{d-1}}
\mathbbm{E}\!\left[\Big. |u\left(S_T+(x,0)\right)| \middle| S_0= (0,N)\right]\\
&\leq 
\mathbbm{E}\!\left[\max_{x\in\I_L^{d-1}} |u\left(S_T+(x,0)\right)| \middle| S_0= (0,N)\right]
= \max_{x\in\I_L^{d-1}} |u(x,0)|
\end{split}
\end{align}
This  completes the proof of \cref{x02}.
\end{proof}
Combining 
\cref{x02} with
a Riesz-Thorin interpolation  we obtain \cref{x13} below.
\begin{corollary}\label{x13}
Assume \cref{p01b} and let $d\in [2,\infty)\cap \N$, 
$L\in\N$,
$u\in\mathbbm{H}_{d,L,\geq 0} $, $p\in [1,\infty]$, $N\in\N_0$.
Then it holds that
$
\left\|u\right\|_{L^p(\I_L^{d-1}\times\{N\})}\leq
\left\|u\right\|_{L^p(\I_L^{d-1}\times\{0\})}.
$
\end{corollary}
\begin{lemma}[$L^\infty$-estimate]\label{x14} Assume \cref{p01b} and let $d\in [2,\infty)\cap \N$, $L\in\N$, $u\in\mathbbm{H}_{d,L,\geq 0} $, $p\in [1,\infty)$.
Then it holds  that
\begin{align}
\max_{z\in\N_0}\left[
\sum_{y\in\mathbbm{I}^{d-2}_L} \left|u(0,y,z)-\frac{1}{|\mathbbm{I}_L|}\sum_{x\in\mathbbm{I}_L}u(x,y,z)\right|^p\right]^{\!\!\nicefrac{1}{p}}
&\leq 2\max_{x\in \mathbbm{I}_L}\left[ \sum_{y\in \mathbbm{I}_N^{d-2}}
|u(x,y,0)|^ p\right]^{\!\!\nicefrac{1}{p}}.
\end{align}
\end{lemma}
\begin{proof}[Proof of \cref{x14}]Recall that we use the notation in \cref{x01c}. First, observe that
\cref{z01},
Jensen's inequality, and linearity of $\E$ show
 that
\begin{align}
\begin{aligned} 
&\sup_{z\in \N_0}\left[\sum_{y\in \mathbbm{I}_L^{d-2}} |u(0,y,z)|^p\right]^{ \!\!\nicefrac{1}{p}}
=\sup_{z\in \N_0}\left[\sum_{y\in \mathbbm{I}_L^{d-2}} \left|\mathbbm{E}\left[\Big.u(S_T+(0,y,0)) \middle|S_0=(0,0,z)\right]\bigg.\!\!\right|^p\right]^{ \!\!\nicefrac{1}{p}}\\
&\leq\sup_{z\in \N_0} \left[
 \mathbbm{E}\left[\sum_{y\in \mathbbm{I}_L^{d-2}}\!\left|\big.u(S_T+(0,y,0))\right|^p \middle|S_0=(0,0,z)\right]\rule{0cm}{1.0cm}\!\right]^{ \!\!\nicefrac{1}{p}}
\leq  \max_{x\in \mathbbm{I}_L} \left[ \sum_{y\in \mathbbm{I}_N^{d-2}}
\left|\big.u(x,y,0)\right|^ p\right]^{ \!\!\nicefrac{1}{p}}.
\end{aligned}\label{d125b}
\end{align}
Next, Jensen's inequality and \cref{x13} ensure  that
\begin{align} 
\begin{aligned}
&\sup_{z\in \N_0}\left[\sum_{y\in\mathbbm{I}^{d-2}_L} \left|\frac{1}{|\mathbbm{I}_L|}\sum_{x\in\mathbbm{I}_L}u(x,y,z)\right|^p\right]^{ \!\!\nicefrac{1}{p}}
\leq\sup_{z\in \N_0}
\left[
\frac{1}{|\mathbbm{I}_L|}\sum_{x\in\mathbbm{I}_L}\sum_{y\in\mathbbm{I}^{d-2}_L} |u(x,y,z)|^p\right]^{ \!\!\nicefrac{1}{p}}
\\
& \leq\left[
\frac{1}{|\mathbbm{I}_L|}
\sum_{x\in\mathbbm{I}_L}\sum_{y\in\mathbbm{I}^{d-2}_L} |u(x,y,0)|^p\right]^{ \!\!\nicefrac{1}{p}}
\leq 
\max_{x\in \mathbbm{I}_N} \left[\sum_{y\in \mathbbm{I}_N^{d-2}}
|u(x,y,0)|^ p\right]^{ \!\!\nicefrac{1}{p}}.
\end{aligned}
\label{d125a}
\end{align}
Combining this, \eqref{d125b}, and the triangle inequality completes the proof of \cref{x14}.
\end{proof}
\begin{lemma}\label{z04}
Assume \cref{p01b} and let $d\in [2,\infty)\cap \N$, $L\in\N$, $u\in\mathbbm{H}_{d,L,\geq 0} $, $\rup \in (0,\infty)$, $p\in [1,\infty)$, $N\in \N$ satisfy that $N/L\leq \overline{r}$. Then it holds that
\begin{align}\begin{split}
\left[\sum_{y\in \mathbbm{I}_L^{d-2}}\sum_{z=1}^N \left|\frac{1}{|\mathbbm{I}_L|}\sum_{x\in\mathbbm{I}_L}u(x,y,z)\right|^ p\right]^{\!\!\nicefrac{1}{p}}
 \leq(\rup/2)^{\nicefrac{1}{p}}\left[\sum_{x\in \mathbbm{I}_L}\sum_{y\in \mathbbm{I}_L^{d-2}}|u(x,y,0)|^p\right]^{\!\!\nicefrac{1}{p}}.\end{split}
\end{align}
\end{lemma}
\begin{proof}[Proof of \cref{z04}]
Jensen's inequality and the assumption
$N/L\leq \overline{r}$ ensure that 
\begin{align}\begin{split}
\sum_{y\in \mathbbm{I}_L^{d-2}}\sum_{z=1}^N \left|\frac{1}{|\mathbbm{I}_L|}\sum_{x\in\mathbbm{I}_L}u(x,y,z)\right|^ p&\leq
\frac{1}{|\mathbbm{I}_L|}\sum_{x\in \mathbbm{I}_L}\sum_{y\in \mathbbm{I}_L^{d-2}}\sum_{z=1}^N |u(x,y,z)|^p\\
& \leq\frac{N}{2L}\sum_{x\in \mathbbm{I}_L}\sum_{y\in \mathbbm{I}_L^{d-2}}|u(x,y,0)|^p
 \leq\frac{\overline{r}}{2}\sum_{x\in \mathbbm{I}_L}\sum_{y\in \mathbbm{I}_L^{d-2}}|u(x,y,0)|^p.\end{split}
\end{align}
This completes the proof of \cref{z04}.
\end{proof}

\subsection{The Poisson kernel revisited}
\begin{setting}\label{x30}
Assume \cref{x01c},
let $P\colon \mathbbm{N}_0\times \mathbbm{Z}\times \mathbbm{Z}^{d-2}\to\mathbbm{R}$ be the function which satisfies for all $(z, x,y)\in \mathbbm{N}_0\times \mathbbm{Z}\times \mathbbm{Z}^{d-2}$ that
$P_z(x,y)= \mathbbm{P} \left[ S_T= (x,y,0)\middle| S_0=(0,0,z) \right]$, and let $\pconst\in [0,\infty]$ be the real extended number given by
\begin{align}
\pconst=\sup_{z\in \N}
\left[z\sum_{({x},{y})\in\mathbbm{Z}\times\mathbbm{Z}^{d-2}} \left|\big. P_z({x},{y})-P_z({x}-1,{y})\right|\right].\label{x30b}
\end{align}
\end{setting}

\begin{lemma}\label{z02}Assume \cref{x30}.
 Then $\pconst<\infty$.
\end{lemma}
\begin{proof}[Proof of \cref{z02}]Throughout this proof let $\omega_d\in (0,\infty)$ be the surface area of the $(d-1)$-dimensional unit sphere and denote by $|\cdot|\colon \cup_{n\in\N}\R^n\to [0,\infty)$ the Euclidean norm.
The definition of $P$
 and Theorem~8.1.2 in Lawler and Limic~\cite{LL10} (applied
with
$z=(x,y)\defeq (-x,-y,z)$ for $(z, x,y)\in \mathbbm{N}_0\times \mathbbm{Z}\times \mathbbm{Z}^{d-2}$ and
combined with
the definition of the Poisson kernel at the beginning of Section~8.1.1 in~\cite{LL10})
shows that there exist  $s_1,s_2 \colon \Z\times \Z^{d-2}\times \N_0\to \R$,  $c_1,c_2\in (0,\infty)$
which satisfy for all $(x,y,z)\in \Z\times \Z^{d-2}\times \N$  that
\begin{align}\begin{split}
P_z(x,y)&=
\P\!\left[\big.S_T=(x,y,0)\middle|S_0=(0,0,z)\right]=
\P\!\left[\big.S_T=(0,0,0)\middle|S_0=(-x,-y,z)\right]\\
&=\frac{2z(1+s_1(x,y,z))}{\omega_d\left(|x|^2+|y|^2+|z|^2\right)^{d/2}}+ s_2(x,y,z),
\end{split}\label{x19}
\end{align}
\begin{align}
|s_1(x,y,z)|\leq \frac{c_1z}{(|x|^2+|y|^2+|z|^2)} \quad\text{and}\quad |s_2(x,y,z)|\leq \frac{c_2}{\left(|x|^2+|y|^2+|z|^2\right)^{(d+1)/2}}.\label{x20}
\end{align}
The triangle inequality then implies for all $(x,y,z)\in \Z\times \Z^{d-2}\times \N$  that
\begin{align}\begin{split}
|P_z(x,y)-P_z(x-1,y)|&\leq \frac{2z}{\omega_d}\left|\frac{1}{\left(|x|^2+|y|^2+|z|^2\right)^{d/2}}-
\frac{1}{\left(|x-1|^2+|y|^2+|z|^2\right)^{d/2}}
\right|\\
&\qquad+\sum_{\xi\in \{x,x-1\}}\left[\frac{2z|s_1(\xi,y,z)|}{\omega_d\left(|\xi|^2+|y|^2+|z|^2\right)^{d/2}}
+|s_2(\xi,y,z)|\right].
\end{split}\label{x21}
\end{align}
Next, \eqref{x20} implies that for all $(\xi,y,z)\in \Z\times \Z^{d-2}\times \N$ it holds that
\begin{align}\label{x22}
\begin{split}
&\frac{2z|s_1(\xi,y,z)|}{\omega_d\left(|\xi|^2+|y|^2+|z|^2\right)^{d/2}}+
|s_2(\xi,y,z)|\\
&\leq \frac{2c_1z^2}{\omega_d\left(|\xi|^2+|y|^2+|z|^2\right)^{\frac{d}{2}+1}}+
\frac{c_2}{\left(|\xi|^2+|y|^2+|z|^2\right)^{(d+1)/2}}
\leq \frac{2c_1+\omega_dc_2}{\omega_d\left(|\xi|^2+|y|^2+|z|^2\right)^{d/2}}.
\end{split}
\end{align}
Furthermore, 
the mean value theorem and
the fact that for all $a\in (0,\infty)$
the function $\R\ni \xi \mapsto 1/(\xi^2+a)\in\R$ never attains local maxima on $\R\setminus \Z$
imply for all $(x,y,z)\in \Z\times \Z^{d-2}\times \N$  that
\begin{align}\begin{split}
&2z\left|\frac{1}{\left(x^2+|y|^2+|z|^2\right)^{d/2}}-
\frac{1}{\left((x-1)^2+|y|^2+|z|^2\right)^{d/2}}
\right|\\
&\leq 
2z\sup_{\xi\in [x-1,x]}\left|\frac{d}{d\xi}\left[\left(\xi ^2+|y|^2+|z|^2\right)^{-d/2}\right]\right|\leq 2z\sup_{\xi\in [x-1,x]}\left|-\tfrac{d}{2}\left(\xi ^2+|y|^2+|z|^2\right)^{-\frac{d}{2}-1}\cdot 2\xi\right|\\
&\leq 
\sup_{\xi\in [x-1,x]}
\frac{2d}{(\xi^2+|y|^2+|z|^2)^{\frac{d}{2}}}\leq 
\sum_{\xi\in \{x-1,x\}}
\frac{2d}{(\xi^2+|y|^2+|z|^2)^{\frac{d}{2}}}.
\end{split}
\end{align}
Combining this, \eqref{x21}, and \eqref{x22}
 we obtain 
\begin{align}\begin{split}
\sum_{({x},{y})\in\mathbbm{Z}\times\mathbbm{Z}^{d-2}}
|P_z(x,y)-P_z(x-1,y)|&\leq \sum_{({x},{y})\in\mathbbm{Z}\times\mathbbm{Z}^{d-2}}
\sum_{\xi\in \{x-1,x\}}
\frac{2d+2c_1+\omega_dc_2}{\omega_d\left(\xi^2+|y|^2+|z|^2\right)^{\frac{d}{2}}}<\infty.
\end{split}
\end{align}
The fact that 
$\sup_{z\in\N}\left(z\sum_{x\in \Z^{d-1}}(|x|^2+z^2)^{-d/2}\right) <\infty$ and \eqref{x30b} then show that 
$M<\infty$ and
complete the proof of \cref{z02}.
\end{proof}
\begin{lemma}[weak $L^1$-estimate]\label{x15}Assume \cref{p01b,x30} and let $L\in\N$, $u\in\mathbbm{H}_{d,L,\geq 0} $, $p\in [1,\infty)$. Then it holds for all $t\in(0,\infty)$ that
\begin{align}
\begin{split}
\left|
\left\{z\in \N\colon 
\sum_{y\in \mathbbm{I}_L^{d-2}} \left|u(0,y,z)-\frac{1}{|\mathbbm{I}_L|}\sum_{x\in\mathbbm{I}_L}u(x,y,z)\right|^p>t^p
\right\}\right|\leq 
\frac{ \pconst}{t}
\sum_{x\in\mathbbm{I}_L} \left[ \sum_{y\in\mathbbm{I}_L^{d-2}} |u(x,y,0)|^p\right]^{\!\!\nicefrac{1}{p}}.
\end{split}
\end{align}
\end{lemma}
\begin{proof}[Proof of \cref{x15}]Throughout the proof  let 
\begin{align}
A\in\R,\quad 
f\colon \N\to\ \R,\quad 
\disDiff{+}{x} u\colon \Z\times \Z^{d-2}\times \N_0\to \K,\quad 
\disDiff{-}{x}P_z\colon  \mathbbm{Z}\times \mathbbm{Z}^{d-2}\to\mathbbm{R}\;(z\in \N_0)
\end{align}
which satisfy for all $x\in \Z$, $y\in\Z^{d-2}$, $z\in \N_0$ that
\begin{align}\begin{split}
A= \sum_{x\in\mathbbm{I}_L} \left[ \sum_{y\in\mathbbm{I}_L^{d-2}} |u(x,y,0)|^p\right]^{\!\!\nicefrac{1}{p}},\quad 
f(z)= \left[\sum_{y\in \mathbbm{I}_L^{d-2}} \left|u(0,y,z)-\frac{1}{|\mathbbm{I}_L|}\sum_{x\in\mathbbm{I}_L}u(x,y,z)\right|^p\right]^{\!\!\nicefrac{1}{p}},\label{x17b}
\end{split}\end{align}
\begin{align}\label{x17}
(\disDiff{+}{x}u)(x,y,z)=u(x+1,y,z)-u(x,y,z)
\end{align}
and 
\begin{align}\label{x18}
(\disDiff{-}{x}P_z)(x,y)=P_z(x-1,y)-P_z(x,y).
\end{align}
The triangle inequality and a telescope sum argument then show for all $y\in \mathbbm{I}_L^{d-2}$, $z\in \N$  that
\begin{align}\begin{split}    
&\left|u(0,y,z)- \frac{1}{|\mathbbm{I}_L|}\sum_{x\in \mathbbm{I}_L} u(x,y,z)\right|\leq \frac{1}{|\mathbbm{I}_L|}\sum_{x\in \mathbbm{I}_L}|u(0,y,z)-u(x,y,z)|\\
&
\leq \frac{1}{|\mathbbm{I}_L|}\sum_{x\in \mathbbm{I}_L}\sum_{x'\in\mathbbm{I}_L}\left|u(x'+1,y,z)-u(x',y,z)\right| 
= \sum_{x\in\mathbbm{I}_L}\left|(\disDiff{+}{x}u)(x,y,z)\right|.\end{split}
\end{align}
This, \eqref{x17b}, and the triangle inequality imply
for all $z\in \N$ that
\begin{align} \begin{split}   
f(z) \leq 
\left[\sum_{y\in\mathbbm{I}^{d-2}_L} \left[\sum_{x\in\mathbbm{I}_L}\left|(\disDiff{+}{x}u)(x,y,z)\right|\right]^p\right]^{\!\!\nicefrac{1}{p}}
&\leq 
\sum_{x\in \mathbbm{I}_L} \left[
\sum_{y\in\mathbbm{I}^{d-2}_L}
\left|(\disDiff{+}{x}u)(x,y,z)\right|^p\right]^{\!\!\nicefrac{1}{p}}\!\!.\end{split}\label{x16b}
\end{align}
Furthermore, \cref{z01} shows for all $(x,y,z)\in\mathbbm{Z}\times\mathbbm{Z}^{d-2}\times (\disint{1}{N}) $ that
\begin{align}\begin{split} 
u(x,y,z)&= \mathbbm{E} \!\left[u(S_T+(x,y,0))\middle| S_0=(0,z)\Big.\right]
\\
&=
\sum_{(\tilde{x},\tilde{y})\in \Z\times \Z^{d-2}}
\P\!\left[ \Big.S_T=(\tilde{x},\tilde{y},0)\middle| S_0=(0,z)\right]u(x+\tilde{x},y+\tilde{y},0)
\\
&= \sum_{(\tilde{x},\tilde{y})\in \Z\times \Z^{d-2}}
P_z(\tilde{x},\tilde{y})
u(x+\tilde{x},y+\tilde{y},0).\end{split}
\label{x16}
\end{align}
The substitution $\Z\ni\tilde{x}\mapsto \tilde{x}-1\in \Z$ then
proves for all $(x,y,z)\in\mathbbm{Z}\times\mathbbm{Z}^{d-2}\times (\disint{1}{N}) $ that
\begin{align}\begin{split} 
u(x+1,y,z)&= \sum_{(\tilde{x},\tilde{y})\in \Z\times \Z^{d-2}}
P_z(\tilde{x},\tilde{y})
u(x+1+\tilde{x},y+\tilde{y},0)\\
&=
\sum_{(\tilde{x},\tilde{y})\in \Z\times \Z^{d-2}}
P_z(\tilde{x}-1,\tilde{y})
u(x+\tilde{x},y+\tilde{y},0).\end{split}
\end{align}
This, \eqref{x16}, \eqref{x17}, \eqref{x18} yield for all $(x,y,z)\in\mathbbm{Z}\times\mathbbm{Z}^{d-2}\times (\disint{1}{N}) $  that
\begin{align}
(\disDiff{+}{x}u)(x,y,z)
 =\sum_{(\tilde{x},\tilde{y})\in \mathbbm{Z}^{d-1}}(\disDiff{-}{x}P_z)(\tilde{x},\tilde{y}) u(\tilde{x}+x,\tilde{y}+y,0).
\end{align}
This,  the triangle inequality, and the fact that
\begin{align}
\forall\, (\tilde{x}, \tilde{y})\in \Z\times  \I_L^{d-2}\colon \quad 
\sum_{y\in\mathbbm{I}_L^{d-2}} \left|u(x,y,0)\right|^p=\sum_{y\in\mathbbm{I}_L^{d-2}} \left|u(x+\tilde{x},y+\tilde{y},0)\right|^p,
\end{align}
which is
a consequence of the periodicity,   \eqref{x18},
\eqref{x17}, and \eqref{x30b}
 imply for all $z\in\N$ that
\begin{align}\begin{split} 
&f(z)\\&\leq 
\sum_{x\in \mathbbm{I}_L} \left(
\sum_{y\in\mathbbm{I}^{d-2}_L}
\left|(\disDiff{-}{x}u)(x,y,z)\right|^p\right)^{\!\!\nicefrac{1}{p}}=
\sum_{x\in\mathbbm{I}_L} \left(\sum_{y\in\mathbbm{I}_L^{d-2}} \left|\sum_{(\tilde{x},\tilde{y})\in\mathbbm{Z}\times \mathbbm{Z}^{d-2}}(\disDiff{-}{x} P)(z,\tilde{x},\tilde{y}) u(x+\tilde{x},y+\tilde{y})\right|^ p\right)^{\!\!\nicefrac{1}{p}}
\\
&\leq
\sum_{x\in\mathbbm{I}_L}
\sum_{\tilde{x}\in\mathbbm{Z}}
\sum_{\tilde{y}\in  \mathbbm{Z}^{d-2}}  \left(\sum_{y\in\mathbbm{I}_L^{d-2}}\big|(\disDiff{-}{x} P)(z,\tilde{x},\tilde{y})u(x+\tilde{x},y+\tilde{y},0)\big|^p\right)^{\!\!\nicefrac{1}{p}}
\\
&=
\sum_{\tilde{x}\in\mathbbm{Z}}\left[
\sum_{\tilde{y}\in  \mathbbm{Z}^{d-2}}\left(\left|(\disDiff{-}{x} P)(z,\tilde{x},\tilde{y})\right|
\sum_{x\in\mathbbm{I}_L}  \left[\sum_{y\in\mathbbm{I}_L^{d-2}} \left|u(x+\tilde{x},y+\tilde{y},0)\right|^p\right]^{\!\!\nicefrac{1}{p}}\right)\right]\\
&=\left[\sum_{(\tilde{x},\tilde{y})\in\mathbbm{Z}\times\mathbbm{Z}^{d-2}} \left|\big.P(z,\tilde{x},\tilde{y})-P(z,\tilde{x}-1,\tilde{y})\right|\right]
\sum_{x\in\mathbbm{I}_L} \left[ \sum_{y\in\mathbbm{I}_L^{d-2}} |u(x,y,0)|^p\right]^{\!\!\nicefrac{1}{p}}\leq \frac{\pconst A}{z}
.
\end{split}
\end{align}
This shows for all $t\in (0,\infty)$ that
\begin{align}
t\left|\left\{z\in\N\colon |f(z)|>t\right\}\right|\leq 
t\left|\left\{z\in\N\colon |\tfrac{\pconst A}{z} |>t\right\}\right| 
=t
\left|\left\{z\in\N\colon |\tfrac{\pconst A}{t} |>z\right\}\right| 
\leq \pconst A.
\end{align}
This and \eqref{x17b} complete the proof of \cref{x15}.
\end{proof}
\begin{corollary}\label{s32}Assume \cref{p01b,x30}, let $L\in\N$, $u\in\mathbbm{H}_{d,L,\geq 0} $,
 $\rup \in (0,\infty)$, $p\in (1,\infty)$, $N\in \N$, and assume that $N/L\leq \overline{r}$.
Then it holds that
\begin{align}\begin{split}
\left[
\sum_{z\in \N}^{}
\sum_{y\in \mathbbm{I}_L^{d-2}} \left|u(0,y,z)-\frac{1}{|\mathbbm{I}_L|}\sum_{x\in\mathbbm{I}_L}u(x,y,z)\right|^p\right]^{\!\!\nicefrac{1}{p}}
\leq 2\left(\tfrac{Mp}{p-1}\right)^{\!\!\nicefrac{1}{p}}
2^{1-\frac{1}{p}}
\left[ 
\sum_{x\in\mathbbm{I}_L} \sum_{y\in\mathbbm{I}_L^{d-2}} |u(x,y,0)|^p\right]^{\!\!\nicefrac{1}{p}}
\end{split}\label{s33}
\end{align}
and 
\begin{align}\begin{split}
\left[
\sum_{z\in \N}^{}
\sum_{y\in \mathbbm{I}_L^{d-2}} \left|u(0,y,z)\right|^p\right]^{\!\!\nicefrac{1}{p}}
\leq \left(2\left(\tfrac{Mp}{p-1}\right)^{\!\!\nicefrac{1}{p}}
2^{1-\frac{1}{p}}+(\rup/2)^{\nicefrac{1}{p}}\right)
\left[ 
\sum_{x\in\mathbbm{I}_L} \sum_{y\in\mathbbm{I}_L^{d-2}} |u(x,y,0)|^p\right]^{\!\!\nicefrac{1}{p}}.
\end{split}\label{s34}\end{align}
\end{corollary}
\begin{proof}[Proof of \cref{s32}]
An interpolation argument (see, e.g., \cref{r21b})) and \cref{x14,x15} imply \eqref{s33}. Next, the triangle inequality, 
\eqref{s33}, and \cref{z04} prove \eqref{s34}. This completes the proof of \cref{s32}.
\end{proof}
\section{Fourier analysis for harmonic functions on the haft space}\label{v03}
\subsection{Main result}

In this section we continue considering harmonic functions on the discrete haft space with periodic boundary conditions,
however, from the viewpoint of Fourier analysis.
The main results are summarized in
\cref{t01} below, whose main part is illustrated by \cref{d16}. 
As Bella, Fehrman, and Otto~\cite{BFO18} we call
the first inequality in \eqref{t01b} the Dirichlet case and
the second inequality in \eqref{t01b} the Neumann case.
In order to show \cref{t01} we combine \cref{x12e} and, in particular, \cref{s28} (the Neumann case) and \cref{s29} (the Dirichlet case).

As in the proof in the continuum case~\cite{BFO18} our proof is 
based on 
Marcinkiewicz-type
 multiplier theorems 
and
the observation that the tangential derivatives and the normal derivatives of harmonic functions on the  haft space are related by mean of Fourier multipliers.
After having finished his dissertation~\cite{Ngu17}, 
the author realized that for the argument with telescope sequences (see the paragraph below inequality (88) in \cite{BFO18})
it suffices to consider haft spaces 
instead of   strips. The calculations here are therefore much simpler than that in \cite{Ngu17}.  However, we still have to overcome some tedious calculations with the discreteness when estimating the higher derivatives of the multipliers. Another issue is to adapt carefully the paragraph between (83) and (84) in \cite{BFO18} into the discrete case for which we have to work with the dyadic sets, see \cref{d133c}.

\begin{figure}\center
\includegraphics[]{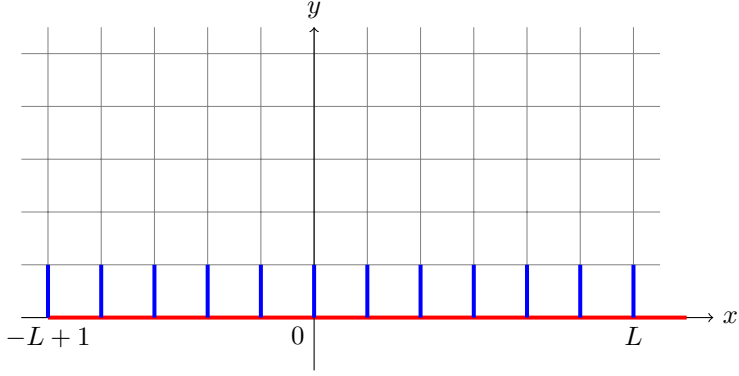}\caption{Estimate \eqref{t01b} in  \cref{t01} bounds by means of $L^p$-norms, $p\in(1,\infty)$, the derivatives with respect to the blue edges by that with respect to the red edges and vice verse.}
\label{d16}
\end{figure}%%%

\begin{corollary}\label{t01}
For every $L\in\N $, $d\in [2,\infty)\cap \N$
let $\I_L$  be the set given by $\I_L=\disint{-L+1}{L} $ and let
$\mathbbm{H}_{d,L,\geq 0}$ be the set of all bounded functions
$u\colon \Z^{d-1}\times \N_0\to \mathbbm{R}$ with the properties that
\begin{enumerate}[a)]
\item 
it holds
for all $x\in \Z^{d-1}\times \N_0$, $  i\in \disint{1}{d-1}$ that
$u(x)=u(x+2L\unit{d}{i})$ and 
\item 
it holds for all $x\in \Z^{d-1}\times \N$
that $(\Laplace u)(x)=0$.
\end{enumerate}
For every $L\in\N $, $d\in [2,\infty)\cap \N$, $u\in \mathbbm{H}_{d,L,\geq 0}$ let
$ \disDiff{+}{i}u\colon 
\Z^{d-1}\times\N_0
\to  \R$, $i\in \disint{1}{d}$,
be the functions which satisfy for all
$i\in \disint{1}{d}$, $x\in \Z^{d-1}\times\N_0$ 
that $
(\disDiff{+}{i}u)(x)= u\!\left(x+\unit{d}{i}\right)-u(x)$.
Then there exist  functions
$C_1\colon ([2,\infty)\cap\N)\times(1,\infty)\times (0,\infty)\to (0,1)$,
$C_2\colon ([2,\infty)\cap\N)\times(1,\infty)\to (0,\infty)$ such that
\begin{enumerate}[i)]
\item \label{t01a}
it holds
for all $d\in [2,\infty)\cap\N$, $p\in (1,\infty)$, $N,L\in\N$,
$u\in \mathbbm{H}_{d,L,\geq 0}$ with $N/L\geq \rdown$ and $\sum_{x\in \I_L^{d-1}}u(x)=0$ that
$\left\|u\right\|_{L^p(\I_L^{d-1}\times\{N\})}\leq C_1(d,p,\rdown)\left\|u\right\|_{L^p(\I_L^{d-1}\times\{0\})}$ and
\item \label{t01c}it holds
for all $d\in [2,\infty)\cap\N$, $p\in (1,\infty)$, $L\in\N$,
$u\in \mathbbm{H}_{d,L,\geq 0}$
that
\begin{align}\label{t01b}
\frac{1}{C_2(d,p)}\left\|\disDiff{+}{d}u\right\|_{L^p(\I_L^{d-1}\times\{0\}) }\leq \left[\sum_{i=1}^{d-1}
\left\|\disDiff{+}{i} u\right\|_{L^p(\I_L^{d-1}\times\{0\}) }\right]\leq C_2(d,p)
\left\|\disDiff{+}{d} u\right\|_{L^p(\I_L^{d-1}\times\{0\}) }.
\end{align}
\end{enumerate}
\end{corollary}
\subsection{Notation and settings}
%%%%%
Instead of $\Z^d$, $d\in \N$,
we will work with $h\Z^d$, $h\in\pi/\N$, $d\in \N$. 
In fact, our notation in \cref{d134} is  inspired by Jovanovi{\'c} and S{\"u}li~\cite[Section 2.5]{JS14} so that we can easily use the Marcinkiewicz multiplier therein. To make the notation consistent we introduce \cref{x01}. 
%For the rest of \cref{v03} we always use the notation given in \cref{d134,d133a,x01}.
\begin{setting}[Periodic functions and $L^p_h$-norms]\label{d134}
For every $N\in \N$ 
let $\I_N$ be the set given by $\I_N=\disint{-N+1}{N} $.
For every $h\in \pi/\N$
let $\omega_h$ be the set given by 
\begin{align}
\omega_h=h\I_{\pi /h} =h\left( \disint{-(\tfrac{\pi}{h}-1)}{\tfrac{\pi}{h}}\right)
=\left\{-h(\tfrac{\pi}{h}-1),-h(\tfrac{\pi}{h}-2),\ldots,h(\tfrac{\pi}{h}-1),h\tfrac{\pi}{h}\right\}.
\end{align}
For every $N\in \N$, $d\in\N$
let $\perfct{2N}{\Z^d} $ be the set of all \emph{$(2N)$-periodic  functions}  on $ \Z^d$, i.e., 
\begin{align}
\perfct{2N}{\Z^d}=\left\{a\colon \Z^d\to\C\colon \forall\, k\in \Z^d,\, i\in \disint{1}{d}\colon a(k)=a(k+2N \unit{d}{i})\right\}.
\end{align}
For every $d\in\N$, $h\in \pi/\N$
 let 
$h\Z^d$ be the lattice given by $h\Z^d=\{hx\colon x\in\Z^d\}$ and
let $\perfct{2\pi}{h\Z^d}$ be the set of all \emph{$2\pi$-periodic  functions} defined on  $h\Z^d$, i.e.,
\begin{align}
\perfct{2\pi}{h\Z^d}=\left\{v\colon h\Z^d\to\C\colon
\forall\, x\in h\Z^d,\, i\in \disint{1}{d}\colon v(x)=v(x+2\pi \unit{d}{i})\right \}.
\end{align}
For every $d\in \N$, $p\in [1,\infty)$, $h\in \pi/\N$, $f\in\perfct{2\pi}{h\Z^d}  $ let $\|f\|_{L^p_h(\omega_h^d)}\in\R$
be the real number which satisfies that
\begin{align}
\|f\|_{L^p_h(\omega_h^d)}=\left[h^{d}\sum_{x\in\omega_h^d}|f(x)|^p\right]^{\nicefrac{1}{p}},\label{x02c}
\end{align}
which is distinguished from $\|\cdot\|_{L^p(A)}$ in \cref{x01d}
by a normalized factor $h^d$.
Denote by $\mathcal{F}$ the so-called \emph{discrete Fourier transform}, i.e.,
\begin{align}
\mathcal{F}\colon\left[ \bigcup_{d\in\N,h\in \pi/\N}\perfct{2\pi}{h\Z^d}\right]\to\left[\bigcup_{d,N\in\N} \perfct{2N}{\Z^d} \right]
\end{align}
is the operator which satisfies 
for every $d\in\N$, $h\in \pi/\N$, $v\in \perfct{2\pi}{h\Z^d} $, $k\in \Z^d$ that
\begin{align}\label{y02b}
\fourier{v}\in \perfct{2\pi /h}{\Z^d} \quad\text{and}\quad
(\mathcal{F}(v))(k)=  h^{d}\sum_{x\in \omega_h^{d}} v(x)e^{-\ima k\cdot x}.
\end{align}
\end{setting}
%\begin{setting}[Discrete Fourier transforms]\label{d134}
%Let \cref{d134} be given.

%Denote by $\mathcal{F}^{-1}$ the discrete inverse Fourier transform, i.e.,  
%\begin{align}
%\mathcal{F}^{-1}\colon\left[\bigcup_{d,N\in\N} \perfct{2N}{\Z^d}\right] \to\left[
% \bigcup_{d\in\N,h\in \pi/\N}\perfct{2\pi}{h\Z^d}\right]
%\end{align}
%is the operator which satisfies for every $d, N\in\N$, $a\in \perfct{2N}{\Z^d}$, $x\in (\pi/N)\Z^d$ that
%\begin{align}
%\inversefourier{a}\in\perfct{2\pi}{(\pi/N)\Z^d}\quad\text{and}\quad
%(\inversefourier{a} )(x)= (2 \pi)^{-d} \sum_{k\in \I_{2N}^{d}} a(k)e^{\ima k\cdot x}.\label{y02}
%\end{align}
%\end{setting}

\begin{setting}%[Harmonic functions on the discrete haft space $h\Z^d$]
[Discrete Laplacian, finite differences, and harmonic functions]
\label{x01}
For every $ d\in [2,\infty)\cap\N$, $h\in\pi/\N$, 
$u\colon (h\Z^{d-1})\times (h\N_0)\to\K$ 
%denote by $u(\cdot,y)$, $y\in h\N_0$, and
%$u(x,\cdot )$, $x\in h\Z^{d-1}$,
%the functions given by
%\begin{align}\begin{split}
%\forall\, y\in h\N_0\colon\quad 
%u(\cdot,y)=\left(h \Z^{d-1}\ni x\mapsto u(x,y)\in \C \right),
%\\
%\forall\, x\in h\Z^{d-1}\colon \quad u(x,\cdot)=( \N_0\ni y\mapsto u(x,y)\in\C),\end{split}
%\end{align}
let $\Laplace ^h\colon  (h\Z^{d-1})\times(h\N)\to \K$ (recall: $\K\in\{\R,\C\}$)
be the \emph{discrete Laplacian} with mesh $h$, i.e., the function which satisfies for all $(x,y)\in (h\Z^{d-1})\times(h\N) $ that
\begin{align}\begin{split}
&(\Laplace ^h u)(x,y)\\&=\left[\sum_{i=1}^{d-1} (u(x+ h\unit{d-1}{i},y)+u(x- h\unit{d-1}{i},y)\right] +u(x,y+h)+u(x,y-h)-2du(x,y),
\end{split}\label{x01b}\end{align}
let $ \disDiff{h}{y}u,\disDiff{h}{x,i}u\colon (h\Z^{d-1})\times(h\N_0)\to \K$, $i\in \disint{1}{d-1}$,
be the functions which satisfy for all $(x,y)\in (h\Z^{d-1})\times(h\N_0) $,
$i\in \disint{1}{d-1}$ that
\begin{align}
(\disDiff{h}{x,i}u)(x,y)= \frac{u(x+h\unit{d-1}{i},y)-u(x,y)}{h}\quad\text{and}\quad
(\disDiff{h}{y}u)(x,y)= \frac{u\!\left(x,y+1\right)-u(x,y)}{h},\label{z13}
\end{align}
and we write $\disDiff{h}{x}u=(\disDiff{h}{x,1}u,\ldots,
\disDiff{h}{x,d-1}u)\colon (h\Z^{d-1})\times(h\N_0)\to \K^{d-1}$.
For every $h\in\pi/\N $, $d\in [2,\infty)\cap\N$ let
$\mathbbm{H}_{d,h,\geq 0}$ be the set of all bounded functions
$u:(h\Z^{d-1})\times (h\N_0)\to \K$ which satisfy that
\begin{enumerate}[i)]
\item 
it holds  
for all $(x,y)\in (h\Z^{d-1})\times (h\N_0)$, $  i\in \disint{1}{d-1}$ that
$u(x,y)=u(x+2\pi\unit{d-1}{i},y)$ and
\item 
it holds for all $(x,y)\in (h\Z^{d-1})\times (h\N)$ that
$(\Laplace^h u)(x,y)=0$.
\end{enumerate}
\end{setting}
In order to obtain \cref{x12e} using a Riesz-Thorin interpolation argument we choose $\K=\C$ in \cref{x01} above. For other results we only need $\K=\R$.

\subsection{Some simple calculations}\label{x01e}
The main results of this subsection,
\cref{z17,z16}, prove that the discrete normal and tangential derivatives are related by means of Fourier multipliers. We start with 
\cref{y15} below that defines the functions which are used to represent the Fourier transform of  harmonic functions and their discrete derivatives. It is useful to consider $Q$ and $f$ in \eqref{y13} as functions of a complex variable. 
The names $Q$ and $\lambda$ are inspired by Guadie~\cite{Gua03} who considers 
harmonic functions on infinite strips with $L^2(\Z^{d-1})$ boundary conditions. 
\begin{setting}\label{y15}
Let $d\in [2,\infty)\cap\N$ be fixed, let $R\in C( \C\setminus(\infty,0),\C)$ be the complex square root, i.e., the function that is
holomorphic on $\C\setminus( -\infty,0]\to\C$ and satisfies   
for all
$z\in \C\setminus (-\infty,0)$ that $R(z)^2={z}
$ (cf. \cref{y07}),
let
$Q,f\colon \C\setminus( -\infty,1)\to\C$ be the functions which satisfy for all $z\in\C\setminus( -\infty,1)$ that 
\begin{align}
Q(z)=z+R(z+1)R(z-1),\quad Q(z)\neq 0,\quad\text{and}\quad
f(z)= \frac{1}{Q(z)}-1,\label{y13}
\end{align}
let 
$\lambda\colon [-\pi,\pi]^{d-1}\to\mathbbm{R}$,
$\DM_i,\NM_i\colon [-\pi,\pi]^{d-1}\to\C$,
$i\in \disint{1}{d-1}$,
 be the functions which satisfy for all
 $t\in [-\pi,\pi]^{d-1}$, 
$i\in \disint{1}{d-1}$ that
\begin{align}
\begin{split}
 \lambda(t)= d-\sum_{i=1}^{d-1}\cos(t_i),\quad 
\DM_i(t)
=
\begin{cases} \displaystyle
\frac{f(\lambda(t))}{e^ {-\ima t_i}-1}\!\!\! &\colon t_i\neq0\\
0\!\!\! &\colon t_i=0,
\end{cases}\quad\text{and}\quad
\NM_i(t)
=
 \begin{cases}\displaystyle
\frac{e^{-\ima t_i}-1}{f(\lambda(t))}\!\!\!&\colon t\neq 0\\
0\!\!\!&\colon t=0,
\end{cases}
\end{split}
\label{y14}
\end{align} 
let 
$\DM_i^h,\NM_i^h\colon [-\pi/h,\pi/h]^{d-1}\to\C$, $i\in \disint{1}{d-1}$, $h\in (0,\infty)$, be the functions which satisfy for all $i\in \disint{1}{d-1}$, $h\in (0,\infty)$, $\xi\in[-\pi/h,\pi/h]^{d-1} $ that
\begin{align}
\DM_i^h(\xi)= \DM_i(h\xi)\quad\text{and}\quad\NM_i^h(\xi)= \NM_i(h\xi),\label{y14b}
\end{align}
and
let $c\in (0,\infty)$ be the real number (cf. \cref{z05}) which satisfies that 
$ c=\inf_{s\in [-\pi,\pi]\setminus \{0\}}\frac{1-\cos (s)}{s^2}$.
\end{setting}

\begin{lemma}\label{x32}
Assume \cref{y15}. Then it holds for all 
$t\in [-\pi,\pi]^{d-1}$  that
\begin{align}
c|t|_\infty^2\leq 
\sum_{i=1}^{d-1}c|t_i|^2\leq 
|\lambda(t)-1|\leq \frac12\sum_{i=1}^{d-1}|t_i|^2
\leq \frac12(d-1)|t|_\infty^2.\label{x32b}
\end{align}
\end{lemma}
\begin{proof}[Proof of \cref{x32}]
The fact that $\forall\, s\in [-\pi,\pi]\colon  cs^2\leq 1-\cos( s)\leq s^2/2$ and the definition of $\lambda$ in \eqref{y14} complete the proof of \cref{x32}.
\end{proof}
\begin{lemma}
\label{y16}
Assume \cref{y15} and let 
$z\in\C\setminus( -\infty,1) $. Then it holds
that 
\begin{align}
Q(z)+\frac{1}{Q(z)}=2z\quad\text{and}\quad
f(z)^2=\frac{2(z-1)}{Q(z)}.
\end{align}
\end{lemma}
\begin{proof}[Proof of \cref{y16}]
First, \eqref{y13}, the assumption that
$\forall\,\zeta\in \C\setminus (-\infty,0)\colon R(\zeta)^2={\zeta}$, and the assumption that $z\in\C\setminus( -\infty,1) $ prove that
\begin{align}\begin{split}
Q(z)\Big(z-R(z+1)R(z-1)\Big)&= \Big(z+R(z+1)R(z-1)\Big)\Big(z-R(z+1)R(z-1)\Big)\\&= z^2- R(z+1)^2R(z-1)^2=z^2-(z+1)(z-1)=1.
\end{split}\end{align}
Therefore, 
$
1/Q(z)= z-R(z+1)R(z-1).
$ This and \eqref{y13} show that $Q(z)+\frac{1}{Q(z)}= 2z .$
Multiplying with $Q(z)$ yields that
$Q(z)^2-2zQ(z)+1=0.$ This and \cref{y13} show that
\begin{align}\begin{split}
f(z)^2&=\left(\frac{1-Q(z)}{Q(z)}\right)^2=\frac{\big(Q(z)^2-2zQ(z)+1\big)+2Q(z)(z-1)}{Q(z)^2} =
\frac{2Q(z)(z-1)}{Q(z)^2}=
\frac{2(z-1)}{Q(z)}.
\end{split}\end{align}
The proof of \cref{y16} is thus completed.
\end{proof}
\cref{x10} below is a classical result and is included for convenience of the reader.
\begin{lemma}[Plancherel's identity]\label{x10}
Assume \cref{d134}.
Let $d\in\N$, $h\in\pi/\N$. Then \begin{align}
\sum_{k\in \I_{\pi/h}^{d}}\left|(\mathcal{F}(v))(k)\right|^2=
 (2\pi h)^{d}\sum_{x,y\in \omega_h^{d}} |v(x)|^2.
\end{align}
\end{lemma}
\begin{proof}[Proof of \cref{x10}]The fact that
$
\forall\, x,y\in\omega_h^{d}\colon  \sum_{k\in \I_{\pi/h}^{d}}e^{-\ima k\cdot (x-y)}=\delta_{xy}(2\pi/h)^d
$
implies that
\begin{align}\begin{split}    
&\sum_{k\in \I_{\pi/h}^{d}}\!\!\left|(\mathcal{F}(v))(k)\right|^2=
\sum_{k\in \I_{\pi/h}^{d}}\!\!
(\mathcal{F}(v))(k)
\overline{(\mathcal{F}(v))(k)}\\&=  \sum_{k\in \I_{\pi/h}^{d}}\!\!\left(h^{d}\sum_{x\in \omega_h^{d}} v(x)e^{-\ima k\cdot x}\right)\!\!\!
\left(h^{d}\sum_{y\in \omega_h^{d}}\overline{v(y)}e^{\ima k\cdot y}\right)=
h^{2d}\left[\sum_{x,y\in \omega_h^{d}} v(x)\overline{v(y)}\sum_{k\in \I_{\pi/h}^{d}}e^{-\ima k\cdot (x-y)}\right]\\&=h^{2d}\left[\sum_{x,y\in \omega_h^{d}} v(x)\overline{v(y)}(2{\pi/h})^d\delta_{xy}\right]=(2\pi)^d
\left[ h^{d}\sum_{x,y\in \omega_h^{d}} |v(x)|^2\right].\end{split}
\end{align}
This completes the proof of \cref{x10}.
\end{proof}
\cref{z15} is straightforward and its proof is therefore omitted.
\begin{lemma}\label{z15}Assume \cref{d134}. Let $d\in \N$, $h\in\pi/\N$, $f\in \perfct{2\pi}{h\Z^d}$,
$k\in\I_{\pi/h}^{d}$, $i\in \disint{1}{d}$. Then 
\begin{enumerate}[i)]
\item \label{z15a}it holds that
$[\mathcal{F}(f(\cdot+ h\unit{d}{i}))](k)=
[\mathcal{F}(f)](k)e^{-\ima hk_i }$,
\item \label{z15b}it holds that
$[\mathcal{F}(f(\cdot- h\unit{d}{i}))](k)=
[\mathcal{F}(f)](k)e^{+\ima hk_i }$, and
\item \label{z15c}it holds that
$[\mathcal{F}
(f(\cdot+ h\unit{d}{i}))+
(f(\cdot- h\unit{d}{i}))](k)=2
[\mathcal{F}(f)](k)\cos(hk_i ).$
\end{enumerate}
\end{lemma}
\begin{lemma}
[Fourier transform of the solution]\label{x12}
Assume \cref{d134,y15}. Let
$h\in \pi/\N$,
$u\in \mathbbm{H}_{d,h,\geq 0}$. Then
it holds
for all
$k\in \Z^{d-1}$, $n\in\N_0$ that
\begin{align}
[\mathcal{F}(u(\cdot,nh))](k)= ((Q\circ\lambda)(hk))^{-n} [\mathcal{F}(u(\cdot,0))](k).
\end{align}
\end{lemma}
\begin{proof}[Proof of \cref{x12}]
Throughout the proof let $v:(h\Z^{d-1})\times (h\N_0)\to \mathbbm{C}$ be the function which satisfies
for all 
$k\in\Z^{d-1}$,
$(x,y)\in (h\Z^{d-1})\times (h\N_0)$, $  i\in \disint{1}{d-1}$, $n\in\N$ that
\begin{align}\label{z11b}
v(x,y)=v(x+2\pi\unit{d-1}{i},y)\quad\text{and}\quad
[\mathcal{F}(v(\cdot,nh))](k)= ((Q\circ\lambda)(hk))^{-n} [\mathcal{F}(u(\cdot,0))](k),
\end{align}
 defined through its Fourier transform, and let 
$\tilde{v},\tilde{u}:\Z^{d-1}\times \N_0\to \C$ be the function given by
\begin{align}\label{z11}
\forall\, (x,y)\in \Z^{d-1}\times \N_0\colon \quad 
\tilde{v}(x,y)= v(hx,hy)
\quad\text{and}\quad
\tilde{u}(x,y)= u(hx,hy).
\end{align}
\cref{y16} then proves for all
 $k\in\I_{2\pi/h}^{d-1}$, $n\in\N$ that
\begin{align}\begin{split}
&[\mathcal{F}(v(\cdot ,(n+1)h))](k)
+[\mathcal{F}(v(\cdot ,(n-1)h))](k)\\
&=
 ((Q\circ\lambda)(hk))^{-(n+1)} [\mathcal{F}(u(\cdot,0))](k)
+ ((Q\circ\lambda)(hk))^{-(n-1)} [\mathcal{F}(u(\cdot,0))](k)\\
&=
2\lambda(hk)
((Q\circ\lambda)(hk))^{-n} \fourier{u(\cdot,0)}=
2\lambda(hk)[\mathcal{F}(v(\cdot ,nh)](k).
\end{split}\label{z10}
\end{align}
\cref{z15} (with $d\defeq d-1$, $f\defeq v(\cdot ,nh)$ for $n\in\N$,
$k\in\I_{\pi/h}^{d-1}$, $i\in \disint{1}{d-1}$) then shows for all 
$k\in\I_{\pi/h}^{d-1}$, $n\in\N$
that
\begin{align}\begin{split}
&[\mathcal{F}((\Laplace^h v)(\cdot,nh))](k)\\
&=\left[\sum_{i=1}^{d-1} 2[\mathcal{F}(v(\cdot ,nh))](k)\cos (hk_i)\right]+[\mathcal{F}(v(\cdot ,(n+1)h))](k)\\&\qquad\qquad
+[\mathcal{F}(v(\cdot ,(n-1)h))](k)
-2d[\mathcal{F}(v(\cdot ,nh))](k)\\
&=[\mathcal{F}(v(\cdot ,(n+1)h))](k)
+[\mathcal{F}(v(\cdot ,(n-1)h))](k)-2\lambda(hk)[\mathcal{F}(v(\cdot ,nh))](k)=0
.\end{split}
\end{align}
This proves for all
$(x,y)\in (h\Z^{d-1})\times (h\N)$ that $(\Laplace^hv)(x,y)=0$. A scaling argument and \eqref{z11} then yield for all
$(x,y)\in \Z^{d-1}\times \N$ that $\Laplace \tilde{v}(x,y)=0$.
Furthermore, \cref{x10} and the fact that
$\forall\, t\in [-\pi,\pi]^{d-1}\colon Q(\lambda(t))\geq \lambda(t)\geq 1$ imply for all $N\in \N_0$ that
\begin{align}\begin{split}    
& ( 2\pi h)^{d-1}\sup_{x\in \omega_h^{d-1}} |v(x,Nh)|^2
\leq 
 ( 2\pi h)^{d-1}\sum_{x\in \omega_h^{d-1}} |v(x,Nh)|^2=
\sum_{k\in \I_{\pi/h}^{d-1}}\left|[\mathcal{F}\left(v(\cdot,Nh)\right)](k)\right|^2 \\&=\sum_{k\in \I_{\pi/h}^{d-1}}\left|Q^{-N}(\lambda(kh))[\mathcal{F}\left(v(\cdot,0)\right)](k)\right|^2\leq
\sum_{k\in \I_{\pi/h}^{d-1}}\left|[\mathcal{F}\left(v(\cdot,0)\right)](k)\right|^2=(2\pi h)^{d-1}\sum_{x\in \omega_h^{d-1}} |v(x,0)|^2.\end{split}
\end{align}
Hence, $v$ is a bounded functions. This and \eqref{z11} imply that $\tilde{v}$ is bounded. Moreover, 
the fact that
$\forall\, (x,y)\in (h\Z^{d-1})\times (h\N)\colon (\Laplace^hu)(x,y)=0$,
 \eqref{x01b},
\eqref{z11}, and a scaling argument prove that 
$\forall\, (x,y)\in \Z^{d-1}\times \N\colon \Laplace \tilde{u}(x,y)=0$. This, the assumption that $u$ is bounded, the fact that
$\forall\, (x,y)\in \Z^{d-1}\times \N\colon (\Laplace \tilde{v})(x,y)=(\Laplace \tilde{u})(x,y)=0$, the fact that $\tilde{v}$ is bounded,
\cref{z01} (with $u\defeq \tilde{v}$ and $u\defeq \tilde{u}$), and the fact that
$\forall\, x\in (h\Z^{d-1})\times\{0\}\colon  \tilde{v}(x)=\tilde{u}(x)$ (see \eqref{z11b} and \eqref{z11})
 ensure that $\tilde{u}=\tilde{v}$. This and \eqref{z11} imply that $u=v$. Combining this with \eqref{z11b} we complete the proof. 
\end{proof}
\begin{lemma}\label{z18}
Assume \cref{d134,x01,y15}
and let $k\in \I_{\pi/h}^{d-1}$, $i\in \disint{1}{d-1}$,
$h\in \pi/\N$,
$u\in \mathbbm{H}_{d,h,\geq 0}$.
Then\begin{enumerate}[i)]
\item it holds that
$
[\mathcal{F}(\disDiff{h}{x,i}u)(\cdot,0)](k)=
h^{-1}
[\mathcal{F}\left( u(\cdot,0)\right)](k)(e^{-\ima hk_i}-1)$ and \item it holds that $
[\mathcal{F}((\disDiff{h}{y}u)(\cdot,0))](k)=
h^{-1}\left(Q(\lambda(hk))^{-1}-1\right) [\mathcal{F}\left( u(\cdot,0)\right)](k)$.
\end{enumerate}
\end{lemma}
\begin{proof}[Proof of \cref{z18}]
Observe that \eqref{z13}, \cref{z15}
(with $d\defeq d-1$, $f\defeq u(\cdot ,0)$), and the fact that
$u(\cdot ,0)\in\perfct{2\pi}{h\Z^{d-1}}$
imply that
\begin{align*}
[(\mathcal{F}(\disDiff{h}{x,i}u)(\cdot,0)](k)=
h^{-1}
[\mathcal{F}( u(\cdot+h\unit{d-1}{i},0))-\mathcal{F}(u(\cdot,0))](k)=
h^{-1}
[\mathcal{F}\left( u(\cdot,0)\right)](k)(e^{-\ima hk_i}-1)
\end{align*} and
\begin{align*}
[\mathcal{F}((\disDiff{h}{y}u)(\cdot,0))](k)=
h^{-1}\left[\mathcal{F}\left(u(\cdot,h)-u(\cdot,0)\right)\right](k)= h^{-1}\left({Q(\lambda(hk))}^{-1}-1\right) [\mathcal{F}\left( u(\cdot,0)]\right)(k).
\end{align*}
The proof of \cref{z18} is thus completed.
\end{proof}
\begin{corollary}[Multipliers in the Neumann case]\label{z17}Assume \cref{d134,x01,y15}
and let 
$h\in \pi/\N$,
$u\in \mathbbm{H}_{d,h,\geq 0}$,
$k\in \I_{\pi/h}^{d-1}$, $i\in \disint{1}{d-1}$. Then
it holds
 that
$
[\mathcal{F}((\disDiff{h}{x,i}u)(\cdot,0))](k)=\NM_i^h(k)[\mathcal{F}(
(\disDiff{h}{y}u)(\cdot,0))](k).
$
\end{corollary}
\begin{proof}[Proof of \cref{z17}]
\cref{z18} and \eqref{y14} prove that in the case $k=0$ it holds that
\begin{align}
[\mathcal{F}((\disDiff{h}{x,i}u)(\cdot,0))](k)=
[\mathcal{F}((\disDiff{h}{y}u)(\cdot,0))](k)=
\NM_i(hk)=0
\end{align} and in the case $k\neq 0$ it holds that $Q(\lambda(hk))\neq 1$ and
\begin{align}\begin{split}
[\mathcal{F}((\disDiff{h}{x,i}u)(\cdot,0))](k)&= h^{-1}
[\mathcal{F}\left( u(\cdot,0)\right)](k)(e^{-\ima hk_i}-1)\\&=
\frac{e^{-\ima hk_i}-1}{{Q(\lambda(hk))}^{-1}-1}h^{-1}\left(Q(\lambda(hk))^{-1}-1\right) [\mathcal{F}\left( u(\cdot,0)\right)](k)
\\
&=
\frac{e^{-\ima hk_i}-1}{{Q(\lambda(hk))}^{-1}-1}[\mathcal{F}((\disDiff{h}{y}u)(\cdot,0))](k)
=
\NM_i^h(k)
[\mathcal{F}((\disDiff{h}{x,i}u)(\cdot,0))](k).
\end{split}\end{align}
This completes the proof of \cref{z17}.
\end{proof}
In \cref{z16} below we see that in the Dirichlet case there are $(d-1)$ multipliers,
 which are the quotients
$\mathcal{F}((\disDiff{h}{y}u)(\cdot,0)) /\mathcal{F}((\disDiff{h}{x,i}u)(\cdot,0)) $,
$i\in\disint{1}{d-1}$,
and therefore not everywhere defined. Fortunately, we can still show that
for each dyadic rectangle
 there is a multiplier well-defined on it. In \cref{d133c} we will develop a Marcinkiewicz-type multiplier theorem to deal with this situation.
\begin{corollary}[Multipliers in the Dirichlet case]\label{z16}Assume \cref{d134,x01,y15}, let
$h\in \pi/\N$, $u\in \mathbbm{H}_{d,h,\geq 0}$,
$k\in\Z^{d-1}\setminus\{0\}$,
$i\in \disint{1}{d-1}$, $\nu\in\prod_{j=1}^{d}(D(k_j)\cap\I_{\pi/h}) $, assume that $k_i\neq 0$, and
let $D(\ell)\subseteq \R$, $\ell\in \mathbbm{Z}$, be the intervals given by
\begin{align}
\forall\, \ell \in \N\colon 
D(\ell)=[2^{\ell-1},2^{\ell}),\quad 
D(0)=(-1,1),\quad 
\forall\, \ell \in (-\N)\colon D(\ell )=
(-2^{|\ell|}, -2^{|\ell|-1}].\label{z19}
\end{align} 
Then 
\begin{enumerate}[i)]
\item\label{z16a} it holds that
$
[\mathcal{F}((\disDiff{h}{y}u)(\cdot,0))](0)=
[\mathcal{F}((\disDiff{h}{x,1}u)(\cdot,0))](0)=\ldots=
[\mathcal{F}((\disDiff{h}{x,d-1}u)(\cdot,0))](0)
=0
$
and
\item \label{z16b}
it holds
 that
$
[\mathcal{F}((\disDiff{h}{y}u)(\cdot,0))](\nu)=
\DM_i^h(\nu)[\mathcal{F}((\disDiff{h}{x,i}u)(\cdot,0))](\nu).
$
\end{enumerate}

\end{corollary}
\begin{proof}[Proof of \cref{z16}]
First, note that \cref{z18}  implies \cref{z16a}.
Next, observe that \eqref{z19}, the assumption that
$k_i\neq 0$, and the assumption that $\nu_i\in D(k_i)\cap\I_{\pi/h}$  prove that
$ e^{-\ima h\nu_i}-1\neq0 $.
\cref{z18} and \eqref{y14}  therefore show that
\begin{align}\begin{split}
&[\mathcal{F}((\disDiff{h}{y}u)(\cdot,0))](\nu)
=
h^{-1}\left(Q(\lambda(hk))^{-1}-1\right) [\mathcal{F}\left( u(\cdot,0)\right)](k)\\
&=\frac{\left(Q(\lambda(hk))^{-1}-1\right)h^{-1}
[\mathcal{F}\left( u(\cdot,0)\right)](k)(e^{-\ima hk_i}-1)}{(e^{-\ima h\nu_i}-1)}
=\frac{\left(Q(\lambda(hk))^{-1}-1\right)[\mathcal{F}(\disDiff{h}{x,i}u)(\cdot,0)](\nu)}{(e^{-\ima h\nu_i}-1)}\\&=
\DM_i^h(\nu)[\mathcal{F}((\disDiff{h}{x,i}u)(\cdot,0))](\nu).\end{split}
\end{align}
This completes the proof of \cref{z16}.
\end{proof}
\begin{lemma}\label{x11}Assume \cref{d134,x01,y15},
let
$h\in \pi/\N$,
$u\in \mathbbm{H}_{d,h,\geq 0}$,
let $\rdown\in [0,\infty)$,
$N\in \N$ satisfy that $Nh\geq \rdown $, and assume that  $\sum_{x\in \omega_h^{d-1}}u(x,0)=0$.
Then 
\begin{enumerate}[i)]
\item \label{x11b} 
it holds for all $k\in \I_{\pi/h}^{d-1}\setminus\{0\}$ that $Q^{-N}(\lambda(kh))\leq (1+\sqrt{c}\rdown)^{-1}$ and
\item \label{x11c}
it holds that  
$\left\|u(\cdot,Nh)\right\|_{L^2_h(\omega^{d-1}_h)}\leq (1+\sqrt{c}\rdown)^{-1}\left\|u(\cdot,0)\right\|_{L^2_h(\omega^{d-1}_h)}$.
\end{enumerate}
\end{lemma}
\begin{proof}[Proof of \cref{x11}]
Observe that \eqref{y14} and \cref{x32} show
for all $t\in[-\pi,\pi]^{d-1}$ that
\begin{align}
Q(\lambda(t))-1= \lambda(t)-1+\sqrt{\lambda(t)^2-1}\geq
\sqrt{\lambda(t)-1}
\geq \sqrt{c}|t|_\infty.
\end{align}
This,   Bernoulli's inequality, and the assumption that 
$Nh\geq \rdown$
show for all $
k\in\mathbbm{I}_L^{d-1}\setminus\{0\}$ that
\begin{align}\label{x11bb}
\frac{1}{Q^N(\lambda(kh))}\leq \frac{1}{\left(1+\sqrt{c}|kh|_\infty\right)^N}\leq
\frac{1}{1+\sqrt{c}N|kh|_\infty}\leq \frac{1}{ 
1+\sqrt{c}|k|_\infty\rdown} \leq \frac{1}{ 1+\sqrt{c}\rdown}.
\end{align}
This proves \cref{x11b}.  Observe that \eqref{y02b} and the assumption that $\sum_{x\in \omega_h^{d-1}}u(x,0)=0$ imply that
$
[\mathcal{F}(u(\cdot,0))](0)=0
$. The Plancherel identity (for details see \cref{x10}), \cref{x12}, and \eqref{x11bb} hence demonstrate that
\begin{align}\begin{split}    
%&(2\pi)^{d-1}\left\|u(\cdot,Nh)\right\|_{L^2_h(\omega^{d-1}_h)}^2=
& ( 2\pi h)^{d-1}\sum_{x\in \omega_h^{d-1}} |u(x,Nh)|^2=
\sum_{k\in \I_{\pi/h}^{d-1}}|[\mathcal{F}(u(\cdot,Nh))](k)|^2
=\sum_{k\in \I_{\pi/h}^{d-1}}\left|Q^{-N}(\lambda(kh))[\mathcal{F}(u(\cdot,0))](k)\right|^2\\
&
=\sum_{k\in \I_{\pi/h}^{d-1}\setminus\{0\}}\left|Q^{-N}(\lambda(kh))[\mathcal{F}(u(\cdot,0))](k)\right|^2\leq ( 1+\sqrt{c}\rdown)^{-2}
\sum_{k\in \I_{\pi/h}^{d-1}\setminus\{0\}}|[\mathcal{F}(u(\cdot,0))](k)|^2\\
&
= ( 1+\sqrt{c}\rdown)^{-2}
\sum_{k\in \I_{\pi/h}^{d-1}}|[\mathcal{F}(u(\cdot,0))](k)|^2=( 1+\sqrt{c}\rdown)^{-2}(2\pi h)^{d-1}\sum_{x\in \omega_h^{d-1}} |u(x,0)|^2.
%=( 1+\sqrt{c}\rdown)^{-2}(2\pi)^{d-1}\left\|u(\cdot,0)\right\|_{L_h^2(\omega^{d-1}_h)}^2.
\end{split}
\end{align}
This and  \eqref{x02c} imply \cref{x11c}. 
The proof of \cref{x11} is thus completed.
\end{proof}
Combining 
\cref{x11}, \cref{x02}, a scaling argument, and a Riesz-Thorin-type interpolation argument we obtain the following result, \cref{x12e}. For later use we only need the fact that the multiplicative constants do not depend on $N$. 
\begin{corollary}\label{x12e}Assume \cref{d134,x01}. Let
$h\in \pi/\N$,
$u\in \mathbbm{H}_{d,h,\geq 0}$,
let $\rdown\in [0,\infty)$,
$N\in \N$ satisfy that $Nh\geq \rdown $, and assume that  $\sum_{x\in \omega_h^{d-1}}u(x,0)=0$.
Then 
\begin{enumerate}[i)]
\item\label{x13b} it holds for all $p\in [1,2]$ that
$\left\|u(\cdot,N)\right\|_{L^p_h(\omega_h^{d-1})}\leq(1+\sqrt{c}\rdown)^{-\left(2-\frac{2}{p}\right)} \left\|u(\cdot,0)\right\|_{L^p_h(\omega_h^{d-1})}$ and
\item \label{x13c} it holds for all $p\in [2,\infty]$ that
$\left\|u(\cdot,N)\right\|_{L^p_h(\omega_h^{d-1})}\leq(1+\sqrt{c}\rdown)^{-\frac{2}{p}} \left\|u(\cdot,0)\right\|_{L^p_h(\omega_h^{d-1})}$.
\end{enumerate}
\end{corollary}
\subsection{A Marcinkiewicz-type theorem for more than one multipliers}\label{d133c}
This subsection
slightly extends the classical Marcinkiewicz multiplier theorem  to the case of more than one multipliers (see \cref{d133}). 
It will be used to bound the normal component by $(d-1)$ tangential components.
In this case there are $(d-1)$ multipliers, however, each multiplier is \emph{not everywhere well-defined} as seen in \cref{z16}.
In the continuum setting this issue is overcome by considering a partition of unity
(see the paragraph between (83) and (84) in \cite{BFO18}).

The argument here also relies on local properties of the multipliers. Roughly speaking, the function $\lvar$ in \cref{y30} below, called the \emph{local variation}, measures the variation of a function on each dyadic rectangle. 
\cref{d133} proves that we still obtain
$L^p$-estimates, $p\in (1,\infty)$, if 
for each dyadic rectangle
 there is a nice multiplier defined on it. 
Moreover, in order to conveniently verify an assumption in \cref{d133} we use
\cref{z21}.

The notation in \cref{y30} below, e.g., \eqref{y04}--\eqref{y01b},
is again inspired by \cite[Section 2.5]{JS14}. 
%Some terminologies here have been already introduced in \cref{d134,d133a}.

\begin{setting}\label{y30}Let \cref{d134} be given.
Let $D(\ell)\subseteq \R$, $\ell\in \mathbbm{Z}$, be the intervals given by
\begin{align}
\forall\, \ell \in \N\colon 
D(\ell)=[2^{\ell-1},2^{\ell}),\quad 
D(0)=(-1,1),\quad 
\forall\, \ell \in (-\N)\colon D(\ell )=
(-2^{|\ell|}, -2^{|\ell|-1}].
\label{y03}
\end{align} 
For every $\beta\in \{0,1\}$ and every finite set $A\subseteq \Z $ we write
\begin{align}
\sum_{\nu\in A}^\beta =\begin{cases}\displaystyle
\sum_{\nu\in A}&\quad\colon \beta=1\\\displaystyle
\sup_{\nu\in A}&\quad \colon \beta=0.
\end{cases}\label{y04}
\end{align}
For every $d\in \N$, 
$\alpha\in \{0,1\}$,
$f\colon \mathbbm{Z}^d\to\C$
let 
$ \Delta_i^\alpha f\colon \Z^d\to \C$ be the functions given by
\begin{align}
\forall\, i\in\disint{1}{d},\, x\in\mathbbm{Z}^d\colon\quad  (\Delta_i^\alpha) f(x)=
\begin{cases}
f(x+\unit{d}{i})-f(x)&\colon \alpha=1 \\
f(x)&\colon \alpha=0.
\end{cases}\label{y04b}
\end{align}
Let
$
\var\colon 
\bigcup_{d,L\in\N} \perfct{2L}{\Z^d}\to\R
$
be the so-called \emph{total variation}, i.e., the function which satisfies for all $d,L\in\N$, $a\in\perfct{2L}{\Z^d}$ that
\begin{align}
\var(a)=\sup_{k\in\Z^d}
\max_{\alpha\in \{0,1\}^d}\left[
\sum_{\nu_{1}\in D(k_1)\cap\mathbbm{I}_L }^{\alpha_1}\ldots
\sum_{\nu_d\in D(k_d)\cap\mathbbm{I}_L }^{\alpha_d} \left|(\Delta_{1}^{\alpha_1}\ldots\Delta_{d}^{\alpha_d} a)(\nu)\right|\right].\label{y01b}
\end{align}
For every $\beta\in \{0,1\}$ and every \emph{finite} set $A\subseteq \Z $ we write
\begin{align}
\sum_{\nu\in A}^{\beta,1} =\begin{cases}\displaystyle
\sum_{\nu\in A\setminus\{\max A\}}&\quad\colon \beta=1\\
\displaystyle\sup_{\nu\in A}&\quad \colon \beta=0.
\end{cases}\label{y26}
\end{align}
Let
$
\lvar\colon 
\bigcup_{d,L\in\N} \left(\perfct{2L}{\Z^d}\times \Z^d\right)\to\R
$
be the function, called the \emph{local variation}, which satisfies
for all 
$d,L\in\N$, $a\in\perfct{2L}{\Z^d}$, $k\in\Z^d$ that
\begin{align}\begin{split}
\lvar(a,k)=
\max_{\alpha\in \{0,1\}^d}\left[
\sum_{\nu_{1}\in D(k_1)\cap\mathbbm{I}_L }^{\alpha_1,1}\ldots
\sum_{\nu_d\in D(k_d)\cap\mathbbm{I}_L }^{\alpha_d,1} \left|(\Delta_{1}^{\alpha_1}\ldots\Delta_{d}^{\alpha_d} a)(\nu_1,\ldots,\nu_d)\right|\right].
\label{y01}\end{split}
\end{align}
\end{setting}
In \cref{z22c} below we explain the purpose of introducing \eqref{y26} and \eqref{y01}.
\begin{lemma}\label{z22c}Assume \cref{y30}, let $d,L\in\N$,   $f,g\in \perfct{2L}{\Z^d}$, $k\in \Z^d $, and assume for all  $\nu\in \prod_{j=1}^{d} (D(k_j)\cap\I_L)$ that $ f(\nu)=g(\nu)$. Then it holds that $\lvar(f,k)=\lvar(g,k)$.
\end{lemma}
\begin{proof}[Proof of \cref{z22c}]Let us shows that for all $\alpha\in\{0,1\}^d$ it holds that
\begin{align}
\sum_{\nu_1\in D(k_1)\cap\I_L}^{\alpha_1,1}\ldots
\sum_{\nu_d\in D(k_d)\cap\I_L}^{\alpha_d,1}
(\Delta_1^{\alpha_1}\ldots \Delta_d^{\alpha_d} f)(\nu)
=
\sum_{\nu_1\in D(k_1)\cap\I_L}^{\alpha_1,1}\ldots
\sum_{\nu_d\in D(k_d)\cap\I_L}^{\alpha_d,1}
(\Delta_1^{\alpha_1}\ldots \Delta_d^{\alpha_d} g)(\nu).\label{z22b}
\end{align}
First, we consider the case $d=1$. If $\alpha=0$, then \eqref{z22b} directly follows from \eqref{y26}. If $\alpha=1$, observe 
that for all 
$\nu\in (D(k)\cap\I_L)\setminus\max\{D(k)\cap\I_L\} $ it holds that
$f(\nu+1 )=g(\nu+1)$, $f(\nu)=g(\nu)$, and hence $(\Delta f)(\nu)=(\Delta g)(\nu)$ and 
\eqref{z22b} then follows from \eqref{y26}. 
Applying the result for $d=1$ successively we obtain  \cref{z22b} in the case  $d\geq 2$. Using \eqref{y01} then completes the proof of \cref{z22c}.
\end{proof}
\begin{lemma}\label{z22}Assume \cref{y30} and let 
$d,L\in\N$, $a\in\perfct{2L}{\Z^d}$. Then $\lvar(a,0)=|a(0)|$.
\end{lemma}
\begin{proof}[Proof of \cref{z22}]
Observe that
\eqref{y26} and the fact that $D(0)\cap\I_L=\{0\}$
show  for all
$f\in\perfct{2L}{\Z}$ that
$\sum_{\nu\in D(0)\cap\I_L}^{1,1}f(\nu)$  is an empty sum and
$\sum_{\nu\in D(0)\cap\I_L}^{0,1}f(\nu)=\sup_{\nu\in D(0)\cap\I_L}f(\nu)=f(0)$.
This, applied successively to each variable, and \eqref{y01} prove
that $\lvar(a,0)=|a(0)|$. 
\end{proof}
For the proof of
\cref{z21} below we use the mean value theorem. This is a routine idea (cf. the proof of Item (b) in Theorem~2.49 in \cite{JS14}).  
The proof is included only for convenience of the reader.
%\begin{remark}
%The definition of $\lvar$ in \eqref{y01} implies  for all
%$d,L\in\N$, $a\in\perfct{2L}{\Z^d}$, $k\in \Z^d\setminus\I_L^{d-1}$ that
%$\lvar(a,k)=0$.
%\end{remark}
\begin{lemma}[A sufficient condition  to bound local variations]\label{z21}
Assume \cref{y30}, let $d,L\in\N$, $M\in (0,\infty)$,
$k\in\Z^d$,
let $J$ be the set given by $J=\prod_{j=1}^{d}(D(k_j)\cap[-L+1,L])$, assume that $J\neq \emptyset$,
let
$A\in C(J,\C) $,
$a\in \perfct{2L}{\Z^d}$
 satisfy for all $\nu\in J\cap\Z^d$ that
$  a(k)=A(k)$, and
assume for all $\alpha\in\{0,1\}^d$, 
$\xi\in J\setminus\I_L^d$
that 
$\partial_1 ^{\alpha_1}\ldots\partial_d ^{\alpha_d} A\in C(J\setminus\I_L^d,\C) $ and
$
|\xi_1 ^{\alpha_1}\ldots\xi_d ^{\alpha_d} (\partial_1 ^{\alpha_1}\ldots\partial_d ^{\alpha_d} A)(\xi)|\leq M$.
Then it holds that $\lvar(a,k)\leq M$.
\end{lemma}
\begin{proof}[Proof of \cref{z21}]First,
\eqref{y03} and \eqref{y26} imply that for all
$k\in\Z$, $\xi\in D(k)$, $\alpha\in\{0,1\}$ it holds that
 that
$
 \sum_{\nu\in D(k)\cap\mathbbm{I}_L }^{\alpha,1}1\leq |\xi|^\alpha
$
where  the sum is an empty sum for $\alpha=1$, $k=0$.
This, the assumption that $\forall\, \nu\in J\cap\Z^d\colon   a(k)=A(k)$, the mean value theorem (applied to all $x_j$ with $\alpha_j\neq 0$), and the assumption that $\forall\,\alpha\in\{0,1\}^d,\,\xi\in J\setminus\I_L^d\colon 
|\xi_1 ^{\alpha_1}\ldots\xi_d ^{\alpha_d} (\partial_1 ^{\alpha_1}\ldots\partial_d ^{\alpha_d} A)(\xi)|\leq M$
 show  for all $\alpha\in\{0,1\}^d$ that
\begin{align}\begin{split}
&\sum_{\nu_{1}\in D(k_1)\cap\mathbbm{I}_L }^{\alpha_1,1}\ldots
\sum_{\nu_d\in D(k_d)\cap\mathbbm{I}_L }^{\alpha_d,1} \left|(\Delta_{1}^{\alpha_1}\ldots\Delta_{d}^{\alpha_d} a)(\nu)\right|\\&\leq 
\sum_{\nu_{1}\in D(k_1)\cap\mathbbm{I}_L}^{\alpha_1,1}\ldots
\sum_{\nu_d\in D(k_d)\cap\mathbbm{I}_L }^{\alpha_d,1}\left[ \sup_{\xi\in J\setminus\I_L^d}\left|(\partial_1 ^{\alpha_1}\ldots\partial_d ^{\alpha_d} A)(\xi)\right|\right]\\
&=
\left[
\sum_{\nu_{1}\in D(k_1)\cap\mathbbm{I}_L }^{\alpha_1,1}1\right]\!\ldots\!\left[
\sum_{\nu_d\in D(k_d)\cap\mathbbm{I}_L }^{\alpha_d,1}1\right]\!\!\left[ \sup_{\xi\in J\setminus\I_L^d}\left|(\partial_1 ^{\alpha_1}\ldots\partial_d ^{\alpha_d} A)(\xi)\right|\right]\\&\leq
 \left[\sup_{\xi\in J\setminus\I_L^d}
|\xi_1 ^{\alpha_1}\ldots\xi_d ^{\alpha_d} (\partial_1 ^{\alpha_1}\ldots\partial_d ^{\alpha_d} A)(\xi)|\right]\leq M.\end{split}
\end{align}
Combining this with \eqref{y01} completes the proof of \cref{z21}.
\end{proof}
\begin{lemma}[From local variations to total variations]\label{y31}
Assume \cref{y30} and let $d,L\in\N$, 
$a\in \perfct{2L}{\Z^d}$,
 $M\in (0,\infty)$ satisfy for all $k\in\Z^d$ 
that $\lvar(a,k)\leq M$. Then $\var(a)\leq 4^dM$.
\end{lemma}
\begin{proof}[Proof of \cref{y31}]
Throughout this proof 
for every $\alpha\in\{0,1\}$,
$f\in \perfct{2L}{\Z}$, $A\subseteq \I_L$ let $\partial A$ be the set given by 
$\partial A= \{\max A,(\max A+1)\mod (2L)\}$ and write
\begin{align}\label{y31b}
\sum_{\nu\in A}^{\alpha,0}f(\nu)= \sum_{\nu\in \partial A} f(\nu).
\end{align}
Then
 \eqref{y01} and the assumption that $\forall\,k\in\Z^d\colon \lvar(a,k)\leq M$ prove that for all $\ell\in\disint{0}{d}$, $s,\alpha\in\{0,1\}^d$ with $s_1=\ldots=s_\ell=0$ and $s_{\ell+1}=s_d=1$ it holds that
\begin{align}\begin{split}
&\sum_{\nu_1\in D(k_1)\cap\I_L}^{\alpha_1,s_1}\ldots
\sum_{\nu_d\in D(k_d)\cap\I_L}^{\alpha_d,s_d} \!\!\left|(\Delta_1^{\alpha_1s_1}\ldots \Delta_d^{\alpha_d s_d}a)(\nu)\right|\\
&= 
\sum_{\nu_1\in \partial (D(k_1)\cap\I_L)}\ldots
\sum_{\nu_\ell\in \partial (D(k_\ell)\cap\I_L)}\left(
\sum_{\nu_{\ell+1}\in  D(k_{\ell+1})\cap\I_L}^{\alpha_{\ell+1},1}\ldots
\sum_{\nu_d\in D(k_d)\cap\I_L}^{\alpha_d,1} \left|(\Delta_{\ell+1}^{\alpha_{\ell+1}}\ldots \Delta_d^{\alpha_d}a)(\nu)\right|\right)\\
&\leq 
\sum_{\nu_1\in \partial (D(k_1)\cap\I_L)}\ldots
\sum_{\nu_\ell\in \partial( D(k_\ell)\cap\I_L)}M\leq
2^dM.
\\
\end{split}\end{align}
A permutation of the coordinates  hence shows
for all $\ell\in\disint{0}{d}$, $s,\alpha\in\{0,1\}^d$ that
\begin{align}
\sum_{\nu_1\in J(k_1)}^{\alpha_1,s_1}\ldots
\sum_{\nu_d\in J(k_d)}^{\alpha_d,s_d} \!\!\left|(\Delta_1^{\alpha_1s_1}\ldots \Delta_d^{\alpha_d s_d}a)(\nu)\right|\leq 2^dM.\label{y36}
\end{align}
Next, the notation given in \eqref{y04}, \eqref{y04b}, and \eqref{y31b} and the fact that
$\forall\,a,b\in\R\colon |a-b|\leq |a|+|b|$
 demonstrate, in the one-dimensional case,  that
 for all $f\in \perfct{2L}{\Z}$, $k\in\Z$, $\alpha\in \{0,1\}$ it holds that
\begin{align}\begin{split}
\sum_{\nu\in D(k)\cap\I_L}^{\alpha} \left|(\Delta^\alpha f)(\nu)\right|
\leq 
\left[\sum_{\nu\in D(k)\cap\I_L}^{\alpha,1} \left|\Delta^\alpha f(\nu)\right|\right]+
\left[
\sum_{\nu\in D(k)\cap\I_L}^{\alpha,0}
|f(\nu)|\right]
=\sum_{s=0}^{1}\sum_{\nu\in D(k)\cap\I_L}^{\alpha,s} \left|\Delta^{\alpha s} f(\nu)\right|
.\end{split}
\end{align}
This (applied to each variable) and \eqref{y36} imply  for all 
$k\in \Z^d$, $\alpha\in\{0,1\}^d$  that
\begin{align}\begin{split}
&\sum_{\nu_1\in D(k_1)\cap\I_L}^{\alpha_1}\ldots
\sum_{\nu_d\in D(k_d)\cap\I_L}^{\alpha_d}\!\! \left|(\Delta_1^{\alpha_1}\ldots \Delta_d^{\alpha_d}a)(\nu)\right|
\\&\leq \sum_{s\in\{0,1\}^d}\!\left[
\sum_{\nu_1\in D(k_1)\cap\I_L}^{\alpha_1,s_1}\ldots
\sum_{\nu_d\in D(k_d)\cap\I_L}^{\alpha_d,s_d} \!\!\left|(\Delta_1^{\alpha_1s_1}\ldots \Delta_d^{\alpha_d s_d}a)(\nu)\right|\right]\leq \sum_{s\in\{0,1\}^d}(2^dM)=4^dM.
\\
\end{split}\end{align}
This and \eqref{y01b} complete the proof of \cref{y31}. 
\end{proof}
\begin{setting}\label{d110c}Let \cref{y30} be given.
Let
$\constMac\colon (1,\infty)\to[0,\infty]$ be the function
which satisfies
for all $p\in(1,\infty)$
that $\constMac(p)\in[0,\infty]$ is the smallest real extended number with the property that for all
$d\in\N$, $h\in\pi/\N$, $U,u\in\perfct{2\pi}{h\Z^d}$, $a\in \perfct{2\pi/h}{\Z^d}$, 
$M\in (0,\infty)$ with 
$\mathcal{F}(U)=a\mathcal{F}(u)$ and
$\var(a)\leq M$  it holds that 
\begin{align}
\left\|U \right\|_{L^p_h(\omega_h^{d})}\leq \constMac(p)M
\left\|u\right\|_{L^p_h(\omega_h^{d})}.\label{y28}
\end{align}
\end{setting}
\cref{d110b} below recalls the discrete Marcinkiewicz multiplier theorem (cf. \cite[Theorem~2.49]{JS14}):
\begin{lemma}[Marcinkiewicz' theorem]\label{d110b}
Assume \cref{d110c} and let $p\in(1,\infty)$. Then $\constMac(p)<\infty$.
\end{lemma}
\begin{corollary}\label{d133}Assume \cref{d110c}, let $d,m\in\N$, 
$h\in\pi/\N$, 
$U, u_1,\ldots,u_m\in \perfct{2\pi}{h\Z^d}$, 
$p\in(1,\infty)$,
$L\in\N$, assume that $h=\pi/L$,
let 
$(A_k)_{k\in\Z^d}\subseteq\perfct{2L}{\Z^d} $,
$J\colon \Z^d\to \disint{1}{m}$ satisfy for all $k\in\Z^d$, $\nu \in\prod_{j=1}^{d}(D(k_j)\cap\I_L)$ that
\begin{align}
\lvar(A_k,k)\leq M\quad\text{and}\quad(\mathcal{F}(U))(\nu)= A_k(\nu)(\mathcal{F}(u_{J(k)}))(\nu).\label{y24b}
\end{align}
Then 
it holds that $\constMac(p)<\infty$ and
\begin{align}
\left\|U\right\|_{L^p_h(\omega_h^{d})}\leq \constMac(p)4^dM\sum_{i=1}^m
\|u_i\|_{L^p_h(\omega_h^{d})}.
\end{align}
\end{corollary}
\begin{proof}[Proof of \cref{d133}]
The discrete Marcinkiewicz multiplier theorem (cf. \cite[Theorem~2.49]{JS14}) proves that  $\constMac(p)<\infty$.
For the rest of this proof 
 let $K\colon \Z^d\to \Z^d$ be the mapping 
which is $2L$-periodic, i.e., $\forall\, \nu \in\Z^d,\,i\in\disint{1}{d}\colon K(\nu)= K(\nu+2L\unit{d}{i})$ and satisfies for all $\nu\in \I_L^d$ that
$\prod_{j=1}^{d}D(K_j(\nu)) $ is the unique dyadic rectangle in the family $\left\{\prod_{j=1}^{d}D(k_j)\colon k\in\Z^d\right\}$ that contains $\nu$, let $a,b\in \perfct{2L}{\Z^d}$
be the functions which satisfy for all $\nu \in\Z^d$ that 
$
 a(\nu)= A_{K(\nu)} (\nu) $
and let $w\colon \omega_h^d\to \C$ be the function which satisfies for all 
$x\in \omega_h^d$ that 
\begin{align}\label{y24c}
w(x)=(2\pi)^{-d}\sum_{\nu\in\mathbbm{I}_L^d}[\mathcal{F}(u_{J(K(\nu))})] (\nu)e^{\ima \nu x}
\end{align}
Then it holds
for all $k\in\Z^d$, $\nu\in \prod_{j=1}^{d}(D(k_j)\cap\I_L) $ that  $K(\nu)=k$, and $a(\nu)=A_k(\nu)$.
\cref{z22c} then
implies 
for all $k\in\Z^d$ with $\prod_{j=1}^{d}(D(k_j)\cap\I_L)\neq\emptyset$
that $\lvar (a,k)=\lvar(A_k,k)$. 
This, \eqref{y24b}, and \cref{y31} ensure that
$\var(a)\leq 4^dM$.
Moreover,  \eqref{y24c} and the Fourier inverse formula (combined with the assumption that
$L=\pi/h$) show 
for all $\nu \in \I_L^d$
that
\begin{align}
[\mathcal{F}(w)](\nu)=[\mathcal{F}(u_{J(K(\nu))})] (\nu)\quad\text{and}\quad|w(x)|\leq 
\sum_{i=1}^{m}\left|(2\pi)^{-d}\sum_{\nu\in\mathbbm{I}_L^d}
(\mathcal{F}\left(u_{i}\right))(\nu)e^{\ima \nu x}\right|=\sum_{i=1}^{m}|u_i(x)|.
\end{align}
Hence, \eqref{y24b}
(with $k\defeq K(\nu)$), and the fact that $\forall\,\nu \in\Z^d\colon  a(\nu)= A_{K(\nu)} (\nu) $ ensure
for all $\nu\in\I_L^d$ that
$
(\mathcal{F}(U))(\nu)= A_{K(\nu)}(\nu)(\mathcal{F}(u_{J(K(\nu))}))(\nu)= a(\nu)(\mathcal{F}(w))(\nu).
$ Hence,  $\mathcal{F}(U)= a\mathcal{F}(w)$.
This, \eqref{y28} (with $M\defeq 4^dM$, $u\defeq w$), the fact that
$\var(a)\leq 4^dM$,  and
the triangle inequality demonstrate that
\begin{align}
\begin{aligned}
&\|U\|_{L^p_h(\omega_h^d)}\leq  \constMac(p)4^dM\left\|w\right\|_{L^p_h(\omega_h^d)}\leq    \constMac(p)4^dM\sum_{i=1}^m\| u_i\|_{L^p_h(\omega_h^d)}.
\end{aligned}
\end{align}
This  completes the proof of \cref{d133}.
\end{proof}
\subsection{Cauchy's integral formula revisited}
\cref{z21} requires estimates on the higher derivatives of the multipliers. Returning to a classical result we can avoid many
tedious calculations.
\begin{setting}\label{y10b}
Assume \cref{y15},
let 
$C,m,M\in[0,\infty] $ be the real extended number given by
\begin{align}\begin{split}
m&= \left[\inf_{\xi\in [1,3d]\times[-2d,2d] }|Q(\xi)|\right],\quad 
M=\left[\sup_{\xi\in [1,3d]\times[-2d,2d] }|Q(\xi)|\right],
\\
C^2&=\max\left\{\frac{3(d-1)}{2m},
\frac{M}{c(d-1)}\right\},
\end{split}
\end{align}
and
let $\mathcal{C}(z)\subseteq\C$, $z\in (1,\infty)\times\{0\}$, be the sets which satisfy for all $z\in (1,\infty)\times\{0\}$ that
\begin{align}\label{y10c}
\mathcal{C}(z)=\left\{\zeta\in\C\colon |\zeta-z|=\tfrac{1}{2}|z-1|\right\}.
\end{align}
\end{setting}
\begin{lemma}\label{y18}
Assume \cref{y10b}.
Then it holds for all 
$t\in [-\pi,\pi]^{d-1}\setminus\{0\}$ 
that 
\begin{align}
C,m,M\in (0,\infty),\quad 
\left[
\sup_{\zeta\in\mathcal{C}(\lambda(t))}\left|f(\zeta)\right|\right]\leq C|t|_\infty,\quad\text{and}\quad
\left[
\sup_{\zeta\in\mathcal{C}(\lambda(t))}\frac{1}{\left|f(\zeta)\right|}\right]\leq \frac{C}{|t|_\infty}.
\end{align}
\end{lemma}
\begin{proof}[Proof of \cref{y18}]The extreme value theorem and the fact that 
$\forall\, z\in \C \setminus(-\infty,0)\colon Q(z)\neq 0 $ ensure that $m,M\in (0,\infty)$ and hence $C\in (0,\infty)$.
Furthermore, 
the triangle inequality and \eqref{y10c} imply that
for all $z\in [1,\infty)$, $\zeta\in \mathcal{C}(z)$ it holds that
\begin{align}
|\zeta-1|= |(z-1)+(\zeta-z)|\geq |z-1|-|\zeta-z|=|z-1|-\tfrac12 |z-1|=\tfrac12 |z-1|
\end{align}
and
\begin{align}\begin{split}
|\zeta-1|&= \tfrac12\left|(\zeta-z)+(\zeta-1)+(z-1)\right|\leq \tfrac12|\zeta-z|+\tfrac12|\zeta-1|+\tfrac12|z-1|\\
&=\tfrac14|z-1|+\tfrac12|\zeta-1|+\tfrac12|z-1|=\tfrac34|z-1|+\tfrac12|\zeta-1|.
\end{split}
\end{align}
This shows for all $z\in [1,\infty)$, $\zeta\in \mathcal{C}(z)$ that $
\tfrac12|z-1|\leq |\zeta-1|\leq \tfrac32|z-1|$. This (with $z\defeq\lambda(t)$ for $t\in [-\pi,\pi]^{d-1}$)
 demonstrate that for all
$t\in[-\pi,\pi]^{d-1} $, $\zeta\in \mathcal{C}(\lambda(t))$ it holds that
\begin{align}
\frac{c}{2}(d-1)|t|_\infty^2\leq 
\frac12|\lambda(t)-1|\leq 
|\zeta-1|\leq \frac32|\lambda(t)-1|\leq 
 \frac{3}{4}(d-1)|t|_\infty^2.
\label{y17}
\end{align}
Furthermore, \eqref{y14} implies 
for all $t\in[-\pi,\pi]^{d-1} $ that $1\leq \lambda(t)\leq 2d$. 
Hence, \eqref{y10c} shows that
for all $t\in[-\pi,\pi]^{d-1} $ it holds that
$\mathcal{C}(\lambda(t))\subseteq [1,3d]\times[-2d,2d]
\subseteq
\C\setminus( -\infty,1) .
$
This, \cref{y16}, and \eqref{y17} imply
for all
$t\in[-\pi,\pi]^{d-1} \setminus\{0\}$, $\zeta\in\mathcal{C}(\lambda(t))$ 
that
\begin{align} 
\left|f(\zeta)\right|^2
= 
\frac{2|\zeta-1|}{|Q(\zeta)|}
\leq  \frac{\tfrac{3}{2}(d-1)|t|_\infty^2}{m}\leq C^2 |t|_\infty^2
\end{align} and
\begin{align} 
\frac{1}{\left|f(\zeta)\right|^2}
= 
\frac{|Q(\zeta)|}{2|\zeta-1|}
\leq  
\frac{M}{c(d-1)|t|_\infty^2}\leq\frac{ C^2}{ |t|_\infty^{2}}.
\end{align}
This completes the proof of \cref{y18}.
\end{proof}
\begin{figure}\center
\includegraphics[]{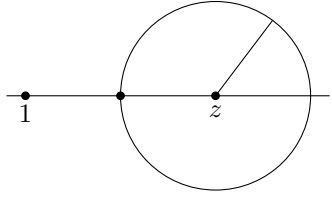}\caption{The circle $\mathcal{C}(z)=\{\zeta\in\C\colon |\zeta-z|=\frac{1}{2}|z-1|\}$ for some $z\in (1,\infty)\times\{0\}$}\label{d103}
\end{figure}
\begin{lemma}[Higher $t$-derivatives]\label{y11}
Assume \cref{y10b} and
let $h\colon \{\zeta\in\C\colon \Re(\zeta)>1\}\to\C$ be a holomorphic function.
Then it holds for all $t\in [-\pi,\pi] ^{d-1}\setminus\{0\}$, $\alpha\in \{0,1\}^{d-1}$ that 
\begin{align}\begin{split}
\left|
\left[\left(\tfrac{\partial }{\partial t_1}\right)^{\alpha_1}
\left(\tfrac{\partial }{\partial t_2}\right)^{\alpha_2}\ldots
\left(\tfrac{\partial }{\partial t_{d-1}}\right)^{\alpha_{d-1}}
 (h\circ \lambda)\right](t)\right|
\leq 
\frac{(2/c)^{|\alpha|}(|\alpha|)!}{|t|_\infty^{|\alpha|}}\!\left[\sup_{\zeta\in C(\lambda(t))}\left|h(\zeta)\right|\right].
\end{split}\label{y10}
\end{align}
\end{lemma}
\begin{proof}[Proof of \cref{y11}]
Cauchy's integral formula together with the assumption that $h$ is holomorphic and \eqref{y10c} proves for all $z\in (1,\infty)\times\{0\}$, $n\in\N_0$ that
\begin{align}\begin{split}
|h^{(n)}(z)|&= \left|\frac{n!}{2\pi\ima }\int_{\mathcal{C}(z)}\frac{h(\zeta)\,d\zeta}{(\zeta-z)^{n+1}}\right|\leq 
\frac{n!}{2\pi\left(\frac{1}{2}|z-1|\right)^{n+1}}
\left[\sup_{\zeta\in \mathcal{C}(z)}\left|h(\zeta)\right|\right]
\left[\int_{\zeta\in \mathcal{C}(z)}d|\zeta|\right]\\
&=\frac{n!}{2\pi(\frac{1}{2}|z-1|)^{n+1}}
\left[\sup_{\zeta\in \mathcal{C}(z)}\left|h(\zeta)\right|\right]
\left[2\pi \frac{1}{2}|z-1|\right]=\frac{2^nn!}{|z-1|^n}\left[\sup_{\zeta\in \mathcal{C}(z)}\left|h(\zeta)\right|\right].
\end{split}
\end{align}
Next,
\eqref{y14} and some elementary facts imply for all
$t\in  [-\pi,\pi]^{d-1}\setminus\{0\}$, $i\in \disint{1}{d-1}$ that
\begin{align}
\lambda(t)\in (1,\infty),
\quad  \tfrac{\partial}{\partial t_i}\lambda (t)=\sin(t_i),\quad \text{and}\quad 
|\sin(t_i)|\leq |t_i|.
\end{align}
This,  \eqref{y10c} (with $z\defeq\lambda(t)$), and
\cref{x32},
 show 
for all $t\in  [-\pi,\pi]^{d-1}\setminus\{0\} $, $\alpha\in \{0,1\}^{d-1}$
that
\begin{align}\begin{split}
&\left|
\left[
\left(\tfrac{\partial }{\partial t_1}\right)^{\alpha_1}
\left(\tfrac{\partial }{\partial t_2}\right)^{\alpha_2}\ldots
\left(\tfrac{\partial }{\partial t_{d-1}}\right)^{\alpha_{d-1}}
 (h\circ\lambda)\right](t)\right|=\left|h^{(|\alpha|)}(\lambda(t))
\prod_{i=1}^{d-1}
(\sin t_i)^{\alpha_i}
\right|\\
&\leq \left[
\left(\frac{2}{|\lambda(t)-1|}\right)^{|\alpha|}(|\alpha|)!\right]\!\!\left[\sup_{\zeta\in \mathcal{C}(\lambda(t))}\left|h(\zeta)\right|\right] |t|_\infty^{|\alpha|}\leq\left[\left(\frac{2}{c|t|^2_\infty}\right)^{|\alpha|}(|\alpha|)!\right]\!\!\left[\sup_{\zeta\in \mathcal{C}(\lambda(t))}\left|h(\zeta)\right|\right]|t|_\infty^{|\alpha|}.\\
\end{split}
\end{align}
This shows \eqref{y10} and completes the proof of \cref{y11}.
\end{proof}
\subsection{Total variations of the multipliers}
\begin{lemma}\label{y23}Assume \cref{y10b,y30}, let $a,\hat{C}\in [0,\infty)$ be the real numbers (cf. \cref{z05,y18}) given by
\begin{align}\label{y23b}
a=
\sup_{s\in [-\pi,\pi ]\setminus\{0\}}\max\left\{\Bigg.\! \left|\frac{d}{ds}(e^{-\ima s}-1) \right|,\left|\frac{e^{-\ima s}-1}{s}\right|\right\}\quad\text{and}\quad
\hat{C}=
2aC((2/c)^d\vee1)d!,
\end{align} 
and let $h\in \pi/\N$.
Then it holds for all $i\in \disint{1}{d-1}$ that
$\var(\NM_i^h)
\leq4^d\hat{C}$.
\end{lemma}
\begin{proof}[Proof of \cref{y23}]We first do some simple calculations on the derivatives. 
First, \eqref{y14} shows  for all 
$i\in \disint{1}{d-1}$, $\alpha\in \{0,1\}^{d-1}$ 
with $\alpha_i=0$   that
\begin{align}\begin{split}
\left(\partial_{1}^{\alpha_{1}} \ldots
\partial_{d-1}^{\alpha_{d-1}} \NM_i\right)(t)&=
\left(\frac{\partial }{\partial t_1}\right)^{\alpha_1}
\left(\frac{\partial }{\partial t_2}\right)^{\alpha_2}\ldots
\left(\frac{\partial }{\partial t_{d-1}}\right)^{\alpha_{d-1}} \left(\frac{e^{-\ima t_i}-1}{f(\lambda(t))}\right)\\
&= \left[
\left(\frac{\partial }{\partial t_1}\right)^{\alpha_1}
\left(\frac{\partial }{\partial t_2}\right)^{\alpha_2}\ldots
\left(\frac{\partial }{\partial t_{d-1}}\right)^{\alpha_{d-1}}
\left(
\frac{1}{f(\lambda(t))}\right)\right]\left(e^{-\ima t_i}-1\right).
\end{split}\label{y20}
\end{align}
Next, note that for all $i\in \disint{1}{d-1}$ it holds that
\begin{align}
\frac{\partial}{\partial t_i}\left(\frac{e^{-\ima t_i}-1}{f(\lambda(t))}\right)
= 
\frac{1}{f(\lambda(t))}\left[\frac{\partial}{\partial t_i}\left(e^{-\ima t_i}-1\right)\right]
+\left[\frac{\partial}{\partial t_i}
\left(\frac{1}{f(\lambda(t))} \right)\right]\left(e^{-\ima t_i}-1\right).
\end{align}
This and \eqref{y14}  ensure
for all $i\in \disint{1}{d-1}$, $\alpha\in \{0,1\}^{d-1}$ 
with $\alpha_i=1$ that
\begin{align}\begin{split}
&\left(\partial_{1}^{\alpha_{1}} \ldots
\partial_{d-1}^{\alpha_{d-1}} \NM_i\right)(t)=\left(\frac{\partial }{\partial t_1}\right)^{\alpha_1}
\left(\frac{\partial }{\partial t_2}\right)^{\alpha_2}\ldots
\left(\frac{\partial }{\partial t_{d-1}}\right)^{\alpha_{d-1}}\left(\frac{e^{-\ima t_i}-1}{f(\lambda(t))}\right)\\
&= \left[
\left(\frac{\partial }{\partial t_1}\right)^{\alpha_1}
\ldots
\left(\frac{\partial }{\partial t_{i-1}}\right)^{\alpha_{i-1}}
\left(\frac{\partial }{\partial t_{i+1}}\right)^{\alpha_{i+1}}
\ldots
\left(\frac{\partial }{\partial t_{d-1}}\right)^{\alpha_{d-1}}
\left(\frac{1}{f(\lambda(t))}\right)\right]\left[\frac{\partial}{\partial t_i}\left(e^{-\ima t_i}-1\right)\right]\\
&\qquad
+\left[\left(\frac{\partial }{\partial t_1}\right)^{\alpha_1}
\left(\frac{\partial }{\partial t_2}\right)^{\alpha_2}\ldots
\left(\frac{\partial }{\partial t_{d-1}}\right)^{\alpha_{d-1}}
\left(\frac{1}{f(\lambda(t))} \right)\right]\left(e^{-\ima t_i}-1\right).
\end{split}\label{y22}
\end{align}
Moreover, \cref{y11} (applied with $h\defeq1/f$), \cref{y18}, 
and the fact that $\forall\,\alpha\in \{0,1\}^{d-1}\colon |\alpha|\leq d$
ensure that 
for all $\alpha\in \{0,1\}^{d-1}$ it holds that
\begin{align}\begin{split}
\left|
\left(\tfrac{\partial }{\partial t_1}\right)^{\alpha_1}
\left(\tfrac{\partial }{\partial t_2}\right)^{\alpha_2}\ldots
\left(\tfrac{\partial }{\partial t_{d-1}}\right)^{\alpha_{d-1}}
\left(
\frac{1}{f(\lambda(t))}\right)\right|
\leq 
\frac{C((2/c)^d\vee1)d!}{|t|_\infty^{|\alpha|+1}}
\end{split}\label{y21}
\end{align}
This 
(with $\alpha\defeq (\alpha_1,\ldots,\alpha_{i-1},\alpha_{i+1},\ldots,\alpha_d)$
and with $\alpha\defeq\alpha$ for $\alpha\in \{0,1\}^{d-1}$, $i\in \disint{1}{d-1}$), \eqref{y20}, \eqref{y22}, the triangle inequality, and \eqref{y23b} prove
for all $i\in \disint{1}{d-1}$, $\alpha\in \{0,1\}^{d-1}$ that
\begin{align}
|t|_\infty^{|\alpha|}\left|\left(\partial_{1}^{\alpha_{1}} \ldots
\partial_{d-1}^{\alpha_{d-1}} \NM_i\right)(t)\right|
\leq|t|_\infty^{|\alpha|}
C((2/c)^d\vee1)d! \left[\frac{a}{|t|_\infty^{|\alpha|}}+\frac{a|t|_\infty}{|t|_\infty^{|\alpha|+1}}\right]=2aC((2/c)^d\vee1)d!=\hat{C}.
\end{align}
This, \eqref{y14}, and the substitution $\xi\defeq t/h$ show 
for all 
$i\in \disint{1}{d-1}$,
$\alpha\in\{0,1\}^{d-1}$,  $ \xi\in [-\tfrac{\pi}{h},\tfrac{\pi}{h}]^{d-1}\setminus\{0\}$
that
$
|\xi|^{|\alpha|}_\infty
\left|(\partial_{1}^{\alpha_{1}} \ldots
\partial_{d-1}^{\alpha_{d-1}} \NM_i^h)(\xi)\right|
\leq\hat{C}.$ Then \cref{z21}  shows for all $i\in \disint{1}{d-1}$,  $k\in\Z^d\setminus\{0\}$ that $\lvar(\NM_i^h,k)
\leq\hat{C}.$ This and the fact that
$ \lvar(\NM_i^h,0)=|\NM^h_i(0)|=0$ (see \eqref{y14}, \eqref{y14b}, and \cref{z22}) prove for all $i\in \disint{1}{d-1}$,  $k\in\Z^d$ that $\lvar(\NM_i^h,k)
\leq\hat{C}.$ Hence, 
\cref{y31} ensures for all $i\in \disint{1}{d-1}$ that
$\var(\NM_i^h)
\leq4^d\hat{C}$.
This completes the proof of \cref{y23}.
\end{proof}
Combining \cref{y23}, \cref{d110b}, and \cref{z17} we obtain \cref{s28} below.
\begin{corollary}[The Neumann case]\label{s28}
Assume \cref{d134,x01}. Then there exists a function  $C\colon ([2,\infty)\cap\N)\times(1,\infty)\to (0,\infty)$ such that for all $h\in \pi/\N$,
$p\in (1,\infty)$, $d\in [2,\infty)\cap\N$,
$u\in \mathbbm{H}_{d,h,\geq 0}$ it holds that
$
\|\disDiff{h}{x}u\|_{L^p_h(\omega^{d-1}_h)}
\leq C(d,p)
\|\disDiff{h}{y}u\|_{L^p_h(\omega^{d-1}_h)}.
$
\end{corollary}

\begin{lemma}\label{s26}Assume \cref{x01,y30,y10b}  
 let $h\in \pi/\N$,
$u\in \mathbbm{H}_{d,h,\geq 0}$,
let $a,\hat{C}\in[0,\infty)$ be the real numbers (cf. \cref{z05,y18}) given by
\begin{align}\label{y20e}
a=\sup_{s\in [-\pi,\pi ]\setminus\{0\}}\max \left\{\Bigg.\!
\left|\frac{s}{e^{-\ima s}-1}\right|, s^2\left|\frac{\partial}{\partial s}\left(\frac{1}{e^{-\ima s}-1}\right)\right|\right\}\quad\text{and}\quad
\hat{C}=
6aC((2/c)^d\vee 1)d!,
\end{align}
let $J \colon \Z^{d-1}\to(\disint{1}{d-1})$ be the function which satisfies 
for all $k\in \Z^{d-1}$ that 
\begin{align}
J(k)=\min\left\{j\in\disint{1}{d-1}\colon |k_j|=
\left[ \max_{i\in\disint{1}{d-1}} |k_i|\right]\right\}.
\label{y12b}
\end{align} 
Then it holds
for all 
$k\in\Z^{d-1}$,
 $\nu\in \prod_{j=1}^{d-1}D(k_j)$ 
 that
\begin{align}\label{y12d}
[\mathcal{F}((\disDiff{h}{y}u)(\cdot,0))](\nu)=
\DM_{J(k)}^h(\nu)[\mathcal{F}((\disDiff{h}{x,J(k)}u)(\cdot,0))](\nu)\quad\text{and}\quad\var(\DM_{J(\cdot)}^h(\cdot))\leq 4^d\hat{C}.
\end{align}
\end{lemma}
\begin{proof}[Proof of \cref{s26}]
First, \eqref{y12b} shows for all
$k\in \Z^{d-1}\setminus\{0\}$ that $ k_{J(k)}\neq 0$.
This, \cref{z16b} in \cref{z16} (with $i\defeq J(k)$ for $k\in\Z^{d-1}\setminus\{0\}$), and \cref{z16a} in \cref{z16} prove for all 
$k\in\Z^{d-1}$,
 $\nu\in \prod_{j=1}^{d-1}D(k_j)$ 
 that
\begin{align}\label{y12e}
[\mathcal{F}((\disDiff{h}{y}u)(\cdot,0))](\nu)=
\DM_{J(k)}^h(\nu)[\mathcal{F}((\disDiff{h}{x,J(k)}u)(\cdot,0))](\nu).
\end{align}
For the rest of this proof 
let $\mu$ be the Lebesgue measure on the real line.
Note that  \eqref{y03} and a simple scaling argument show that 
it holds for all
$\ell\in\Z\setminus\{0\}$, $\xi\in (hD(\ell))$ that  $|\xi|\geq h\mu(D(\ell))$ and
it holds for all
$
\ell\in\Z$,  $\xi\in (hD(\ell))$ that $ |\xi|\leq 2h\mu(D(\ell))$.
Therefore,
 the fact that $\forall\,k\in \Z^{d-1}\setminus\{0\}\colon k_{J(k)}\neq 0$ and \eqref{y12b}  prove
for all
$k\in \Z^{d-1}\setminus\{0\}$, 
$t\in\prod_{j=1}^{d-1} (hD(k_j))$  that
\begin{align}
\left|t_{J(k)}\right|\geq h\mu(D(k_{J(k)}))= \left[\max_{i\in\disint{1}{d-1}}h\mu(D(k_{i}))\right]\geq\frac{1}{2}\left[\max_{i\in\disint{1}{d-1}}\left|t_{i}\right|\right]=\frac{1}{2}|t|_\infty.\label{y12c}
\end{align}
Next,  \eqref{y14} shows  for all 
$t\in [-\pi,\pi]^{d-1}$,
$i\in \disint{1}{d-1}$, $\alpha\in \{0,1\}^{d-1}$ 
with $\alpha_i=0$, $t_i\neq 0$  that
\begin{align}\begin{split}
\left(\partial_{1}^{\alpha_{1}} \ldots
\partial_{d-1}^{\alpha_{d-1}} \DM_i\right)(t)&=
\left(\frac{\partial }{\partial t_1}\right)^{\alpha_1}
\left(\frac{\partial }{\partial t_2}\right)^{\alpha_2}\ldots
\left(\frac{\partial }{\partial t_{d-1}}\right)^{\alpha_{d-1}} \left(\frac{f(\lambda(t))}{e^{-\ima t_i}-1}\right)\\
&= \left[
\left(\frac{\partial }{\partial t_1}\right)^{\alpha_1}
\left(\frac{\partial }{\partial t_2}\right)^{\alpha_2}\ldots
\left(\frac{\partial }{\partial t_{d-1}}\right)^{\alpha_{d-1}}
\Big(f(\lambda(t)\Big)\right]\left(\frac{1}{e^{-\ima t_i}-1}\right).
\end{split}\label{y20b}
\end{align}
Moreover, note that for all $i\in \disint{1}{d-1}$,
$t\in [-\pi,\pi]^{d-1}$ with $t_i\neq 0$
 it holds that
\begin{align}
\frac{\partial}{\partial t_i}\left(\frac{f(\lambda(t))}{e^{-\ima t_i}-1}\right)
= 
f(\lambda(t))\left[\frac{\partial}{\partial t_i}\left(\frac{1}{e^{-\ima t_i}-1}\right)\right]
+\left[\frac{\partial}{\partial t_i}
\Big(f(\lambda(t)) \Big)\right]\frac{1}{e^{-\ima t_i}-1}.
\end{align}
Then \eqref{y14}  ensures
for all $t\in [-\pi,\pi]^{d-1}$, $i\in \disint{1}{d-1}$, $\alpha\in \{0,1\}^{d-1}$ with $t_i\neq 0$, $\alpha_i=1$ that
\begin{align}\begin{split}
&\left(\partial_{1}^{\alpha_{1}} \ldots
\partial_{d-1}^{\alpha_{d-1}} \DM_i\right)(t)=\left(\frac{\partial }{\partial t_1}\right)^{\alpha_1}
\left(\frac{\partial }{\partial t_2}\right)^{\alpha_2}\ldots
\left(\frac{\partial }{\partial t_{d-1}}\right)^{\alpha_{d-1}}\left(\frac{f(\lambda(t))}{e^{-\ima t_i}-1}\right)\\
&= \left[
\left(\frac{\partial }{\partial t_1}\right)^{\alpha_1}
\ldots
\left(\frac{\partial }{\partial t_{i-1}}\right)^{\alpha_{i-1}}
\left(\frac{\partial }{\partial t_{i+1}}\right)^{\alpha_{i+1}}
\ldots
\left(\frac{\partial }{\partial t_{d-1}}\right)^{\alpha_{d-1}}
\Big(f(\lambda(t))\Big)\right]\left[\frac{\partial}{\partial t_i}\left(\frac{1}{e^{-\ima t_i}-1}\right)\right]\\
&\qquad
+\left[\left(\frac{\partial }{\partial t_1}\right)^{\alpha_1}
\left(\frac{\partial }{\partial t_2}\right)^{\alpha_2}\ldots
\left(\frac{\partial }{\partial t_{d-1}}\right)^{\alpha_{d-1}}
\Big(f(\lambda(t)) \Big)\right]\left(\frac{1}{e^{-\ima t_i}-1}\right).
\end{split}\label{y20c}
\end{align}
Furthermore,
\cref{y11} (with $h\defeq f$) and \cref{y18} ensure that  for all $t\in [-\pi,\pi]^{d-1}\setminus\{0\}$, $\alpha\in \{0,1\}^{d-1}$ it holds that
\begin{align}\begin{split}
\left|
\left(\tfrac{\partial }{\partial t_1}\right)^{\alpha_1}
\left(\tfrac{\partial }{\partial t_2}\right)^{\alpha_2}\ldots
\left(\tfrac{\partial }{\partial t_{d-1}}\right)^{\alpha_{d-1}}
\Big(f(\lambda(t))\Big)\right|
\leq \frac{C((2/c)^d\vee 1)d!}{|t|_\infty^{|\alpha|-1}}.
\end{split}\label{y21b}
\end{align}
Combining 
\eqref{y20b}, \eqref{y20c}, and the triangle inequality then 
shows that for all $t\in [-\pi,\pi]^{d-1}\setminus\{0\}$, $\alpha\in \{0,1\}^{d-1}$,
$i\in\disint{1}{d-1}$ it holds that
\begin{align}\begin{split}
\left|\partial_{1}^{\alpha_{1}} \ldots
\partial_{d-1}^{\alpha_{d-1}} \DM_{i}(t)\right|
&\leq C((2/c)^d\vee 1)d!
\left[
\frac{1}{|t|_\infty^{|\alpha|-2}}\frac{a}{|t_{i}|^2}+
\frac{1}{|t|_\infty^{|\alpha|-1}}\frac{a}{|t_{i}|}\right]
\end{split}\label{y20d}
\end{align}
This (with $i\defeq J(k)$) and \eqref{y12c} prove 
for all
$k\in \Z^{d-1}\setminus\{0\}$, 
$t\in \prod_{j=1}^{d-1}(hD(k_j))$,
$\alpha\in \{0,1\}^{d-1}$  that
\begin{align}\begin{split}
|t|^{|\alpha|}_\infty
\left|(\partial_{1}^{\alpha_{1}} \ldots
\partial_{d-1}^{\alpha_{d-1}} \DM_{J(k)})(t)\right|
\leq C((2/c)^d\vee 1)d!
\left[2^2a+2a\right]\leq 6aC((2/c)^d\vee 1)d!=\hat{C}.
\end{split}
\end{align}
This, \eqref{y14b}, and the chain rule  prove for all
$k\in \Z^{d-1}\setminus\{0\}$, 
$\xi\in \prod_{j=1}^{d-1}D(k_j)$,
$\alpha\in \{0,1\}^{d-1}$ that
$
|\xi|_\infty^{|\alpha|}
|(\partial_{1}^{\alpha_{1}} \ldots
\partial_{d-1}^{\alpha_{d-1}} \DM_{J(k)}^h)(\xi)|
\leq \hat{C}.
$
\cref{z21} hence ensures for all $k\in\Z^{d-1}\setminus\{0\}$ that
$
\lvar (\DM_{J(k)}^h,k)\leq \hat{C}.
$
This and the fact that
$ \lvar(\DM_{J(0)}^h,0)=|\DM^h_1(0)|=0$ (see \eqref{y12b}, \eqref{y14}, \eqref{y14b}, and \cref{z22}) prove for all  $k\in\Z^d$ that $\lvar(\DM_{J(k)}^h,k)
\leq\hat{C}.$ Hence, \cref{y31} ensures that $\var(\DM_{J(\cdot)}^h(\cdot))\leq 4^d\hat{C}$.
This and \eqref{y12d} complete the proof of \cref{s26}.
\end{proof}
Combining \cref{s26} and 
\cref{d133} we obtain \cref{s29} below.
\begin{corollary}[The Dirichlet case]\label{s29}
Assume \cref{d134,x01}. Then there exists   a function  $C\colon ([2,\infty)\cap\N)\times(1,\infty)\to (0,\infty)$ such that for all $h\in \pi/\N$,
$p\in (1,\infty)$, $d\in [2,\infty)\cap\N$,
$u\in \mathbbm{H}_{d,h,\geq 0}$ it holds that
\begin{align}\begin{split}
\big\|\disDiff{h}{y}u\big\|_{L^p_h(\omega^{d-1}_h)}\leq C(d,p)\left[
\sum_{i=1}^{d-1}\big\|\disDiff{h}{x}u\big\|_{L^p_h(\omega^{d-1}_h)}\right].\end{split}\end{align}
\end{corollary}
\section{Proof of the main theorem}\label{se02}
This section combines the results in the previous sections to prove \cref{d06}.
Recall that we use the same terminology as in the continuum case~\cite{BFO18}: estimate~\eqref{a01} is called
\emph{the Dirichlet case} and estimate~\eqref{a03} is called \emph{the Neumann case}.
\cref{e01} adapts inequality (78) in \cite{BFO18} into 
\cref{s01}, which
is proven by the same idea as in \cite{BFO18}, i.e., 
by constructing a telescope series of harmonic functions on haft spaces by means of Dirichlet conditions, see \cref{p05,p06} below. 
In 
\cref{e02}, \cref{r20} proves the main result in the Dirichlet case.  
In the proof of \cref{r20} we use the idea of odd reflections as in Step~2 in the proof of Lemma 4 in \cite{BFO18} and \cref{r15} illustrates this idea in the discrete case (cf. Section IV.2.1 in \cite{Ngu17} for a simple illustration 
in the two-dimensional case).
\cref{e03} adapts inequality (79) in \cite{BFO18} into \cref{s12}. Here, we also construct a telescope series of harmonic functions on haft spaces, however, now by means of Neumann conditions. 
The reader will see that there are quite a lot of similarities between the Dirichlet and the Neumann case. However, the two cases are not identical and it is necessary to adapt rigorously every step of the proof due to the discreteness. In \cref{e04} we prove carefully the main theorem in the Neumann case, 
see \cref{r20a},
although the argument is quite straightforward in the continuum case, as said in the last sentence in the proof of Lemma 4 in~\cite{BFO18}.
The idea of even reflections is explained in \cref{r15a} where some minor arguments are used to deal with the discreteness (see Section IV.2.1 in \cite{Ngu17} for an illustration 
in the two-dimensional case).

Throughout this section we always use the notation given by \cref{v02} below.
\begin{setting}\label{v02}
For every $d\in\N$, $A\subset \Z^d$, 
$u\colon A\to\R$ let
$ \disDiff{+}{i}u\colon 
\{x\in A\colon x+\unit{d}{i}\in A\}
\to  \R$, $i\in \disint{1}{d}$,
be the functions which satisfy for all
$i\in \disint{1}{d}$, $x\in A$ with $x+\unit{d}{i}\in A$
that
$
(\disDiff{+}{i}u)(x)= u\!\left(x+\unit{d}{i}\right)-u(x),
$
and
$ \disDiff{-}{i}u\colon 
\{x\in A\colon x-\unit{d}{i}\in A\}
\to  \R$, $i\in \disint{1}{d}$,
be the functions which satisfy for all
$i\in \disint{1}{d}$, $x\in A$ with $x-\unit{d}{i}\in A$
that
$
(\disDiff{-}{i}u)(x)= u\!\left(x-\unit{d}{i}\right)-u(x).
$
For every finite set $A$ 
and every function $u\colon A\to\R$ 
let $\mean{u}{A}\in \R$ be given by $\mean{u}{A}=\frac{1}{|A|}\sum_{x\in A}u(x)$.
%For every finite set $A$, $\R\subset \{\C,\R\}$, and $u\colon A\to\R$ denote by $$
\end{setting}
\subsection{Construction of Dirichlet extensions}\label{e01}

\begin{setting}[Harmonic functions and boundary conditions]\label{p05}Let \cref{v02} be given.
For $L\in\N $ let $\I_L$  be the set given by $\I_L=\disint{-L+1}{L} $.
For every $N,L\in\N $, $d\in [2,\infty)\cap \N$
let
$\mathbbm{S}_{d,L,N}$ be the set of all  functions
$u\colon \Z^{d-1}\times (\disint{0}{N})\to \mathbbm{R}$ with the properties that
\begin{enumerate}[(i)]
\item 
it holds
for all $x\in \Z^{d-1}\times (\disint{0}{N})$, $  i\in \disint{1}{d-1}$ that
$u(x)=u(x+2L\unit{d}{i})$ and
\item 
it holds for all $x\in \Z^{d-1}\times (\disint{1}{N-1})$
that $(\Laplace u)(x)=0$,
\end{enumerate}
let
$\mathbbm{B}_{d,L,N}$ be the set of all boundary conditions
$u\colon \Z^{d-1}\times \{0,N\}\to \mathbbm{R}$  which satisfy
for all $x\in \Z^{d-1}\times \{0,N\}$, $  i\in \disint{1}{d-1}$ that
$u(x)=u(x+2L\unit{d}{i})$,
let $\efull{d,L,N}\subseteq E_d$ be the set of edges
given by
\begin{align}
\efull{d,L,N}=\left\{(x,y)\in E_d\colon 
\tfrac12(x+y)\in \left([0,L]^{d-1}\times[0,N]\right)\setminus
\left([1,L-1]^{d-1}\times[1,N-1]\right)\right\},
\end{align}
let
$\mathbbm{H}_{d,L,\geq 0}$ be the set of all bounded functions
$u\colon \Z^{d-1}\times \N_0\to \mathbbm{R}$ with the properties that
\begin{enumerate}[(i)]
\item 
it holds
for all $x\in \Z^{d-1}\times \N_0$, $  i\in \disint{1}{d-1}$ that
$u(x)=u(x+2L\unit{d}{i})$ and
\item 
it holds for all $x\in \Z^{d-1}\times \N$
that $(\Laplace u)(x)=0$,
\end{enumerate}
and let
$\mathbbm{H}_{d,L,\leq N}$ be the set of all bounded functions
$u\colon \Z^{d-1}\times ((-\infty,N]\cap\Z)\to \mathbbm{R}$ with the properties that
\begin{enumerate}[(i)]
\item 
it holds
for all $x\in \Z^{d-1}\times ((-\infty,N]\cap\Z)$, $  i\in \disint{1}{d-1}$ that
$u(x)=u(x+2L\unit{d}{i})$ and
\item 
it holds for all $x\in \Z^{d-1}\times ((-\infty,N-1]\cap\Z)$
that $(\Laplace u)(x)=0$.
\end{enumerate}
\end{setting}
\cref{p05b,q24} below prepare two important inequalities, which follow from the results in the last sections. We will bound the telescope series by a geometric series using the fact that 
$\forall\, d\in [2,\infty),\,p\in (1,\infty)\colon C_1(d,p)<1$.
\begin{setting}[Regularity constants]\label{p05b}
Assume \cref{p05} and let $C_1,C_2\colon([2,\infty)\cap\N)\times(1,\infty)\to[0,\infty]$
 be the functions which satisfy that
\begin{enumerate}[i)]
\item it holds
for all $d\in [2,\infty)\cap\N$, $p\in (1,\infty)$
that
$C_1(d,p)$ is the smallest 
real extended number such  that for all
$N,L\in\N$,
$u\in \mathbbm{H}_{d,L,\geq 0}$ with $1/4\leq N/L$ and $\mean{u}{ \I_L^{d-1}\times\{0\}}=0$
it holds 
 that
\begin{align}\label{p02}
\left\|u\right\|_{L^p(\I_L^{d-1}\times\{N\})}\leq C_1(d,p)\left\|u\right\|_{L^p(\I_L^{d-1}\times\{0\})}
\end{align} and
\item it holds
for all $d\in [2,\infty)\cap\N$, $p\in (1,\infty)$
that
$C_2(d,p)$ is the smallest 
real extended number such  that for all
 $N,L\in\N$,
$i\in \disint{1}{d-1}$, $u\in \mathbbm{H}_{d,L,\geq 0}$ with $ N/L\leq 4$
it holds 
 that
\begin{align}
\|\nabla u\|_{L^p(\efull{d,L,N})}\leq  C_2(d,p)
\left[\sum_{j=1}^{d-1}
\left\|\disDiff{+}{j} u\right\|_{L^p(\I_L^{d-1}\times\{0,N\}) }
\right].\label{p09}
\end{align}
\end{enumerate}
\end{setting}
\begin{lemma}\label{q24}Assume \cref{p05b} and let 
$d\in [2,\infty)\cap\N$, $p\in (1,\infty)$ be fixed. Then it holds that
$C_1(d,p)<1$ and $C_2(d,p)<\infty$.
\end{lemma}
\begin{proof}[Heuristic proof of \cref{q24}]
First, note that the discrete derivatives of a harmonic function are still harmonic.  Using 
\cref{p01} (with the function replaced by the derivatives)
we bound the differences with respect to the edges
with endpoints on the face $\{x_1=0\}$ of the box by the tangential differences on the bottom $\Z^{d-1}\times\{0\}$ (see \eqref{q11}). 
Next, using \cref{t01a} in \cref{t01} we  bound
the normal differences on the top $\Z^{d-1}\times\{N\}$
by the normal differences on the bottom $\Z^{d-1}\times\{0\}$ (see \cref{q23}).
Furthermore, using \cref{t01c} in \cref{t01} we bound the normal differences on the bottom $\Z^{d-1}\times\{0\}$ 
(and hence also that on the top)
by the tangential differences the edges on the bottom $\Z^{d-1}\times\{0\}$ (see \eqref{q22}).
Using a permutation of the coordinates  we hence bound the differences with respect to all  edges with one endpoints on the boundary of the box by the tangential differences on the bottom. 
\end{proof}
\begin{proof}[Rigorous proof of \cref{q24}]
\cref{t01} implies that there exists
$c_0\colon  (0,\infty)\to (0,1)$
 such that
for all  $N,L\in\N$,
$u\in \mathbbm{H}_{d,L,\geq 0}$ with $N/L\geq \rdown$ and $\mean{u}{\I_L^{d-1}\times\{0\}}=0$ it holds that
$\left\|u\right\|_{L^p(\I_L^{d-1}\times\{N\})}\leq c_0(\rdown)\left\|u\right\|_{L^p(\I_L^{d-1}\times\{0\})}$. This (with $\rdown \defeq 1/4$) proves  
that $C_1(d,p)<1$.
Next, recall that 
\cref{p01} shows that there exists $c_1\colon(0,\infty)\to (0,\infty)$ such that for all 
$\rup\in (0,\infty)$, $N\in (0,L\rup]\cap\N$,
$u\in \mathbbm{H}_{d,L,\geq 0}$
it holds that
$
\|u\|_{L^p(\{0\}\times \I_L^{d-2}\times (\disint{1}{N}))}\leq  c_1(\rup)
\|u\|_{L^p(\I_L^{d-1}\times\{0\})}.
$
This (with $\rup \defeq 4$ and 
$u\defeq \disDiff{\pm}{i}u $ for 
$i\in\disint{1}{d}$, $L,N\in \N$, $u\in \mathbbm{H}_{d,L,\geq 0}$
 with
$N/L\leq 4$)
implies that
for all 
$L,N\in \N$ with
$N/L\leq 4$,
$u\in \mathbbm{H}_{d,L,\geq 0}$
it holds that
\begin{align}\label{q11}
\|\disDiff{\pm}{i}u\|_{L^p(\{0\}\times \I_L^{d-2}\times (\disint{1}{N}))}\leq  c_1(4)
\|\disDiff{\pm}{i}u\|_{L^p(\I_L^{d-1}\times\{0\})}.
\end{align}
Next, \cref{x13} (with $u\defeq \disDiff{+}{d}u$ and $N\defeq N-1$) shows for all $L\in\N$,
$u\in \mathbbm{H}_{d,L,\geq 0}$ that
\begin{align}\label{q23}
\left\|\disDiff{-}{d}u\right\|_{L^p(\I_L^{d-1}\times\{N\}) }\leq
 \left\|\disDiff{+}{d}u\right\|_{L^p(\I_L^{d-1}\times\{0\}) }.
\end{align}
Hence, \cref{t01}
 shows that there exists
$c_2\in  (0,\infty)$ such that
for all $L\in\N$,
$u\in \mathbbm{H}_{d,L,\geq 0}$ it holds 
that
\begin{align}\label{q22}
\left\|\disDiff{-}{d}u\right\|_{L^p(\I_L^{d-1}\times\{N\}) }\leq
\left\|\disDiff{+}{d}u\right\|_{L^p(\I_L^{d-1}\times\{0\}) }\leq
c_2 \left[\sum_{i=1}^{d-1}
\left\|\disDiff{+}{i} u\right\|_{L^p(\I_L^{d-1}\times\{0\}) }\right].
\end{align}
 Combining \eqref{q11} and \eqref{q22} then yields that there exists 
$c_3\in  (0,\infty)$ such  that for all
 $N,L\in\N$,
$i\in \disint{1}{d-1}$, $u\in \mathbbm{H}_{d,L,\geq 0}$ with $ N/L\leq 4$
it holds 
 that
\begin{align}
\|\nabla u\|_{L^p(\efull{d,L,N})}\leq  c_3
\left[\sum_{j=1}^{d-1}
\left\|\disDiff{+}{j} u\right\|_{L^p(\I_L^{d-1}\times\{0,N\}) }
\right].
\end{align}
This shows  that $C_2(d,p)<\infty$.
The proof of \cref{q24} is thus completed.
\end{proof}
Existence and uniqueness of the solutions to the Dirichlet problems on haft spaces (shown, e.g., by means of Fourier transforms in \cref{v03}) ensure that the sequences
$(u_k)_{k\in \N}$ in \cref{p06} below are well-defined by
\eqref{p03a}--\eqref{p03b}.
\begin{setting}[Telescope sequence for the Dirichlet case]\label{p06}
Assume \cref{p05b}, let  $N,L\in\N $, $p\in(1,\infty)$, $d\in [2,\infty)\cap \N$ be fixed and satisfy that $1/4\leq L/N\leq 4$, let 
$v\in\mathbbm{B}_{d,L,N}$ be a boundary condition which satisfies that
\begin{align}
\mean{v}{\I_L^{d-1}\times\{0\}}=0\quad\text{and}\quad\forall\, x\in\Z^{d-1}\times\{N\}\colon\quad  v(x)=0,\label{p03d}
\end{align}
and
let
\begin{align}
(u_{2k+1})_{k\in \N_0} \subseteq \mathbbm{H}_{d,L,\geq 0}\quad\text{and}\quad
(u_{2k+2})_{k\in \N_0}\subseteq  \mathbbm{H}_{d,L,\leq N}\label{p03e}
\end{align}
 be the sequences given by
\begin{align}
&\forall\, x\in\Z^{d-1}\times\{0\}\colon \quad u_1(x)=v(x),\quad 
\label{p03a}
\\
& \forall\, k\in\N_0,\, x\in \Z^{d-1}\times \{N\}\colon \quad u_{2k+2}(x)= u_{2k+1}(x),\label{p03c}\\
&\forall\,
k\in\mathbbm{N},\, x\in \Z^{d-1}\times\{0\}\colon \quad
u_{2k+1}(x)= u_{2k}(x).\label{p03b}
\end{align}

\end{setting}
\begin{lemma}[Convergence of the telescope series]\label{p03}
Assume \cref{p06}.
Then 
\begin{enumerate}[i)]
\item \label{p03f}it holds for all $n \in\N$ that
$\max_{y\in \{0,N\}}
\|u_{n}\|_{L^p(\I_{L}^{d-1}\times\{y\})}
=C_1(d,p)^{n-1}\|u_{1}\|_{L^p(\I_{L}^{d-1}\times\{0\})}$ and
\item \label{p03g}it holds for all
 $ x\in \Z^{d-1}\times( \disint{0}{N})$ that $ 
\sum_{k=1}^{\infty}|u_k(x)|<\infty $.
\end{enumerate}
\end{lemma}
\begin{proof}[Proof of \cref{p03}]
Observe that \eqref{p03a} and \eqref{p03d} imply
that $\mean{u_1}{\I_L^{d-1}\times\{0\}}=0$. This,
\eqref{p03c},
\eqref{p03b}, 
 and
\cref{z01} ensure for all $n\in\N$ that 
$\mean{u_n}{\I_L^{d-1}\times\{0\}}=
\mean {u_n}{ \I_L^{d-1}\times\{N\}}=0$. This,
\eqref{p03c}, \eqref{p03b}, \cref{x02},
 and \eqref{p02}  show for all $k\in\N_0$,
$y\in \disint{0}{N}$
 that
\begin{align}\begin{split}
\|u_{2k+2}\|_{L^p(\I_{L}^{d-1}\times\{y\})}&\leq \left\|u_{2k+2}\right\|_{L^p(\mathbbm{I}_L^{d-1}\times\{N\})} \\&
=\|u_{2k+1}\|_{L^p(\mathbbm{I}_L^{d-1}\times\{N\})} 
\leq 
C_1(d,p)\|u_{2k+1}\|_{L^p(\mathbbm{I}_L^{d-1}\times\{0\})}\end{split}\end{align}and
\begin{align}\begin{split}
\|u_{2k+3}\|_{L^p(\I_{L}^{d-1}\times\{y\})}
&\leq
\left\|u_{2k+3}\right\|_{L^p(\mathbbm{I}_L^{d-1}\times\{0\})}\\
&= 
\left\|u_{2k+2}\right\|_{L^p(\mathbbm{I}_L^{d-1}\times\{0\})} \leq C_1(d,p)
\left\|u_{2k+2}\right\|_{L^p(\mathbbm{I}_L^{d-1}\times\{N\})} 
 .\end{split}
\end{align}
This and an induction argument prove for all $n\in\N$  that 
\begin{align}\begin{split}
&\max_{y\in \disint{0}{N}}
\|u_{n+1}\|_{L^p(\I_{L}^{d-1}\times\{y\})}\leq 
C_1(d,p)\left[
 \max_{y\in \{0,N\}}\|u_{n}\|_{L^p(\I_{L}^{d-1}\times\{y\})}\right]\\
&\leq C_1(d,p)^n\left[\max_{y\in \{0,N\}}\|u_{1}\|_{L^p(\I_{L}^{d-1}\times\{y\})}\right]
=C_1(d,p)^n\|u_{1}\|_{L^p(\I_{L}^{d-1}\times\{0\})}.
\end{split}\label{p10}
\end{align}
This shows \cref{p03f}. Next, 
\eqref{p10}, the fact that $C_1\in (0,1)$, and the convergence of the geometric series assure for all $y\in \disint{0}{N}$ that
$\sum_{n=1}^{\infty}\|u_{n}\|_{L^p(\I_{L}^{d-1}\times\{y\})}<\infty $. 
This implies \cref{p03g}.  The proof of \cref{p03} is thus completed.
\end{proof}
\begin{lemma}[Upper bound for the telescope series]\label{p07}Assume \cref{p06} and let $w\colon \Z^{d-1}\times(\disint{0}{N})\to\R$ be the function given by
\begin{align}
\forall\, x\in \Z^{d-1}\times( \disint{0}{N})\colon \quad 
w(x)=  \left[\sum_{k=1}^\infty (-1)^{k+1} u_k(x)\right].\label{p08}
\end{align}
Then 
it holds for all $x\in \Z^{d-1}\times \{0,N\}$ that
$w(x)= v(x) $,
$w\in \mathbbm{S}_{d,L,N} $, and
\begin{align}
\begin{split}
\|\nabla w\|_{L^p(\efull{d,L,N})}
\leq \frac{C_2(d,p)}{1-C_1(d,p)}
\left[\sum_{i=1}^{d-1}
\left\|\disDiff{+}{i} v\right\|_{L^p(\I_L^{d-1}\times\{0,N\}) }
\right].\label{p11}
\end{split}\end{align}
\end{lemma}
\begin{proof}[Proof of \cref{p07}]
Note that
\eqref{p08}, \eqref{p03a}--\eqref{p03b}, 
 and a telescope sum argument demonstrate that 
for all 
$x\in \I_L^{d-1}\times\{0\}$ it holds  that
$w(x)=\sum_{k=1}^\infty (-1)^{k+1} u_k(x)=u_1(x)=v(x)$ and 
for all $x\in \I_L^{d-1}\times\{N\}$ it holds  that
$w(x)=\sum_{k=1}^\infty (-1)^{k+1} u_k(x)=0=v(x)$. Next, 
\eqref{p03e}  proves that 
\begin{align}\begin{split}
&\forall\,n\in \N_0,  x\in \I_L^{d-1}\times(\disint{0}{N}) ,
i\in\disint{1}{d-1}\colon\quad u_n(x+2L\unit{d}{i})=u_n(x)\quad\text{and}
\\
&\forall\,n\in \N_0,  x\in \I_L^{d-1}\times(\disint{1}{N-1}) 
\colon \quad  (\Laplace u_n)(x)=0 .
\end{split}\end{align}
This and \eqref{p08} imply that
$w\in \mathbbm{S}_{d,L,N} $.
Next, observe that \eqref{p03e} and a simple calculation imply for all 
$ k\in \N_0$, $i\in\disint{1}{d-1}$ that
$
\disDiff{+}{i}u_{2k+1} \in  \mathbbm{H}_{d,L,\geq 0}$ and $
\disDiff{+}{i}u_{2k+2}\in  \mathbbm{H}_{d,L,\leq N}$. Roughly speaking,
 discrete derivatives of harmonic functions are also harmonic. This, \eqref{p08}, the triangle inequality, \eqref{p09} (applied  with $u\defeq u_k $ for $k\in\N $), \cref{p03f} in \cref{p03} (with 
$v\defeq \disDiff{+}{i}u_1$ and
$(u_k)_{k\in \N}\defeq (\disDiff{+}{i} u_k)_{k\in\N}$ for $i\in\disint{1}{d-1}$),   the fact that $\forall\, x\in (0,1)\colon \sum_{k=1}^{\infty}x^{k-1}=(1-x)^{-1}$, and \eqref{p03a} ensure that
\begin{align}
\begin{split}
&\|\nabla w\|_{L^p(\efull{d,L,N})}\\
&\leq 
\left\|\sum_{k=1}^\infty\nabla u_k\right\|_{L^p(\efull{d,L,N})}
\leq  
\sum_{k=1}^\infty\left\|\nabla u_k\right\|_{L^p(\efull{d,L,N})}
\leq 
\sum_{k=1}^\infty
\left[C_2(d,p)\left(\sum_{i=1}^{d-1}
\left\|\disDiff{+}{i} u_k\right\|_{L^p(\I_L^{d-1}\times\{0,N\}) }\right)
\right]\\
&\leq 
\sum_{k=1}^\infty
\left[C_2(d,p)\left(\sum_{i=1}^{d-1}
C_1(d,p)^{k-1}
\left\|\disDiff{+}{i} u_1\right\|_{L^p(\I_L^{d-1}\times\{0,N\}) }\right)
\right]\\
&= \frac{C_2(d,p)}{1-C_1(d,p)}
\left[\sum_{i=1}^{d-1}
\left\|\disDiff{+}{i} u_1\right\|_{L^p(\I_L^{d-1}\times\{0,N\}) }
\right]
\leq \frac{C_2(d,p)}{1-C_1(d,p)}
\left[\sum_{i=1}^{d-1}
\left\|\disDiff{+}{i} v\right\|_{L^p(\I_L^{d-1}\times\{0,N\}) }
\right].
\end{split}\end{align}
The proof of \cref{p07} is thus completed.
\end{proof}
\cref{p15} below considers general boundary conditions (i.e. without the restriction \eqref{p03d}).
\begin{lemma}[Upper bound for harmonic functions on strips]\label{p15}
Assume \cref{p05}, let  $N,L\in\N $, $d\in [2,\infty)\cap \N$,
$v\in\mathbbm{B}_{d,L,N}$, $p\in(1,\infty)$ satisfy that $1/4\leq L/N\leq 4$. Then there exists $w\in \mathbbm{S}_{d,L,N}$ such that
for all 
$x\in \Z^{d-1}\times\{0,N\}$  it holds 
 that 
$w(x)=v(x)$ and
\begin{align}
\|\nabla w\|_{L^p(\efull{d,L,N})}\leq  \frac{2C_2(d,p)}{1-C_1(d,p)}
\left[\sum_{i=1}^{d-1}
\left\|\disDiff{+}{i} v\right\|_{L^p(\I_L^{d-1}\times\{0,N\}) }\right]
+\frac{4}{N}\left\|v\right\|_{L^p(\I_L^{d-1}\times\{0,N\}) }
.
\end{align}
\end{lemma}
\begin{proof}[Proof of \cref{p15}]
Throughout this proof 
let $v_1,v_2,v_3\in\mathbbm{B}_{d,L,N}$,
$w_3\colon \Z^{d-1}\times(\disint{0}{N})\to\R$ 
 be the functions given by
\begin{align}\begin{split}\label{p11a}
&\forall\, x\in\Z^{d-1}\times \{0\}\colon\quad  v_1(x)= v(x)- \mean{v}{\I_L^{d-1}\times \{0\} },
\quad v_2(x)=0, \quad v_3(x)=\mean{v}{\I_L^{d-1}\times \{0\} },
\\[3pt]
&\forall\, x\in\Z^{d-1}\times \{N\}\colon \quad v_1(x)=0,\quad   v_2(x)= v(x)- \mean{v}{\I_L^{d-1}\times \{N\} },\quad 
v_3(x)=\mean{v}{\I_L^{d-1}\times \{N\} },
\\
&\forall\, n\in \disint{0}{N},\, x\in \Z^ {d-1}\times\{n\}\colon \quad 
w_3(x)=\frac{n}{N}\mean{v}{\I_L^{d-1}\times\{0\}}+
\frac{N-n}{N}\mean{v}{\I_L^{d-1}\times\{N\}}.
\end{split}\end{align}
This implies that
$w_3\in \mathbbm{S}_{d,L,N}$. Next, \eqref{p11a}, the fact that
$\forall\,a,b\in \R\colon |a-b|^p\leq 2^{p-1}(|a|^p+|b|^p) $,
the fact that
$\left\|1\right\|^p_{L^p(\I_L^{d-1}\times\{0\})}=\left\|1\right\|^p_{L^p(\I_L^{d-1}\times\{N\})}$,
 and Jensen's inequality imply that
\begin{align} \begin{split}
&\|\nabla w_3\|_{L^p(\efull{d,L,N})}^p
=2 \left\|\disDiff{+}{d}w_3\right\|^p_{L^p(\I_L^{d-1}\times\{0\})}+2
\left\|\disDiff{-}{d}w_3\right\|^p_{L^p(\I_L^{d-1}\times\{N\})}\\
&=2 \left\|\frac1N\left(\mean{v}{\I_L^{d-1}\times \{0\}}-\mean{v}{\I_L^{d-1}\times \{N\}}\right)\right\|^p_{L^p(\I_L^{d-1}\times\{0\})}+2
\left\|\frac{1}{N}\left(\mean{v}{\I_L^{d-1}\times \{0\}}-\mean{v}{\I_L^{d-1}\times \{N\}}\right)\right\|^p_{L^p(\I_L^{d-1}\times\{N\})}\\
&=\frac{4}{N^p}\left|\mean{v}{\I_L^{d-1}\times \{0\}}-\mean{v}{\I_L^{d-1}\times \{N\}}\right|^p \left\|1\right\|^p_{L^p(\I_L^{d-1}\times\{0\})}\\
&
\leq 
\frac{4\cdot 2^{p-1}}{N^p}\left[\mean{|v|^p}{\I_L^{d-1}\times \{0\}}
+\mean{|v|^p}{\I_L^{d-1}\times \{N\}}\right]
\left\|1\right\|^p_{L^p(\I_L^{d-1}\times\{N\})}\leq  \frac{4^p}{N^p}\left\|v\right\|^p_{L^p(\I_L^{d-1}\times\{0,N\})}.
\end{split}\label{p12}
\end{align}
Furthermore, the fact that
\begin{align}\begin{split}
&\forall\, x\in\Z^{d-1}\times \{0\} \colon (\disDiff{+}{i}v_1)(x)=(\disDiff{+}{i}v)(x),\\
&\forall\, x\in\Z^{d-1}\times \{N\} \colon (\disDiff{+}{i}v_2)(x)=(\disDiff{+}{i}v)(x),\quad\text{and}\quad
\mean{v_1}{\I_L^{d-1}\times \{0\}}=\mean{v_2}{\I_L^{d-1}\times \{N\}}=0\end{split}
\end{align}
and
\cref{p07} imply that there exist $w_1,w_2\in \mathbbm{S}_{d,L,N}$ such that 
\begin{align}
\forall\, x\in \Z^{d-1}\times  \{0,N\}\colon\quad  w_1(x)=v_1(x)\quad\text{and}\quad w_2(x)=v_2(x)\label{p14}
\end{align}
and such that for all $j\in \{1,2\}$ it holds that
\begin{align}
\begin{split}
&\|\nabla w_j\|_{L^p(\efull{d,L,N})}\\
&\leq \frac{C_2(d,p)}{1-C_1(d,p)}
\left[\sum_{i=1}^{d-1}
\left\|\disDiff{+}{i} v_j\right\|_{L^p(\I_L^{d-1}\times\{0,N\}) }
\right]\leq 
\frac{C_2(d,p)}{1-C_1(d,p)}
\left[\sum_{i=1}^{d-1}
\left\|\disDiff{+}{i} v\right\|_{L^p(\I_L^{d-1}\times\{0,N\}) }
\right].
\end{split}\label{p13}\end{align}
Now, let $w\colon \Z^{d-1}\times(\disint{0}{N})\to\R$ be the function which satisfies that
\begin{align}
\forall\, x\in\Z^{d-1}\times(\disint{0}{N}) \colon \quad w(x)=w_1(x)+w_2(x)+w_3(x). \label{p16}
\end{align}
Then \eqref{p11a} and \eqref{p14} imply for all $x\in \Z^{d-1}\times\{0,N\}$  
 that 
$w(x)=v(x)$. Next, \eqref{p16}, the triangle inequality, \eqref{p12}, and \eqref{p13} imply that
\begin{align}
\begin{split}
&\|\nabla w\|_{L^p(\efull{d,L,N})}\\
&\leq \sum_{j=1}^{3}\|\nabla w_j\|_{L^p(\efull{d,L,N})}\leq 
\frac{2C_2(d,p)}{1-C_1(d,p)}
\left[\sum_{i=1}^{d-1}
\left\|\disDiff{+}{i} v\right\|_{L^p(\I_L^{d-1}\times\{0,N\}) }
\right]+\frac{4}{N}\left\|v\right\|^p_{L^p(\I_L^{d-1}\times\{0,N\})}.
\end{split}\end{align}
The proof of \cref{p15} is thus completed.
\end{proof}
\cref{p15} shows the existence of solutions to Dirichlet problems. Combining this with the uniqueness, which easily follows, e.g., from the maximum principle, we obtain 
\cref{s01b} below.
\begin{corollary}\label{s01b}
Let $d\in [2,\infty)\cap\N$, $p\in (1,\infty)$,
$v\colon \Z^{d-1}\times\{0,N\}\to\R$ satisfy 
for all $x\in \Z^{d-1}\times \{0,N\}$, $  i\in \disint{1}{d-1}$ that
$v(x)=v\left(x+2L\unit{d}{i}\right)$. Then  there exists uniquely $w \colon
\Z^{d-1}\times (\disint{0}{N}) \to\R$ such that
\begin{enumerate}[i)]
\item it holds for all $x\in \Z^{d-1}\times\{0,N\}$  that 
$w(x)=v(x)$ and
\item it holds for all $x\in \Z^{d-1}\times(\disint{1}{N-1})$  that 
$(\Laplace w)(x)=v(x)$.
\end{enumerate}
\end{corollary}
The existence and uniqueness, stated in \cref{s01b}, and \cref{p15} imply \cref{s01} below.
\begin{corollary}\label{s01}
For $L\in\N $ let $\I_L$  be the discrete interval given by $\I_L=\disint{-L+1}{L} $ and
let $\efull{d,L,N}\subseteq E_d$ be the set of edges
given by
\begin{align}
\efull{d,L,N}=\left\{(x,y)\in E_d\colon 
\tfrac12(x+y)\in \left([0,L]^{d-1}\times[0,N]\right)\setminus
\left([1,L-1]^{d-1}\times[1,N-1]\right)\right\}.
\end{align}
Then there exists
$C\colon([2,\infty)\cap\N)\times(1,\infty)\to(0,\infty) $
such that
for all $d\in [2,\infty)\cap\N$, $p\in (1,\infty)$, $N,L\in\N$, and for all functions
$w\colon \Z^{d-1}\times (\disint{0}{N})\to\R$ with the property that
\begin{align}\begin{split}
\forall\, x\in \Z^{d-1}\times (\disint{0}{N}), \,
  i\in \disint{1}{d-1}\colon\quad  w(x)=w(x+2L\unit{d}{i}), \\
1/4\leq L/N\leq 4,\quad\text{and}\quad 
\forall\, x\in \Z^{d-1}\times(\disint{1}{N-1})\colon\quad  (\Laplace w)(x)=0 \end{split}
\end{align}
it holds that
\begin{align}
\|\nabla w\|_{L^p(\efull{d,L,N})}\leq  C(d,p)
\left[\sum_{i=1}^{d-1}
\left\|\disDiff{+}{i} w\right\|_{L^p(\I_L^{d-1}\times\{0,N\}) }
+\frac{1}{N}\left\|w\right\|_{L^p(\I_L^{d-1}\times\{0,N\}) }
\right].
\end{align}
\end{corollary}
\subsection{Proof of the main result in the Dirichlet case}\label{e02}
\begin{lemma}\label{r15b}
Let $N\in \N$, $f\colon \Z\to \R $ satisfy  that $\forall\, x\in\Z\colon f(x)=f(x+2N)$ and $\forall\, x\in \disint{-N+1}{N}\colon f(x)=-f(-x)$. Then $f(0)=f(N)=0$.
\end{lemma}
\begin{proof}[Proof of \cref{r15b}]The fact $f(0)=-f(-0)=-f(0)$ proves that $f(0)=0$. Next, the fact that
$f(N)=f(-N+2N)=f(-N)=-f(N)$ proves that $f(N)=0$. This completes the proof of \cref{r15b}.
\end{proof}
\begin{lemma}[Odd reflections]\label{r15}Let $d\in [2,\infty)\cap\N$,
$w\colon \Z^{d-1}\times (\disint{0}{N})\to\R$, $j\in \disint{1}{d-1}$ satisfy that
\begin{align}
&\forall\, x\in \Z^{d-1}\times (\disint{0}{N}),\,  i\in \disint{1}{d-1}\colon\quad w(x)=w(x+2L\unit{d}{i}),\label{s06a}\\
&\forall\, x\in \Z^{d-1}\times(\disint{1}{N-1})\colon \quad (\Laplace w)(x)=0,\label{s06b}\\
&\forall\, x\in \Z^{j-1}\times (\disint{-N+1}{N})\times \Z^{d-j-1}\times\{0,N\} \colon\quad w(x)=-w(x-2x_j\unit{d}{j})\label{s06c}.
\end{align}
Then it holds that 
\begin{align}
\forall\, x\in \Z^{j-1}\times \{0,N\}\times \Z^{d-j-1}\times (\disint{0}{N})\colon \quad w(x)=0.
\end{align}
\end{lemma}
\begin{proof}[Proof of \cref{r15}]Let $\tilde{w}\colon \Z^{d-1}\times (\disint{0}{N})\to\R $ be the function which satisfies for all
$ x\in \Z^{d-1}\times(\disint{0}{N}) $ that $ \tilde{w}(x)=-w(x-2x_j\unit{d}{j}).$ 
Then \eqref{s06b} implies that 
$\forall\, x\in \Z^{d-1}\times(\disint{1}{N-1})\colon (\Laplace \tilde{w})(x)=0$
and \eqref{s06a} implies that $\forall\, x\in \Z^{d-1}\times (\disint{0}{N}),\,  i\in \disint{1}{d-1}\colon\tilde{w}(x)=\tilde{w}(x+2L\unit{d}{i})$.
This and \eqref{s06c} yield that
$\forall\, x\in \Z^{d-1}\times\{0,N\} \colon w(x)=\tilde{w}(x)$.
\cref{s01b} hence shows 
for all $x\in \Z^{d-1}\times (\disint{0}{N})$
that  $w(x)= \tilde{w}(x)$, i.e., 
$w(x)= -w(x-2x_j\unit{d}{j})$. This and \cref{r15b} (with $f\defeq (\Z\ni\xi\mapsto w(x_1,\ldots,x_{j-1},\xi,x_{j+1},\ldots,x_d)\in\R)$, i.e., applied 
to the $j$-th coordinate) complete the proof of \cref{r15}.
\end{proof}

The sets $\efull{d,N}$, $\etan{d,N}$, $\etanT{d,N}$ in 
\cref{r20c} below are illustrated in \cref{f01}:
$\efull{d,N}$ consists of all red and blue edges; $\etan{d,N}$ consists of all red edges; $\etanT{d,N}$ consists of all red points.
Furthermore, this setting provides two ingredients that we need for the next step:
\cref{s01} (see \eqref{s02}) and
the  Poincar\'e inequality  (see \eqref{s05}).

\begin{setting}\label{r20c}
Let \cref{v02} be given.
For every $d,N\in [2,\infty)\cap\N$ 
let $\efull{d,N}, \etan{d,N}\subseteq E_d$, $\I_N\subseteq \Z$, $\etanT{d,N}\subseteq \Z^d$ be the sets given by
\begin{align}
 \I_N&=\disint{-N+1}{N} ,\quad 
\etanT{d,N}=\Z^d \cap ([0,N]^{d})\setminus ((0,N)^d).
\\
\efull{d,N}&=\left\{(x,y)\in E_d\colon 
\tfrac12(x+y)\in ([0,N]^d)\setminus
([1,N-1]^{d})\right\},
\\
\etan{d,N}&=\left\{(x,y)\in E_d\colon x,y\in 
([0,N]^d)\setminus((0,N)^d)\right\}.
\end{align}
Let $c\colon([2,\infty)\cap\N)\times(1,\infty)\to(0,\infty) $
be a function whose existence is ensured by \cref{s01} and which satisfies that
for all $d,N\in [2,\infty)\cap\N$, $p\in (1,\infty)$,
$w\colon  \Z^{d-1}\times( \disint{0}{N})\to\R$ with
\begin{align}\begin{split}
&\forall\, i\in\disint{1}{d-1},\,x\in \Z^{d-1}\times( \disint{0}{N})
\colon \quad
  w(x+2N\unit{d}{i}) =w(x)\quad\text{and}\quad\\ 
&\forall\, i\in\disint{1}{d-1},\,x\in \Z^{d-1}\times (\disint{1}{N-1})
\colon   \quad (\Laplace w)(x)=0\end{split}
\end{align}
it holds
that
\begin{align}
\|\nabla w\|_{L^p(\efull{d,N})}\leq  c(d,p)
\left[\sum_{i=1}^{d-1}
\left\|\disDiff{+}{i} w\right\|_{L^p(\I_N^{d-1}\times\{0,N\}) }
+\frac{1}{N}\left\|w\right\|_{L^p(\I_N^{d-1}\times\{0,N\}) }
\right].\label{s02}
\end{align}
Let $\constPI\colon (1,\infty)\to[0,\infty)$ be a function which satisfies the Poincar\'e inequality, i.e., it holds for all $p\in (1,\infty)$,
$d,N\in \N\cap[2,\infty)$, $u\in \etanT{d,N}$ 
that
\begin{align}\label{s05}
\tfrac{1}{N}\inf_{a\in\R}\|u-a\|_{L ^p(\etanT{d,N})}\leq c_{\mathrm{PI}}(p)
\|\nabla u \|_{L^p(\etan{N})}.
\end{align}
Let $\mathbbm{Q}_{d,N}$ be the set of all functions
$u\colon (\disint{0}{N})^{d}\to\R$ which satisfy for all $x\in (\disint{1}{N-1})^d$ that $(\Laplace u)(x) =0$.
\end{setting}

\begin{lemma}\label{r20b}Assume \cref{r20c} and let
$d,N\in [2,\infty)\cap\N$, $p\in (1,\infty)$,  $u\in \mathbbm{Q}_{d,N}$ be fixed. Then \begin{align}
\|\nabla u \|_{L^p(\efull{d,N})}\leq 
d(d3^d+2)^{d+1}c(d,p)
\left[\|\nabla u\|_{L^p( \etan{N})}+\frac{1}{N}\|u\|_{L^p(Q_N)}\right].\label{s04}
\end{align}
\end{lemma}

\begin{proof}[Proof of \cref{r20b}]
First, we will successively construct functions $w_i:\mathbbm{Z}^{i-1}\times(\disint{0}{N})\times\mathbbm{Z}^{d-i}\to\mathbbm{R}$,
$i\in\disint{1}{d}$, 
with the property $\mathcal{A}(w_1,\ldots,w_d)$ defined as follows:
For every
$\ell\in \disint{1}{d}$,
$w_i:\mathbbm{Z}^{i-1}\times(\disint{0}{N})\times\mathbbm{Z}^{d-i}\to\mathbbm{R}$,
$i\in \disint{1}{\ell}$ let
$\mathcal{A}(w_1,\ldots,w_\ell)$ be
the statement that
\begin{enumerate}[i)]
\item for all $i\in \disint{1}{\ell}$, $j\in(\disint{1}{d})\setminus \{i\}$,
$x\in \mathbbm{Z}^{i-1}\times(\disint{0}{N})\times\mathbbm{Z}^{d-i}$
it holds that
$
w_i(x)=w_i(x+2N\mathbf{e}_j)$,
\item for all
$i\in \disint{1}{\ell}$, $x\in \mathbbm{Z}^{i-1}\times(\disint{1}{N-1})\times\mathbbm{Z}^{d-i}$ it holds that $(\Laplace w_i)(x)=0$,
\item  for all
$i\in \disint{1}{\ell}$
$x\in (\disint{0}{N})^{i-1}\times\{0,N\}\times(\disint{0}{N})^{d-i}$ it holds that
\begin{align}
u(x)=\sum_{\nu=1}^{i}w_\nu(x) ,
\end{align}
\item  for all
$i\in \disint{1}{\ell}$, $j\in \disint{1}{i-1}$, $x\in \left(\disint{0}{N}\right)^{j-1}\times\{0,N\}\times(\disint{0}{N})^{d-j}$ it holds that
$w_i(x)=0$,
and \item 
for all
$i\in \disint{1}{\ell}$ it holds that
\begin{align}
\|\nabla w_i\|_{L^p(\efull{d,N})}\leq  c(d,p)(d3^d+2)^i
\left[\sum_{i=1}^{d-1}
\left\|\nabla u\right\|_{L^p(\etan{d,N}) }
+\frac{1}{N}\left\|u\right\|_{L^p(\etanT{d,N} )}
\right].\label{r18}
\end{align}
\end{enumerate}
As a first step, let $v_1\colon \{0,N\}\times\Z^{d-1}\to\R$ be the function given by
\begin{align}
&\forall\, x\in \{0,N\}\times(\disint{0}{N})^{d-1}\colon\quad v_1(x)=u(x),\label{r05}\\
&\forall\, j\in \disint{2}{d},\,x\in \{0,N\}\times(\disint{-(N-1)}{N})^{d-1}\colon\quad v_1(x)=v_1(x-2x_j\unit{d}{j}),\label{r06}\\
&\forall\, j\in \disint{2}{d},\, x\in \{0,N\}\times\Z^{d-1}\colon\quad v(x)=v(x+2N\unit{d}{j})\label{r07}
\end{align}
and let 
$w_1\colon (\disint{0}{N})\times\mathbbm{Z}^{d-1}\to\mathbbm{R}$ be the function 
(cf. \cref{s01b})
which satisfies that
\begin{align}\begin{split}\label{r08}
&\forall\, x\in \{0,N\}\times \Z^{d-1} \colon\quad w_1(x)=v_1(x),\\
&\forall\, x\in (\disint{1}{N-1})\times \Z^{d-1} \colon\quad (\Laplace w_1)(x)=0.
\end{split}
\end{align}
Then \eqref{s02}, \eqref{r05}, and \eqref{r06} imply that
\begin{align}\begin{split}
\|\nabla w_1\|_{L^p(\efull{N})}
&\leq  c(d,p)
\left[\sum_{i=1}^{d-1}
\left\|\disDiff{+}{i} v_1\right\|_{L^p(\{0,N\}\times\I_N^{d-1}) }
+\frac{1}{N}\left\|v_1\right\|_{L^p(\{0,N\}\times\I_N^{d-1}) }
\right]\\
&\leq 
  c(d,p)d3^d
\left[\sum_{i=1}^{d-1}
\left\|\nabla u\right\|_{L^p(\etan{d,N}) }
+\frac{1}{N}\left\|u\right\|_{L^p(\etanT{d,N})}
\right].\end{split}\label{r09}
\end{align}
Combining \eqref{r05}, \eqref{r07}, \eqref{r08}, and  \eqref{r09} yields that $\mathcal{A}(w_1)$ is true.
For the recursive step
let $\ell\in \disint{1}{d}$ and suppose 
that we have constructed $w_1,\ldots,w_\ell$ so that
$\mathcal{A}(w_1,\ldots,w_\ell)$
 holds. Now,
let 
$
v_{\ell+1}\colon \Z^{\ell}\times\{0,N\}\times\Z^{d-\ell-1}\to\R
$
be the function which satisfies that
\begin{enumerate}[i)] 
\item 
for all $x\in (\disint{0}{N})^{\ell}\times\{0,N\}\times(\disint{0}{N})^{d-\ell-1}$
it holds that
\begin{align}
v_{\ell+1}(x)=u(x)-\sum_{\nu=1}^{\ell}w_\nu(x),\label{r10} 
\end{align}
\item for all
$j\in \disint{1}{\ell},\,x\in (\disint{-N+1}{N})^{\ell}\times \{0,N\}\times(\disint{-N+1}{N})^{d-\ell-1}$ it holds that
\begin{align}
 v_{\ell+1}(x)=-v_{\ell+1}(x-2x_j\unit{d}{j}),\label{r11}
\end{align}
\item for all
$ j\in \disint{\ell+2}{d}$, $x\in (\disint{-N+1}{N})^{\ell}\times \{0,N\}\times(\disint{-N+1}{N})^{d-\ell-1}$ it holds that
\begin{align}
 v_{\ell+1}(x)=v_{\ell+1}(x-2x_j\unit{d}{j}),
\end{align}
and
\item for all 
$j\in(\disint{1}{d})\setminus \{\ell+1\}$,
$x\in \mathbbm{Z}^{\ell}\times\{0,N\}\times\mathbbm{Z}^{d-\ell-1}$
it holds that
\begin{align}
v_{\ell+1}(x)=v_{\ell+1}(x+2N\mathbf{e}_j),
\end{align}
\end{enumerate}
and let $w_{\ell+1}\colon
\Z^{\ell}\times  (\disint{0}{N})\times\mathbbm{Z}^{d-\ell-1}\to\mathbbm{R}$ be the function 
(cf. \cref{s01}) given by
\begin{align}
&\forall\, x\in \Z^{\ell}\times \{0,N\}\times \Z^{d-\ell-1} \colon \quad w_{\ell+1}(x)=v_{\ell+1}(x),\label{r13}
\\
\label{r12}
&\forall\, x\in \Z^{\ell}\times (\disint{1}{N-1})\times \Z^{d-\ell-1}\colon\quad 
(\Laplace w_{\ell+1})(x)=0,\\
&
\begin{aligned}
\forall\, j\in(\disint{1}{d})\!\setminus\! \{\ell+1\},\,x\in \Z^{\ell}\times (\disint{0}{N})\times \Z^{d-\ell-1}\colon\quad 
w_{\ell+1}(x+2N\unit{d}{j})=w_{\ell+1}(x).
\end{aligned}\label{r14}
\end{align}
Then \eqref{s02} implies that
\begin{align}
\|\nabla w_{\ell+1}\|_{L^p(\efull{N})}
&\leq  c(d,p)
\left[\sum_{\nu=1}^{d-1}\left\|\disDiff{+}{\nu} v_{\ell+1}\right\|_{L^p(\I_N^{\ell}\times\{0,N\}\times\I_N^{d-\ell-1}) }
+\frac{1}{N}\|v_{\ell+1}\|_{L^p(\I_N^{\ell}\times\{0,N\}\times\I_N^{d-\ell-1}) }
\right].
\end{align}
Note that \eqref{r10} and \eqref{r13} imply that
\begin{align}
\forall\, 
x\in (\disint{0}{N})^{\ell}\times\{0,N\}\times(\disint{0}{N})^{d-\ell-1}\colon \quad 
w_{\ell+1}(x)=u(x)-\sum_{\nu=1}^{\ell}w_\nu(x).\label{r10b}
\end{align}
To lighten the notation let
$U^{\ell+1}_{d,N}, V^{\ell+1}_{d,N}\subseteq \Z^d$ be the sets given by
\begin{align}
U^{\ell+1}_{d,N}=\I_N^{\ell}\times\{0,N\}\times\I_N^{d-\ell-1}\quad\text{and}\quad
V^{\ell+1}_{d,N} =(\disint{0}{N})^{\ell}\times\{0,N\}\times(\disint{0}{N})^{d-\ell-1}.
\end{align}
Then \eqref{r10b} and the triangle inequality show that
\begin{align}\begin{split}
&\left(\sum_{\nu=1}^{d-1}\left\|\disDiff{+}{\nu} v_{\ell+1}\right\|_{L^p(U^{\ell+1}_{d,N}) }\right)
+\frac{1}{N}\left\|v_{\ell+1}\right\|_{L^p(U^{\ell+1}_{d,N}) }\\
&\leq 
3^d\left[
\left(\sum_{\nu=1}^{d-1}\left\|\disDiff{+}{\nu} v_{\ell+1}\right\|_{L^p(V^{\ell+1}_{d,N}) }\right)
+\frac{1}{N}\|v_{\ell+1}\|_{L^p(V^{\ell+1}_{d,N}) }\right]
\\
&\leq 
3^d
\left[\sum_{\mu=1}^{\ell}\left(\sum_{\nu=1}^{d-1}\left\|\disDiff{+}{\nu} w_{\mu}\right\|_{L^p(V^{\ell+1}_{d,N}) }+\frac{1}{N}\left\|w_{\mu}\right\|_{L^p(V^{\ell+1}_{d,N}) }\right)
+\sum_{\nu=1}^{d-1}\left\|\disDiff{+}{\nu} u\right\|_{L^p(V^{\ell+1}_{d,N}) }
+\frac{1}{N}\left\|u\right\|_{L^p(V^{\ell+1}_{d,N}) }\right]\\
&\leq d3^d \left[\sum_{\mu=1}^{\ell}\left[\left\|\nabla w_\mu\right\|_{L^p( \etan{N})}+\tfrac{1}{N}\|w_\mu\|_{L^p(\etanT{d,N})}\right]
+\|\nabla u\|_{L^p( \etan{N})}+\tfrac{1}{N}\|u\|_{L^p(\etanT{d,N})}\right]\\
&\leq d3^d\left[\|\nabla u\|_{L^p( \etan{N})}+\tfrac{1}{N}\|u\|_{L^p(\etanT{d,N})}\right]
\left[
\sum_{\mu=0}^{\ell}  (d3^d+2)^i
\right]\\&\leq (d3^d+2)^{\ell+1}
\left[\|\nabla u\|_{L^p( \etan{N})}+\tfrac{1}{N}\|u\|_{L^p(\etanT{d,N})}\right].
\end{split}\label{r16}
\end{align}
Furthermore,
\eqref{r11}, \eqref{r12}, and uniqueness of the Dirichlet problem
 show that 
\begin{align}
\forall\, 
j\in \disint{1}{\ell},\,
x\in \Z^{\ell}\times (\disint{0}{N})\times \Z^{d-\ell-1} \colon\quad 
w_{\ell+1}(x)=-w_{\ell+1}(x-2x_j\unit{d}{j}).
\end{align}
This and \cref{r15} (applied to the $j$-th coordinate for $j\in\disint{1}{\ell}$) show that
\begin{align}
\forall\, j\in \disint{1}{\ell},\, x\in \left(\disint{0}{N}\right)^{j-1}\times\{0,N\}\times(\disint{0}{N})^{d-j}\colon \quad w_{\ell+1}(x)=0.
\end{align}
Combining this, \eqref{r14}, \eqref{r12}, \eqref{r10b}, \eqref{r16}, and the fact that 
$w_1,\ldots,w_\ell$ were constructed with the property
 $\mathcal{A}(w_1,\ldots,w_\ell)$ 
yields that  $\mathcal{A}(w_1,\ldots,w_{\ell+1})$ is true. Thus we have iteratively constructed $w_1,\ldots,w_d$ with $\mathcal{A}(w_1,\ldots,w_{d})$. This implies for all $x\in \etanT{d,N}$  that $u(x)=\sum_{i=1}^{d}w_i(x)$
and for all
$ x\in (\disint{1}{N})^{d-1}$, $i\in\disint{1}{d}$ that $ (\Laplace w_i)(x)=0$.
The fact that
$\forall\, x\in (\disint{1}{N})^{d-1}\colon (\Laplace u)(x)=0$ hence implies for all
$ x\in (\disint{1}{N})^{d-1}$ that $u(x)=\sum_{i=1}^dw_d(x)$. This, the triangle inequality, and \eqref{r18}
 demonstrate that
\begin{align}
\|\nabla u\|_{L^p(\efull{d,N})}
\leq \sum_{i=1}^{d}
\|\nabla w_i\|_{L^p(\efull{d,N})}\leq 
d(d3^d+2)^{d+1}c(d,p)\left[
\left\|\nabla u\right\|_{L^p(\etan{d,N}) }
+\frac{1}{N}\left\|u\right\|_{L^p(\etanT{d,N})}
\right].\label{r29}
\end{align}
This completes the proof of \cref{r20b}.
\end{proof}
\begin{corollary}\label{r20}
Assume \cref{r20c} and let
$d,N\in [2,\infty)\cap\N$, $p\in (1,\infty)$,  $u\in \mathbbm{Q}_{d,N}$ be fixed. Then it holds that 
$
\|\nabla u \|_{L^p(\efull{d,N})}\leq d(d3^d+2)^{d+1}c(d,p)(1+c_{\mathrm{PI}}(p))\|\nabla u\|_{L^p(\etan{d,N})}.
$
\end{corollary}
\begin{proof}[Proof of \cref{r20}]
\cref{r20b} (with $u\defeq u-a$ for $a\in\R$)
and \eqref{s05} show that
\begin{align}\begin{split}
&\|\nabla u\|_{L^p(\efull{d,N})}= \inf_{a\in\R}
\|\nabla (u-a)\|_{L^p(\efull{d,N})}\\&\leq 
d(d3^d+2)^{d+1}c(d,p)\left[\inf_{a\in\R}\left(
\left\|\nabla (u-a)\right\|_{L^p(\etan{N}) }
+\frac{1}{N}\left\|u-a\right\|_{L^p(\etanT{d,N} )}
\right)\right]\\
&= d(d3^d+2)^{d+1}c(d,p)\left[\left\|\nabla u\right\|_{L^p(\etan{N}) }
+\frac{1}{N}\inf_{a\in\R}\left\|u-a\right\|_{L^p(\etanT{d,N} )}\right]
\\
&\leq d(d3^d+2)^{d+1}c(d,p)(1+c_{\mathrm{PI}}(p))\left\|\nabla u\right\|_{L^p(\etan{N}) }.
\end{split}
\end{align}
This completes the proof of \cref{r20}.
\end{proof}
\subsection{Construction of the Neumann extensions}\label{e03}
In the Neumann case we also use a telescope sequence. First of all, instead of  \cref{p05b} we start with  
\cref{q06a} below with Neumann conditions on the right hand sides of
\eqref{q02} and \eqref{q09}.
\begin{setting}\label{q06a}
Assume \cref{p05} and let $C_1,C_2\colon([2,\infty)\cap\N)\times(1,\infty)\to[0,\infty]$
 be the functions which satisfy 
for all $d\in [2,\infty)\cap\N$, $p\in (1,\infty)$
that 
$C_1(d,p)$, $C_2(d,p)$ are the smallest real extended numbers such that for all
$N,L\in\N$, $u\in \mathbbm{H}_{d,L,\geq 0}$ with $1/4\leq N/L$ and 
$N\geq 2 $
it holds 
 that
\begin{align}\label{q02}
\left\|\disDiff{-}{d}u\right\|_{L^p(\I_L^{d-1}\times\{N\})}\leq C_1(d,p)\left\|\disDiff{+}{d}u\right\|_{L^p(\I_L^{d-1}\times\{0\})}
\end{align}
and
\begin{align}
\|\nabla u\|_{L^p(\efull{d,L,N})}\leq  C_2(d,p)
\left\|\disDiff{+}{d} u\right\|_{L^p(\I_L^{d-1}\times\{0\}) }
.\label{q09}
\end{align}
\end{setting}
\begin{lemma}\label{q25}
Assume \cref{q06a} and let $d\in [2,\infty)\cap\N$, $p\in (1,\infty)$ be fixed. Then it holds that $C_1(d,p)<1$ and $C_2(d,p)<\infty$.
\end{lemma}
\begin{proof}[Heuristic proof of \cref{q25}]
First, note that the discrete derivatives of a harmonic function are still harmonic. 
Using \cref{t01a} in \cref{t01} (with the function replaced by the derivative)
 we bound the normal differences on the top $\Z^{d-1}\times\{N\}$ by the normal differences on the bottom
$\Z^{d-1}\times\{0\}$ (see \eqref{r01}). Next, using \cref{x13} (with the function replaced by the derivative) and \cref{t01c} in \cref{t01} we bound the tangential differences on the top and bottom by the normal differences on the bottom (see \eqref{r04}).
Furthermore, using \cref{p01} we 
 bound the differences with respect to
all edges with one endpoint on the face $\{x_1=0\}$ by the tangential differences on the bottom
(see \eqref{r02}) and hence again by the normal differences on the bottom.
A permutation of the coordinates then shows that we can bound the differences with respect to all  edges with one endpoints on the boundary by the normal differences on the bottom. 
\end{proof}
\begin{proof}[Rigorous proof of \cref{q25}]
\cref{t01} implies that
 there exists
$c_1\colon (0,\infty)\to (0,1)$ such that
for all $N,L\in\N$,
$u\in \mathbbm{H}_{d,L,\geq 0}$ with $N/L\geq \rdown$ and $\mean{u}{ \I_L^{d-1}\times\{0\}}=0$ it holds that
\begin{align}
\left\|u\right\|_{L^p(\I_L^{d-1}\times\{N\})}\leq c_1(\rdown)\left\|u\right\|_{L^p(\I_L^{d-1}\times\{0\})}.
\end{align}
This (with 
$N\defeq N-1$, 
$\rdown \defeq 1/4$, and $u\defeq \disDiff{+}{d}u$  for $N\in [ 2,\infty)\cap\N$, 
$L\in\N$,
$u\in \mathbbm{H}_{d,L,\geq 0}$) 
and the fact that $\forall\, N\in\N\cap[2,\infty)\colon (N-1)\geq N/2$
show that for all 
$N,L\in\N$,
$u\in \mathbbm{H}_{d,L,\geq 0}$ with $N/L\geq 1/4$ and $\mean{u}{ \I_L^{d-1}\times\{0\}}=0$ it holds that $(N-1)/L\geq N/(2L)\geq 1/8$ and
\begin{align}
\left\|\disDiff{-}{d}u\right\|_{L^p(\I_L^{d-1}\times\{N\})}=
\left\|\disDiff{+}{d}u\right\|_{L^p(\I_L^{d-1}\times\{N-1\})}\leq c_1(1/8)\left\|\disDiff{+}{d}u\right\|_{L^p(\I_L^{d-1}\times\{0\})}.\label{r01}
\end{align}
Furthermore, \cref{p01} shows that there exists 
$c_2\colon (0,\infty)\to (0,\infty)$ such that 
for all 
$\rup\in (0,\infty)$, $L\in\N$, $N\in (0,L\rup]\cap\N$,
$u\in \mathbbm{H}_{d,L,\geq 0}$ it holds 
that
$\|u\|_{L^p(\{0\}\times \I_L^{d-2}\times (\disint{1}{N}))}\leq  c_2(\rup)
\|u\|_{L^p(\I_L^{d-1}\times\{0\})}.$
This (with 
 $\rup \defeq4$, 
$u\defeq \disDiff{\pm}{i}u $, and
$u\defeq \disDiff{+}{d}u $
 for $i\in\disint{1}{d-1}$,
$L,N\in\N$ with $N/L\leq 1/4$) shows 
for all 
$L,N\in\N$ with $N/L\leq 4$,
$u\in \mathbbm{H}_{d,L,\geq 0}$, $i\in\disint{1}{d-1}$ that
\begin{align}\label{r02}
\left\|\disDiff{\pm }{i}u\right\|_{L^p(\{0\}\times \I_L^{d-2}\times (\disint{1}{N}))}\leq  c_2(4)
\left\|\disDiff{\pm }{i}u\right\|_{L^p(\I_L^{d-1}\times\{0\})}
\end{align}
and 
\begin{align}\label{r03}
\left\|\disDiff{+}{d}u\right\|_{L^p(\{0\}\times \I_L^{d-2}\times (\disint{1}{N}))}\leq  c_2(4)
\left\|\disDiff{+ }{d}u\right\|_{L^p(\I_L^{d-1}\times\{0\})}.
\end{align}
\cref{x13} (with 
$u\defeq \disDiff{\pm}{i}u $,
for $i\in\disint{1}{d-1}$,
$u\in \mathbbm{H}_{d,L,\geq 0}$) and
\cref{t01} show that there exists
$c_3\in  (0,\infty)$ such that
for all 
$u\in \mathbbm{H}_{d,L,\geq 0}$, $i\in\disint{1}{d-1}$ it holds
that
\begin{align}
\left\|\disDiff{\pm}{i} u\right\|_{L^p(\I_L^{d-1}\times\{N\}) }\leq 
\left\|\disDiff{\pm}{i} u\right\|_{L^p(\I_L^{d-1}\times\{0\}) }\leq c_3
\left\|\disDiff{+}{d} u\right\|_{L^p(\I_L^{d-1}\times\{0\}) }.\label{r04}
\end{align}
Combining \eqref{r01}--\eqref{r04} we obtain that
 there exists $c_4\in (0,\infty)$ 
such that 
for all $d\in [2,\infty)\cap \N$, $p\in (1,\infty)$, 
$L,N\in\N$ with $N/L\leq 4$,
$u\in \mathbbm{H}_{d,L,\geq 0}$ it holds that
\begin{align}
\|\nabla u\|_{L^p(\efull{d,L,N})}\leq  c_4
\left\|\disDiff{+}{d} u\right\|_{L^p(\I_L^{d-1}\times\{0\}) }.
\end{align}
This shows  that
$C_2(d,p)<\infty$.
The proof of \cref{q25} is thus completed.
\end{proof}
 \cref{q06} below introduces a telescope sequence which is similar to that in the  Dirichlet case (cf. \cref{p06}).
In \eqref{q03a}
the means on each layer are set to be zero, since otherwise
 the Neumann problems on the haft spaces do not determine unique solutions.
\begin{setting}[Telescope sequence for the Neumann case]\label{q06}
Assume \cref{p05}, let  $N,L\in\N $, $d\in [2,\infty)\cap \N$, $p\in (1,\infty)$ be fixed and satisfy that $1/4\leq L/N\leq 4$,  let 
$v\in\mathbbm{B}_{d,L,N}$ satisfy that 
\begin{align}
\sum_{x\in \I_L^{d-1}\times\{0\}}v(x)=0\quad\text{and}\quad
\forall\, x\in\Z^{d-1}\times\{N\}\colon\quad  v(x)=0,\label{q03e}
\end{align}
and let
\begin{align}
(u_{2k+1})_{k\in \N_0} \subseteq \mathbbm{H}_{d,L,\geq 0}\quad\text{and}\quad
(u_{2k+2})_{k\in \N_0}\subseteq  \mathbbm{H}_{d,L,\leq N}\label{r30}
\end{align} 
be the sequences which satisfy that
\begin{align}\begin{split}
&\forall\, x\in\Z^{d-1}\times\{0\}\colon\quad  (\disDiff{+}{i}u_1)(x)=v(x),\quad 
\label{q03a}
\\
& \forall\, k\in\N_0,\, x\in \Z^{d-1}\times \{N\}\colon \quad (\disDiff{-}{d}u_{2k+2})(x)=
(\disDiff{-}{d} u_{2k+1})(x),\\
&\forall\,
k\in\mathbbm{N},\, x\in \Z^{d-1}\times\{0\}\colon \quad 
(\disDiff{+}{d}
u_{2k+1})(x)= (\disDiff{+}{d}u_{2k})(x),\quad\text{and}\\
&\forall\, n\in \N,\, y\in \disint{0}{N}\colon\quad  \mean{u_n}{\I_{L}^{d-1}\times\{y\}}=0.
\end{split}\end{align}
\end{setting}
\begin{lemma}\label{q03}
Assume \cref{q06}.
Then it holds for all $n \in\N$ that
\begin{align}
\max\left\{
\|\disDiff{+}{d}u_{n}\|_{L^p(\I_{L}^{d-1}\times\{0\})},
\|\disDiff{-}{d}u_{n}\|_{L^p(\I_{L}^{d-1}\times\{N\})}\right\}
\leq C_1(d,p)^{n-1}\|\disDiff{+}{d}u_{1}\|_{L^p(\I_{L}^{d-1}\times\{0\})}
\end{align}
\end{lemma}
\begin{proof}[Proof of \cref{p03}]
The assumption that $(u_{2k+1})_{k\in \N_0} \subseteq \mathbbm{H}_{d,L,\geq 0}$ in \eqref{r30} and 
\cref{x13} (with 
$u\defeq \disDiff{+}{d}u_{2k+3}$ for $k\in \N_0$) show $k\in\N_0$ that
\begin{align}
\left\|\disDiff{-}{d}u_{2k+1}\right\|_{L^p(\mathbbm{I}_L^{d-1}\times\{N\})}=
\left\|\disDiff{+}{d}u_{2k+1}\right\|_{L^p(\mathbbm{I}_L^{d-1}\times\{N-1\})}
\leq 
\left\|\disDiff{+}{d}u_{2k+1}\right\|_{L^p(\mathbbm{I}_L^{d-1}\times\{0\})}.\label{q05}
\end{align}
Similarly, the assumption that $(u_{2k+2})_{k\in \N_0}\subseteq  \mathbbm{H}_{d,L,\leq N}$ in \eqref{r30} and \cref{x13} (together with a simple change of coordinates) show for all $k\in\N_0$ that
\begin{align}
\left\|\disDiff{+}{d}u_{2k+2}\right\|_{L^p(\mathbbm{I}_L^{d-1}\times\{0\})} \leq 
\left\|\disDiff{-}{d}u_{2k+2}\right\|_{L^p(\mathbbm{I}_L^{d-1}\times\{N\})} .\label{q07}
\end{align}
Next, \eqref{r30}, \eqref{q02}, and possibly a simple change of coordinates show for all $k\in\N_0$ that
\begin{align}\begin{split}
\left\|\disDiff{-}{d}u_{2k+2}\right\|_{L^p(\mathbbm{I}_L^{d-1}\times\{N\})}
=\left\|\disDiff{-}{d}u_{2k+1}\right\|_{L^p(\mathbbm{I}_L^{d-1}\times\{N\})}
\leq  C_1(d,p)
\left\|\disDiff{+}{d}u_{2k+1}\right\|_{L^p(\mathbbm{I}_L^{d-1}\times\{0\})} 
 .\end{split}\label{q08}
\end{align}
and
\begin{align}
\left\|\disDiff{+}{d}u_{2k+3}\right\|_{L^p(\mathbbm{I}_L^{d-1}\times\{0\})} \leq 
\left\|\disDiff{+}{d}u_{2k+2}\right\|_{L^p(\mathbbm{I}_L^{d-1}\times\{0\})} 
\leq 
C_1(d,p)\left\|\disDiff{-}{d}u_{2k+2}\right\|_{L^p(\mathbbm{I}_L^{d-1}\times\{N\})}.\label{q10}
\end{align}
Combining \eqref{q05}--\eqref{q10}, an induction argument, and \eqref{q05} (with $k\defeq0$) proves that for all $n\in\N$ it holds that
\begin{align}\begin{split}
&\max\left\{
\|\disDiff{+}{d}u_{n+1}\|_{L^p(\I_{L}^{d-1}\times\{0\})},
\|\disDiff{-}{d}u_{n+1}\|_{L^p(\I_{L}^{d-1}\times\{N\})}\right\}\\
&\leq C_1(d,p)
\max\left\{
\|\disDiff{+}{d}u_{n}\|_{L^p(\I_{L}^{d-1}\times\{0\})},
\|\disDiff{-}{d}u_{n}\|_{L^p(\I_{L}^{d-1}\times\{N\})}\right\}\\
&\leq C_1(d,p)^{n-1}
\max\left\{
\|\disDiff{+}{d}u_{1}\|_{L^p(\I_{L}^{d-1}\times\{0\})},
\|\disDiff{-}{d}u_{1}\|_{L^p(\I_{L}^{d-1}\times\{N\})}\right\}\\
&\leq C_1(d,p)^{n-1}\|\disDiff{+}{d}u_{1}\|_{L^p(\I_{L}^{d-1}\times\{0\})}.
\end{split}\end{align}
  The proof of \cref{q03} is thus completed.
\end{proof}
\begin{lemma}\label{q13}
Assume \cref{q06}. Then it holds for all $x\in\Z^{d-1}\times(\disint{0}{N})$ that $\sum_{n=1}^\infty|u_n(x)|<\infty$.
\end{lemma}
\begin{proof}[Proof of \cref{q13}]
The triangle inequality, a telescope sum argument, Jensen's inequality, and \eqref{q09} show for all
$a\in\{0,N\}$, $x\in\I_L^{d-1}\times\{a\}$, $u\in\mathbbm{H}_{d,L,\geq 0}\cup\mathbbm{H}_{d,L,\leq N}$ with 
$\mean{u}{\I_L^{d-1}\times \{a\}}=0$ that
\begin{align}\begin{split}
|u(x)|&=\frac{1}{| \I_L^{d-1}|}\left| \sum_{y\in \I_L^{d-1}\times \{a\}}u(x)-u(y)\right|
\leq 
\frac{1}{| \I_L^{d-1}|} \sum_{y\in \I_L^{d-1}\times \{a\}}\left|u(x)-u(y)\right|\\
&\leq 
\frac{1}{| \I_L^{d-1}|} \sum_{y\in \I_L^{d-1}\times \{a\}}
\sum_{e\in\efull{d,L,N}}|\nabla_e u|=
\sum_{e\in\efull{d,L,N}}|\nabla_e u|\leq 
\left|\efull{d,L,N}\right|\|\nabla u\|_{L^p(\efull{d,L,N})}\\
&
\leq \left|\efull{d,L,N}\right|C_2(d,p)
\max\left\{
\left\|\disDiff{+}{d} u\right\|_{L^p(\I_L^{d-1}\times\{0\})},
\left\|\disDiff{-}{d} u\right\|_{L^p(\I_L^{d-1}\times\{N\})}
\right\},
\end{split}\label{q12}
\end{align}
This (with $u\defeq u_n$ for $n\in\N$ and combined with
 \eqref{r30}),
and \cref{q03}
imply for all $x\in \I_L^{d-1}\times(\disint{0}{N})$ that
\begin{align}\begin{split}
|u_n(x)|
&\leq \left|\efull{d,L,N}\right|C_2(d,p)
\max\left\{
\left\|\disDiff{+}{d} u_n\right\|_{L^p(\I_L^{d-1}\times\{0\})},
\left\|\disDiff{-}{d} u_n\right\|_{L^p(\I_L^{d-1}\times\{N\})}
\right\}\\
&\leq \left|\efull{d,L,N}\right|C_2(d,p)C_1(d,p)^{n-1}\left\|\disDiff{+}{d} u_1\right\|_{L^p(\I_L^{d-1}\times\{0\})}.
\end{split}
\end{align}
The fact that $C_1(d,p)<1$ and the fact that $\forall\, x\in (0,1)\colon \sum_{k=1}^{\infty}x^{k-1}=(1-x)^{-1}$ then complete the proof of \cref{q13}.
\end{proof}
\begin{lemma}\label{q14}Assume \cref{q06a} and let $w\colon \Z^{d-1}\times(\disint{0}{N})\to\R$ satisfy that
\begin{align}
\forall\, x\in \Z^{d-1}\times( \disint{0}{N})\colon \quad 
w(x)=  \left[\sum_{k=1}^\infty (-1)^{k+1} u_k(x)\right].\label{q16}
\end{align}
Then it holds that
\begin{align}
w\in \mathbbm{S}_{d,L,N} ,\quad \forall\, x\in \Z^{d-1}\times \{0\}\colon 
\disDiff{+}{d}w(x)= v(x),\quad 
\forall\, x\in \Z^{d-1}\times \{N\}\colon 
\disDiff{-}{d}w(x)= v(x),\label{q16a}
\end{align}
and
\begin{align}
\begin{split}
\|\nabla w\|_{L^p(\efull{d,L,N})}
\leq \frac{C_2(d,p)}{1-C_1(d,p)}
\left\|v\right\|_{L^p(\I_L^{d-1}\times\{0\}) }.\label{q15}
\end{split}\end{align}
\end{lemma}
\begin{proof}[Proof of \cref{q14}]
First, \eqref{r30}  proves that for all  $n\in \N_0$, $ x\in \I_L^{d-1}\times(\disint{1}{N-1}) $
it holds
that $ (\Laplace u_n)(x)=0 $
and for all  $n\in \N_0$, $ x\in \I_L^{d-1}\times(\disint{0}{N}) $,
$i\in\disint{1}{d-1}$ it holds that $u_n(x+2L\unit{d}{i})=u_n(x)$. 
This, \eqref{q16}, and
the definition of $\mathbbm{S}_{d,L,N}$ imply that
$w\in \mathbbm{S}_{d,L,N} $. Furthermore, 
\cref{q13}, \eqref{q16},
\eqref{q03a}, and \eqref{q03e} show   that
\begin{align}
\forall\, x\in\Z^{d-1}\times\{0\}&\colon \quad 
 (\disDiff{+}{d}w) (x)= \sum_{k=1}^\infty(-1)^{k+1}(\disDiff{+}{d} u_k)(x)
=
 (\disDiff{+}{d} u_1)(x)=v(x)\quad\text{and}\quad
\\
\forall\, x\in\Z^{d-1}\times\{N\}&\colon \quad 
 (\disDiff{-}{d}w) (x)= \sum_{k=1}^\infty(-1)^{k+1}(\disDiff{-}{d} u_k)(x)
=0 =v(x).
\end{align}
This completes \eqref{q16a}.
Next, observe that \eqref{q16}, the triangle inequality, \eqref{q09} (with $u\defeq u_k $ for $k\in\N $), \cref{q13}, the fact that $\forall\, x\in (0,1)\colon \sum_{k=1}^{\infty}x^{k-1}=(1-x)^{-1}$, and \eqref{q03a} ensure that
\begin{align}
\begin{split}
&\|\nabla w\|_{L^p(\efull{d,L,N})}\leq 
\left\|\sum_{k=1}^\infty\nabla u_k\right\|_{L^p(\efull{d,L,N})}
\leq  
\sum_{k=1}^\infty\left\|\nabla u_k\right\|_{L^p(\efull{d,L,N})}
\\&\leq 
\sum_{k=1}^\infty\left[
C_2(d,p)
\max\left\{
\left\|\disDiff{+}{d} u_k\right\|_{L^p(\I_L^{d-1}\times\{0\}) },
\left\|\disDiff{-}{d} u_k\right\|_{L^p(\I_L^{d-1}\times\{N\}) }\right\}\right]
\\
&\leq 
\sum_{k=1}^\infty\left[
C_2(d,p)
C_1(d,p)^{k-1}
\left\|\disDiff{+}{d} u_1\right\|_{L^p(\I_L^{d-1}\times\{0\}) }\right]
= \frac{C_2(d,p)}{1-C_1(d,p)}
\left[\sum_{i=1}^{d-1}
\left\| v\right\|_{L^p(\I_L^{d-1}\times\{0\}) }
\right].
\end{split}\end{align}
This implies  \eqref{q15}.
The proof of \cref{q14} is thus completed.
\end{proof}
\begin{lemma}\label{q21}
Assume \cref{q06a} and let  $N,L\in\N $, $d\in [2,\infty)\cap \N$, $p\in (1,\infty)$ $v\in\mathbbm{B}_{d,L,N}$ satisfy that $1/4\leq L/N\leq 4$ and
$\mean{v}{\I_L^{d-1}\times\{0,N\}}=0$. Then there exists
$w\in \mathbbm{S}_{d,L,N}$
such that
\begin{align}
 \forall\, x\in \Z^{d-1}\times \{0\}\colon \quad 
\disDiff{+}{d}w(x)= v(x),\quad 
\forall\, x\in \Z^{d-1}\times \{N\}\colon \quad 
\disDiff{-}{d}w(x)= v(x),\label{q19}
\end{align}
and
\begin{align}
\begin{split}
\|\nabla w\|_{L^p(\efull{d,L,N})}
\leq\left[\frac{4C_2(d,p)}{1-C_1(d,p)}+2d16^d\right]
\left\|v\right\|_{L^p(\I_L^{d-1}\times\{0,N\}) }.
\end{split}\end{align}
\end{lemma}
\begin{proof}[Proof of \cref{q21}]
Let $v_1,v_2,v_3\in \mathbbm{B}_{d,L,N}$ be the functions which satisfy that
\begin{align}\begin{split}
\forall\, x\in\I_L^{d-1}\times\{0\}&\colon \quad  v_1(x)= v(x)-\mean{v}{\I_L^{d-1}\times\{0\}},\quad  v_2(x)=0,\quad 
v_3(x)=\mean{v}{\I_L^{d-1}\times\{0\}};
\\
\forall\, x\in\I_L^{d-1}\times\{N\}&\colon \quad 
v_1(x)=0,\quad 
v_2(x)= v(x)-\mean{v}{\I_L^{d-1}\times\{N\}},\quad v_3(x)=\mean{v}{\I_L^{d-1}\times\{N\}} ,\end{split}\label{q18}
\end{align}
and let $w_3\colon \Z^{d-1}\times(\disint{0}{N})\to\R$ be the function which satisfies that
\begin{align}
\forall\, n\in \disint{0}{N},\, 
x\in \I_L^{d-1}\times\{n\}\colon \quad 
w_3(x)= n \mean{v}{\I_L^{d-1}\times\{0\}}.
\end{align}
This construction and the fact that
$\mean{v}{\I_L^{d-1}\times\{N\}}+\mean{v}{\I_L^{d-1}\times\{0\}}=0$ show that
\begin{align}
\begin{split}
&\forall\, x\in \I_L^{d-1}\times(\disint{1}{N-1}) \colon \quad (\Laplace w_3) (x) =0,\\
&\forall\, x\in \I_L^{d-1}\times\{0\}
 \colon\quad (\disDiff{+}{d}w_3)(x)= \mean{v}{\I_L^{d-1}\times\{0\}}=v_3(x),\\
&\forall\, x\in \I_L^{d-1}\times\{N\}
 \colon\quad (\disDiff{-}{d}w_3)(x)= \mean{v}{\I_L^{d-1}\times\{N\}}=v_3(x).
\end{split}\label{q18b}
\end{align}
Next, the fact that
\begin{align}\begin{split}
\{e\in \efull{d,L,N}\colon \nabla_ew_3\neq 0\}\subseteq
\left\{(x,x\pm \unit{d}{d})\colon x\in \Z^d\cap ([0,L]^{d-1}\times [0,N])\setminus ((0,L)^{d-1}\times (0,N))\right\},
\end{split}\end{align}
the fact that $[0,L]^{d-1}\times[0,N]$
has $2d$ faces, the fact that $N/L\in [1/4,4]$, \eqref{q18b},
\eqref{q18}, and Jensen's inequality imply that
\begin{align}
\|\nabla w_3\|_{L^p(\efull{L,N})}^p\leq (2d)16^d\|v\|_{L^p(\I_L^{d-1}\times\{0,N\})}^p,\label{q18c}
\end{align}
which is a very rough estimate, however, gives a constant depending only on $d$.
Furthermore,
\cref{q14} (together with a simple change of coordinates),  shows that there exists 
$w_1\in \mathbbm{H}_{d,L,\geq 0}$,
$w_2\in \mathbbm{H}_{d,L,\leq N}$ such that it holds
that 
\begin{align}\begin{split}
\forall\, x\in \Z^{d-1}\times \{0\},\,i\in \{1,2\}\colon \quad\disDiff{+}{d}w_i(x)= v_i(x),\\
\forall\, x\in \Z^{d-1}\times \{N\},\, i\in \{1,2\}\colon \quad\disDiff{-}{d}w_i(x)= v_i(x),\end{split}\label{q17}
\end{align}
and such that for all $ i\in \{1,2\}$ it holds that
\begin{align}
\begin{split}
\|\nabla w_i\|_{L^p(\efull{d,L,N})}
\leq \frac{C_2(d,p)}{1-C_1(d,p)}
\left\|v_i\right\|_{L^p(\I_L^{d-1}\times\{0,N\}) }.\label{q20}
\end{split}\end{align}
Now let $w\in \mathbbm{B}_{d,L,N}$ be the function which satisfies for all $ x\in\I_L^{d-1}\times\{0,N\}$ that $w(x)=w_1(x)+w_2(x)+w_3(x)$. Then \eqref{q17}, \eqref{q18b}, 
and \eqref{q18}
imply \eqref{q19}.
Furthermore, the triangle inequality,
Jensen's inequality, and \eqref{q18} show  that
 $\forall\, i\in\{1,2\}\colon \left\|v_i\right\|_{L^p(\I_L^{d-1}\times\{0\}) }\leq 2\left\|v\right\|_{L^p(\I_L^{d-1}\times\{0,N\}) }$.
This, the triangle inequality, \eqref{q20}, \eqref{q18}, and
\eqref{q18c}  prove that
\begin{align}
\begin{split}
&\|\nabla w\|_{L^p(\efull{d,L,N})}
\leq \left[\frac{4C_2(d,p)}{1-C_1(d,p)}+2d16^d\right]
\left\|v\right\|_{L^p(\I_L^{d-1}\times\{0,N\}) }
.
\end{split}\end{align}
This completes the proof of \cref{q21}.
\end{proof}
Observe that \cref{q21} shows the existence of the Neumann problem on strips. Furthermore, the uniqueness is straightforward (e.g. by means of the maximum principle applied to the derivatives $ \disDiff{+}{d}u$ defined on $\Z^{d-1}\times(\disint{0}{N-1}) $ and harmonic on
$\Z^{d-1}\times(\disint{1}{N-2})$ and 
the derivatives $ \disDiff{-}{d}u$ defined on $\Z^{d-1}\times(\disint{1}{N}) $ and harmonic on
$\Z^{d-1}\times(\disint{2}{N-1})$). We therefore obtain \cref{s11} below. However, more important for us is \cref{s12} that follows from \cref{q21}, \cref{q25}, and the uniqueness in \cref{s11}.
\begin{corollary}\label{s11}Let \cref{v02} be given.
Let $d,N\in [2,\infty)\cap\N$, $p\in (1,\infty)$,
$L\in\N$,
let $\I_L$  be the set given by $\I_L=\disint{-L+1}{L} $, let
$v\colon \Z^{d-1}\times\{0,N\}\to\R$ satisfy 
for all $x\in \Z^{d-1}\times \{0,N\}$, $  i\in \disint{1}{d-1}$ that
$v(x)=v(x+2L\unit{d}{i})$ and 
$\mean{v}{\I_L^{d-1}\times\{0\}}=\mean{v}{\I_L^{d-1}\times\{N\}}=0$
. Then  there exists uniquely $w \colon
\Z^{d-1}\times (\disint{0}{N}) \to\R$ such that
\begin{enumerate}[i)]
\item it holds that $\mean{w}{\I_L^{d-1}\times\{0\}} =\mean{w}{\I_L^{d-1}\times\{N\}}=0$,
\item it holds for all $x\in \Z^{d-1}\times\{0\}$  that 
$\disDiff{+}{d}w(x)=v(x)$,
\item it holds for all $x\in \Z^{d-1}\times\{N\}$  that 
$\disDiff{-}{d}w(x)=v(x)$,
 and
\item it holds for all $x\in \Z^{d-1}\times(\disint{1}{N-1})$  that 
$(\Laplace w)(x)=v(x)$.
\end{enumerate}
\end{corollary}
\begin{corollary}\label{s12}Let \cref{v02} be given.
For $L\in\N $ let $\I_L$  be the set given by $\I_L=\disint{-L+1}{L} $ and
let $\efull{d,L,N}\subseteq E_d$ be the set of edges
given by
\begin{align}
\efull{d,L,N}=\left\{(x,y)\in E_d\colon 
\tfrac12(x+y)\in \left([0,L]^{d-1}\times[0,N]\right)\setminus
\left([1,L-1]^{d-1}\times[1,N-1]\right)\right\},
\end{align}
Then there exists
$C\colon([2,\infty)\cap\N)\times(1,\infty)\to(0,\infty) $
such that
for all $d\in [2,\infty)\cap\N$, $p\in (1,\infty)$, $N,L\in\N$,
$w\colon \Z^{d-1}\times (\disint{0}{N})\to\R$ with $1/4\leq L/N\leq 4$,
$\forall\, x\in \Z^{d-1}\times (\disint{0}{N})$, $  i\in \disint{1}{d-1}\colon w(x)=w(x+2L\unit{d}{i})$, 
and
$\forall\, x\in \Z^{d-1}\times(\disint{1}{N-1})\colon (\Laplace w)(x)=0$ it holds that
\begin{align}
\|\nabla w\|_{L^p(\efull{d,L,N})}\leq  C(d,p)
\left[
\left\|\disDiff{+}{d} w\right\|_{L^p(\I_L^{d-1}\times\{0\}) }+
\left\|\disDiff{-}{d} w\right\|_{L^p(\I_L^{d-1}\times\{N\}) }
\right].
\end{align}
\end{corollary}
\subsection{Proof of the main result in the Neumann case}\label{e04}
Observe that if $f\in C^1( \R ,\R)$ is an even $2N$-periodic function, then $f'$ is an odd function and in particular $f'(0)=f'(N)=0$. \cref{r15c} below adapts this simple observation into the discrete case. 
\begin{lemma}\label{r15c}
Let $N\in \N$, $c\in\R$, $f\colon \Z\to \R $ satisfy  that $\forall\, x\in\Z\colon f(x)=f(x+2(N-1))$ and $\forall\, x\in \disint{-(N-2)}{N-1}\colon f(x)=f(1-x)+c$. Then $f(1)-f(0)=f(N-1)-f(N)=0$.
\end{lemma}
\begin{proof}[Proof of \cref{r15c}]The fact $f(1)-f(0)=(f(0)+c)-(f(1)+c)$ proves that $f(1)-f(0)=0$. Next, the fact that
$f(N-1)-f(N)=(f(1-(N-1)+2(N-1))+c)-(f(1-N+2(N-1))+c)=f(N)-f(N-1)$ proves that $f(N-1)-f(N)=0$. This completes the proof  of \cref{r15c}.
\end{proof}
\cref{r15a} below gives the technical details in order to make the Neumann conditions vanish.
\begin{lemma}[Even reflections]\label{r15a}Let $d\in [2,\infty)\cap\N$,
$w\colon \Z^{d-1}\times (\disint{0}{N})\to\R$, $j\in \disint{1}{d-1}$ satisfy that
\begin{align}
&
\begin{aligned}
&\forall\, x\in \Z^{d-1}\times (\disint{0}{N}),\,  i\in \disint{1}{d-1}\colon\quad w(x)=w(x+2(N-1)\unit{d}{i}),\\
&\forall\, x\in \Z^{d-1}\times(\disint{1}{N-1})\colon\quad (\Laplace w)(x)=0,
\end{aligned} \label{s06ba}
\end{align}
and
\begin{align}
\begin{aligned}
\forall\, x\in \Z^{j-1}\times (\disint{-(N-2)}{N-1})\times \Z^{d-j-1}\times\{0\} &\colon\quad (\disDiff{+}{d}w)(x)=(\disDiff{+}{d}w)(x+\unit{d}{j}-2x_j\unit{d}{j}),\\
\forall\, x\in \Z^{j-1}\times (\disint{-(N-2)}{N-1})\times \Z^{d-j-1}\times\{N\} &\colon\quad (\disDiff{-}{d}w)(x)=(\disDiff{-}{d}w)(x+\unit{d}{j}-2x_j\unit{d}{j}).
\end{aligned}
\label{s06ca}
\end{align}
Then it holds that 
\begin{align}
\forall\, x\in \Z^{j-1}\times \{0\}\times \Z^{d-j-1}\times (\disint{0}{N})\colon \quad (\disDiff{+}{j}w)(x)=0,\\
\forall\, x\in \Z^{j-1}\times \{N\}\times \Z^{d-j-1}\times (\disint{0}{N})\colon \quad (\disDiff{-}{j}w)(x)=0.
\end{align}
\end{lemma}
\begin{proof}[Proof of \cref{r15a}]Let $\tilde{w}\colon \Z^{d-1}\times (\disint{0}{N})\to\R$ be the function given by 
\begin{align}
\forall\, x\in \Z^{d-1}\times (\disint{0}{N})\colon \quad 
 \tilde{w}(x)=w(x+\unit{d}{j}-2x_j\unit{d}{j})\label{s07}.
\end{align}
Then  \eqref{s06ba} implies that
\begin{align}\begin{split}
&\forall\, x\in \Z^{d-1}\times (\disint{0}{N}),\,  i\in \disint{1}{d-1}\colon\quad \tilde{w}(x)=\tilde{w}(x+2(N-1)\unit{d}{i}),\\
&\forall\, x\in \Z^{d-1}\times(\disint{1}{N-1})\colon\quad (\Laplace \tilde{w})(x)=0.\end{split}
\end{align}
Furthermore, \eqref{s07}, \eqref{s06ca}, the fact that
$j\neq d$,
and the periodicity in \eqref{s06ba} and \eqref{s07} imply that 
\begin{align}\begin{split}
\forall\, x\in \Z^{d-1}\times\{0\} \colon\quad (\disDiff{+}{d}w)(x)=(\disDiff{+}{d}\tilde{w})(x),\\
\forall\, x\in \Z^{d-1}\times\{N\} \colon \quad (\disDiff{-}{d}w)(x)=(\disDiff{-}{d}\tilde{w})(x).
\end{split}
\end{align}
This and
uniqueness ("up to a constant") of the Neumann problem show that there exists $c\in\R$ such that for all
$\Z^{d-1}\times (\disint{0}{N})$ it holds that
$w(x)=\tilde{w}(x)+c= w(x+\unit{d}{j}-2x_j\unit{d}{j})+c $.
This and \cref{r15c} (applied to the $j$-th variable) complete the proof of \cref{r15a}.
\end{proof}
\begin{theorem}\label{r20a}For every $d,N\in [2,\infty)\cap\N$ 
let $\efull{d,N}\subseteq E_d$ be the sets of  edges
given by
\begin{align}
\efull{d,N}=\left\{(x,y)\in E_d\colon 
\tfrac12(x+y)\in ([0,L]^{d-1}\times[0,N])\setminus
\left([1,L-1]^{d-1}\times[1,N-1]\right)\right\},
\end{align}
let $\enorT{d,N}\subseteq \Z^d$ be the set of vertices given by
\begin{align}
\enorT{d,N}=\left[\bigcup_{i=1}^d (\disint{1}{N-1})^{i-1}\times\{0,N\}\times (\disint{1}{N-1})^{d-i}\right],
\end{align}
let $\mathcal{N}_{d,N} $ be the set of all functions 
$v\colon\enorT{d,N}\to\R $ with $\mean{v}{\enorT{d,N}}=0$,
and
let $\mathbbm{Q}_{d,N}$ be the set of  functions $u\colon (\disint{0}{N})^{d}\to\R$ with the property that $\forall\, x\in (\disint{1}{N-1})^d\colon (\Laplace u)(x) =0$.
Then there exists 
a  function
$C\colon ([2,\infty)\cap\N)\times(1,\infty)\to\R$ such that for all 
$d,N\in [2,\infty)\cap\N$, $p\in (1,\infty)$,  $v\colon \mathcal{N}_{d,N}\to\R$ there
exists a function $u\in \mathbbm{Q}_{d,N}$ 
such that 
\begin{align}\begin{split}
&\forall\, i\in \disint{1}{d},\, x\in (\disint{1}{N-1})^{i-1}\times\{0\}\times (\disint{1}{N-1})^{d-i}\colon \disDiff{+}{i}u=v,\\
&\forall\, i\in \disint{1}{d},\, x\in (\disint{1}{N-1})^{i-1}\times\{N\}\times (\disint{1}{N-1})^{d-i}\colon \disDiff{-}{i}u=v,
\end{split}
\end{align}
(i.e., $v$ is the Neumann condition of $u$)
and such that
$\|\nabla u \|_{L^p(\efull{d,N})}\leq C(d,p)\|v\|_{L^p(\enorT{d,N})}$.
\end{theorem}
\cref{f01} 
illustrates the sets $\efull{d,N}$ and $\etanT{d,N}$ in \cref{r20a} above in the case $d=2$, $N=10$: the
elements of $\efull{d,N}$ are all red and blue edges and the elements of $\etanT{d,N}$ are all red vertices without the ones at four corners.
\begin{proof}[Proof of \cref{r20a}]
First, let
$d,N\in [2,\infty)\cap\N$, $p\in (1,\infty)$,  $v\in \mathcal{N}_{d,N} $ be arbitrary but fixed and we will successively construct functions $w_i:\mathbbm{Z}^{i-1}\times(\disint{0}{N})\times\mathbbm{Z}^{d-i}\to\mathbbm{R}$,
$i\in\disint{1}{d}$, 
with the property $\mathcal{A}(w_1,\ldots,w_d)$ defined as follows:
For every
$\ell\in \disint{1}{d}$ and every collection of functions
$w_i:\mathbbm{Z}^{i-1}\times(\disint{0}{N})\times\mathbbm{Z}^{d-i}\to\mathbbm{R}$,
$i\in \disint{1}{\ell}$ let
$\mathcal{A}(w_1,\ldots,w_\ell)$ be
the statement which is true if 
\begin{enumerate}[(i)]
\item it holds 
for all $i\in \disint{1}{\ell}$,
$x\in \mathbbm{Z}^{i-1}\times(\disint{1}{N-1})\times\mathbbm{Z}^{d-i}$ that
$
\Laplace w_i(x)=0,
$\item it holds 
for all $i\in(\disint{1}{\ell})$, $j\in(\disint{1}{d})\setminus \{i\}$,
$x\in \mathbbm{Z}^{i-1}\times(\disint{0}{N})\times\mathbbm{Z}^{d-i}$
that
\begin{align}
w_i(x)=w_i(x+2(N-1)\unit{d}{j}),
\end{align}
\item it holds 
for all $i\in \disint{1}{\ell}$,
$x\in (\disint{1}{N-1})^{i-1}\times\{0\}\times (\disint{1}{N-1})^{d-i}$ that
\begin{align}
v(x)=\sum_{\nu=1}^{i}(\disDiff{+}{i} w_\nu)(x),\label{s20}
\end{align}
\item it holds 
for all $i\in \disint{1}{\ell}$,
$x\in (\disint{1}{N-1})^{i-1}\times\{N\}\times (\disint{1}{N-1})^{d-i}$ that
\begin{align}\label{s21}
v(x)=\sum_{\nu=1}^{i}(\disDiff{-}{i} w_\nu)(x),
\end{align}
\item it holds
for all $i\in \disint{1}{\ell}$, $j\in \disint{1}{i-1}$,
$x\in (\disint{1}{N-1})^{j-1}\times\{0\}\times (\disint{1}{N-1})^{d-j}$ that 
$(\disDiff{+}{j} w_i)(x)=0,$ and if
\item it holds
for all $i\in \disint{1}{\ell}$, $j\in \disint{1}{i-1}$,
$x\in (\disint{1}{N-1})^{j-1}\times\{N\}\times (\disint{1}{N-1})^{d-j}$ that 
$
(\disDiff{-}{j} w_i)(x)=0,
$
\end{enumerate}
As a first step, let $v_1\colon \{0,N\}\times\Z^{d-1}\to\R$ be the function which satisfies that
\begin{align}
&\forall\, x\in \{0,N\}\times(\disint{1}{N-1})^{d-1}\colon v_1(x)=v(x),\label{r05a}\\
& \forall\, j\in \disint{2}{d},\,x\in \{0,N\}\times\left(\disint{-(N-2)}{N-1}\right)^{d-1}\colon v_1(x)=-v_1(x+\unit{d}{j}-2x_j\unit{d}{j}),\label{r06a}\\
&\forall\, j\in \disint{2}{d},\, x\in \{0,N\}\times\Z^{d-1}\colon v_1(x)=v_1(x+2(N-1)\unit{d}{j}).\label{r07a}
\end{align}
This implies that 
$\mean{v_1}{\{0,N\}\times\I_{N-1}^{d-1}}=0$.
Let
$w_1\colon (\disint{0}{N})\times\mathbbm{Z}^{d-1}\to\mathbbm{R}$  be a function
whose existence is ensured by \cref{s11} and
which satisfies that
\begin{align}\begin{split}\label{r08a}
&\forall\, j\in \disint{2}{d},\, x\in \{0,N\}\times\Z^{d-1}\colon\quad w_1(x)=w_1(x+2(N-1)\unit{d}{j})\\
&\forall\, x\in \{0\}\times \Z^{d-1} \colon\quad (\disDiff{+}{i}w_1)(x)=v_1(x),\\
&\forall\, x\in \{N\}\times \Z^{d-1} \colon\quad (\disDiff{-}{i}w_1)(x)=v_1(x),\\
&\forall\, x\in (\disint{1}{N-1})\times \Z^{d-1} \colon\quad (\Laplace w_1)(x)=0.
\end{split}
\end{align}
Observe that $\mathcal{A}(w_1)$ holds. Next, let $\ell\in \disint{1}{d-1}$ and suppose that we have constructed 
$w_1,\ldots,w_\ell$ such that $\mathcal{A}(w_1,\ldots,w_\ell)$ holds.
Let 
$
v_{\ell+1}\colon \Z^{\ell}\times\{0,N\}\times\Z^{d-\ell-1}\to\R
$
be the function which satisfies that
\begin{enumerate}[i)] 
\item 
for all $x\in (\disint{1}{N-1})^{\ell}\times\{0\}\times(\disint{1}{N-1})^{d-\ell-1}$
it holds that
\begin{align}
v_{\ell+1}(x)=v(x)-\sum_{\nu=1}^{\ell}(\disDiff{+}{\ell+1}w_\nu)(x)\label{r10a} 
\end{align}
\item 
for all $x\in (\disint{1}{N-1})^{\ell}\times\{N\}\times(\disint{1}{N-1})^{d-\ell-1}$
it holds that
\begin{align}
v_{\ell+1}(x)=v(x)-\sum_{\nu=1}^{\ell}(\disDiff{-}{\ell+1}w_\nu)(x),
\label{r10aa} 
\end{align}

\item for all 
$j\in(\disint{1}{d})\setminus \{\ell+1\}$,
$x\in \mathbbm{Z}^{\ell}\times\{0,N\}\times\mathbbm{Z}^{d-\ell-1}$
it holds that
\begin{align}
v_{\ell+1}(x)=v_{\ell+1}(x+2(N-1)\mathbf{e}_j)
\end{align}
\item for all
$j\in \disint{1}{\ell},\,x\in (\disint{-(N-2)}{N-1})^{\ell}\times \{0,N\}\times(\disint{-(N-2)}{N-1})^{d-\ell-1}$ it holds that
\begin{align}
 v_{\ell+1}(x)=v_{\ell+1}(x+\unit{d}{j}-2x_j\unit{d}{j})\label{r11a},
\end{align}
and
\item for all
$ j\in \disint{\ell+2}{d}$, $x\in (\disint{-(N-2)}{N-1})^{\ell}\times \{0,N\}\times(\disint{-(N-2)}{N-1})^{d-\ell-1}$ it holds that
\begin{align}\label{r17}
 v_{\ell+1}(x)=-v_{\ell+1}(x+\unit{d}{j}-2x_j\unit{d}{j}).
\end{align}
\end{enumerate}
Next, we show that 
\begin{align}\mean{v_{\ell+1}}{\I_{N-1}^{\ell}\times\{0,N\}\times\I_{N-1}^{d-\ell-1}}=0.\label{r19}
\end{align}
To this end we distinguish two cases: $\ell+1=d$ and $\ell+1<1$. First, when $\ell+1<d$, then the odd reflection in \eqref{r17} implies \eqref{r19}. Next, we consider the case $\ell+1=d$.
Note that in this case
$\disint{\ell+2}{d}=\emptyset$ and  we therefore cannot use \eqref{r17}.
In this step, to shorten the notation, 
for every $i\in\disint{1}{d}$, $j\in \{0,N\}$
let $F_i^{j}$ be the set given by
\begin{align}
F_i^{j}= (\disint{1}{N-1})^{i}\times\{j\}\times(\disint{1}{N-1})^{d-i-1}.
\end{align}
The fact that $v\in \mathcal{N}_{d,N}$ implies that $\mean{v}{\enorT{d,N}}=0$ and hence that
\begin{align}
\sum_{i=1}^d\left[\sum_{x\in F_i^0}v(x)+\sum_{x\in F_i^N}
v(x)
\right]=0.\label{s23}
\end{align}
The fact that
$\forall\, j\in\disint{1}{d-1},\, x\in (\disint{1}{N-1})^d\colon (\Laplace w_j)(x)=0$, following from the induction hypothesis and the case assumption $\ell+1=d$,
shows that the Neumann conditions of $w_j$, $j\in \disint{1}{d-1}$, on $(\disint{0}{N})^d$ have vanishing means, i.e., it holds for all $j\in\disint{1}{d-1}$ that
\begin{align}
\sum_{i=1}^d\left[\sum_{x\in F_i^0}(\disDiff{+}{i}w_j)+\sum_{x\in F_i^N}
(\disDiff{-}{i}w_j)(x)
\right]
=0.
\end{align}
This and \eqref{s23} prove that
\begin{align}
\sum_{i=1}^d\left[\sum_{x\in F_i^0}
\left(v(x)-\sum_{j=1}^{d-1}(\disDiff{+}{i}w_j)(x)\right)+
\sum_{x\in F_i^N}\left(v(x)-\sum_{j=1}^{d-1}
(\disDiff{-}{i}w_j)(x)\right)
\right]=
0.\label{s22}
\end{align}
Furthermore,  \eqref{s20} and \eqref{s21},
following from the induction hypothesis,
 imply that
\begin{align}
\sum_{i=1}^{d-1}\left[\sum_{x\in F_i^0}
\left(v(x)-\sum_{j=1}^{d-1}(\disDiff{+}{i}w_j)(x)\right)+
\sum_{x\in F_i^N}\left(v(x)-\sum_{j=1}^{d-1}
(\disDiff{-}{i}w_j)(x)\right)
\right]=
0.
\end{align}
This, \eqref{r10a}, \eqref{r10aa} and \eqref{s22} show that
\begin{align}
\sum_{x\in F_d^0\cup F_d^N}v_d(x)=
\sum_{x\in F_d^0}
\left(v(x)-\sum_{j=1}^{d-1}(\disDiff{+}{d}w_j)(x)\right)+
\sum_{x\in F_d^N}\left(v(x)-\sum_{j=1}^{d-1}
(\disDiff{-}{d}w_j)(x)\right)=0.
\end{align}
This and \eqref{r11a} complete the proof of \eqref{r19}. Now, \eqref{r19} and \cref{s11} imply that there exists 
$w_{\ell+1}\colon \Z^{\ell}\times (\disint{0}{N})\times  \Z^{d-\ell-1}\to\R $ such that
\begin{enumerate}[i)] 
\item 
for all $x\in (\disint{1}{N-1})^{\ell}\times\{0\}\times(\disint{1}{N-1})^{d-\ell-1}$
it holds that
\begin{align}\label{s18}
(\disDiff{+}{\ell+1}w_{\ell+1})(x)=v_{\ell+1}(x)
\end{align}
\item 
for all $x\in (\disint{1}{N-1})^{\ell}\times\{N\}\times(\disint{1}{N-1})^{d-\ell-1}$
it holds that
\begin{align}\label{s19}
(\disDiff{-}{\ell+1}w_{\ell+1})(x)=v_{\ell+1}(x)
\end{align}
\item for all 
$j\in(\disint{1}{d})\setminus \{\ell+1\}$,
$x\in \mathbbm{Z}^{\ell}\times\{0,N\}\times\mathbbm{Z}^{d-\ell-1}$
it holds that
\begin{align}
w_{\ell+1}(x)=w_{\ell+1}(x+2(N-1)\unit{d}{j}),\label{s30}
\end{align}
and 
\item 
for all $i\in \disint{1}{\ell}$,
$x\in \mathbbm{Z}^{i-1}\times(\disint{1}{N-1})\times\mathbbm{Z}^{d-i}$ that
\begin{align}\label{s31}
(\Laplace w_i)(x)=0.
\end{align}
\end{enumerate}
Observe that \eqref{s18}, \eqref{s19}, the even reflection in \eqref{r11a}, and \cref{r15a} show that
\begin{align}\begin{split}
&\forall\, j\in \disint{1}{\ell},\,
x\in (\disint{1}{N-1})^{j-1}\times\{0\}\times (\disint{1}{N-1})^{d-j}\colon\quad (\disDiff{+}{j} w_{\ell+1})(x)=0,\\
&\forall\, j\in \disint{1}{\ell},\,
x\in (\disint{1}{N-1})^{j-1}\times\{N\}\times (\disint{1}{N-1})^{d-j}\colon\quad (\disDiff{-}{j} w_{\ell+1})(x)=0.\end{split}
\end{align}
Combining \eqref{s18}, \eqref{s19}, \eqref{r10a}, and \eqref{r10aa}
yields that
\begin{align}\begin{split}
\forall\, x\in (\disint{1}{N-1})^{\ell}\times\{0\}\times (\disint{1}{N-1})^{d-\ell-1}\colon\quad 
v(x)=\sum_{\nu=1}^{\ell+1}(\disDiff{+}{i} w_\nu)(x),\\
\forall\, x\in (\disint{1}{N-1})^{\ell}\times\{N\}\times (\disint{1}{N-1})^{d-\ell-1}\colon\quad
v(x)=\sum_{\nu=1}^{\ell+1}(\disDiff{-}{i} w_\nu)(x).
\end{split}\end{align}
This, \eqref{s30},  \eqref{s31}, and the induction hypothesis $\mathcal{A}(w_1,\ldots,w_\ell)$ imply that 
$\mathcal{A}(w_1,\ldots,w_{\ell+1})$ holds. We have thus recursively constructed a sequence $w_1,\ldots,w_d$ with $\mathcal{A}(w_1,\ldots,w_d)$. 
Now, let $u\colon (\disint{0}{N})^d\to\R$ be the function which satisfies for all 
$x\in (\disint{0}{N})^d$ that $u(x)=\sum_{i=1}^{d}w_i(x)$. 
The property $ \mathcal{A}(w_1,\ldots,w_d)$ then implies that $u\in \mathbbm{Q}_{d,N}$, 
\begin{align}
&\forall\, i\in \disint{1}{d},\, x\in (\disint{1}{N-1})^{i-1}\times\{0\}\times (\disint{1}{N-1})^{d-i}\colon\quad (\disDiff{+}{i}u)(x)=v(x),
\end{align} and
\begin{align}
&\forall\, i\in \disint{1}{d},\, x\in (\disint{1}{N-1})^{i-1}\times\{N\}\times (\disint{1}{N-1})^{d-i}\colon \quad (\disDiff{-}{i}u)(x)=v(x).
\end{align}
The rest of the proof 
is now clear. We only give a sketch.
We write $C(d,p)$ to denote possibly different real numbers that only depend on $d,p$ and write
for $i\in\disint{1}{d}$ to lighten the notation
\begin{align}
F_i^{j}= (\disint{1}{N-1})^{i}\times\{j\}\times(\disint{1}{N-1})^{d-i-1},\quad 
\hat F_i^{j}= \I_{N-1}^{i}\times\{j\}\times\I_{N-1}^{d-i-1}.
\end{align}
Then \cref{s12} shows that
\begin{align}
\|\nabla u\|_{L^p(\efull{d,L,N})}
&\leq \sum_{i=1}^d
\|\nabla w_i\|_{L^p(\efull{d,L,N})}
\leq C(d,p)\sum_{i=1}^d
\left[
\left\|\disDiff{+}{d} w_i\right\|_{L^p(\hat{F}_i^0) }+
\left\|\disDiff{-}{d} w_i\right\|_{L^p(\hat{F}_i^N) }
\right]\\
&\leq C(d,p)\sum_{i=1}^d
\left[
\left\|v\right\|_{L^p({F}_i^0) }+
\left\|v\right\|_{L^p({F}_i^N) }
\right]\leq C(d,p)\|v\|_{L^p(\enorT{d,N})}.
\end{align}
The proof of \cref{r20a} is thus completed.
\end{proof}

\appendix
\section{Appendix}
For convenience we include here some simple results.
\subsection{Some basic results}
\begin{lemma}[Complex square root]\label{y07}
There exists a unique function $R\in C\left( \C\setminus (-\infty,0), \C\right)$ such that
 $R\upharpoonright_{\C\setminus(-\infty,0]}$ is holomorphic  and
 $\forall\,  z\in \C\setminus( -\infty,0)\colon R(z)^2={z}$.
\end{lemma}
\begin{proof}[Proof of \cref{y07}]
Let $\log\colon \C\setminus\{0\}\to\C$ be 
the principle branch of the logarithm, i.e., it holds 
for all $z\in\C\setminus\{0\}$ that $\exp(\log z)=z $ and $-\pi\leq \Im (\log z) \leq \pi $, where $\Im$ denotes the imaginary part
(see, e.g., Theorem I.2.11 in Freitag and Busam~\cite{FB05}), and let  $R\colon  \C\setminus(-\infty,0)\to\C$ be given by
\begin{align}
\forall\, z\in \C\setminus (-\infty,0]\colon\quad R(z)= \exp(\tfrac{1}{2}\log z)\quad\text{and}\quad
R(0)=0 .
\end{align}
An elementary property of the function $\exp$ then shows that
\begin{align}\label{y19}
\forall\, z\in \C\setminus(-\infty,0)\colon\quad  R(z)^2=z. 
\end{align}
Furthermore, the fact that $\exp$ and $\log \upharpoonright_{\C\setminus( -\infty,0]}$ are holomorphic (cf.  Theorem I.5.8 in \cite{FB05}) and the chain rule then show that $R\upharpoonright_{\C\setminus( -\infty,0]}$ is holomorphic. 
Finally,  we prove by contradiction that $R$ is continuous at $0$. Suppose there exist $(z_n)_{n\in\N}\subseteq \C\setminus (-\infty,0)$, $\epsilon\in (0,1)$ such that $\lim_{n\to\infty} z_n=0$ and 
$\forall\, n\in\N\colon |R(z_n)|>\epsilon $ and without lost of generality assume for all $n\in \N$ that
$|z_n| <1$.
Then \eqref{y19} implies for all
$ n\in\N$ that $ |R(z_n)|^2=|z_n|\leq 1$. The Bolzano theorem hence proves that there exists a sequence $(n_k)_{k\in\N}\subseteq\N$ such that
$(R(z_{n_k}))_{k\in\N}$ converges. This, \eqref{y19}, and
the fact that  $\lim_{n\to\infty} z_n=0$ then
 demonstrate that 
$\lim_{k\to\infty}|R(z_{n_k})|^2= \lim_{k\to\infty}|z_{n_k}|=0$.  This contradicts the assumption that $\forall\, n\in\N\colon |R(z_n)|>\epsilon $. 
Thus, $R$ is continuous at $0$. The proof of \cref{y07} is thus completed.
\end{proof}
\begin{lemma}\label{z05}
\begin{enumerate}[i)]
\item \label{z05a}It holds that
$
0<\inf_{s\in [-\pi,\pi]\setminus \{0\}}\frac{1-\cos (s)}{s^2}<
\sup_{s\in [-\pi,\pi]\setminus \{0\}}\frac{1-\cos (s)}{s^2}<\infty
$,
\item\label{z05b} it holds that
$0<
\inf_{s\in [-\pi,\pi]\setminus\{0\}}\left|\frac{e^{-\ima s}-1}{s}\right|\leq 
\sup_{s\in [-\pi,\pi]\setminus\{0\}}\left|\frac{e^{-\ima s}-1}{s}\right|<\infty$,
%\item it holds that $\sup_{s\in [-\pi,\pi]\setminus\{0\}}\left|\frac{e^{-\ima s}-1}{s}\right|<\infty$, and
and
\item \label{z05d}
it holds that
$\sup_{s\in [-\pi,\pi]\setminus\{0\}}\left|s^2\frac{d}{ds}\left(\frac{1}{e^{-\ima s}-1}\right)\right|<\infty$.
\end{enumerate}
\end{lemma}
\begin{proof}[Proof of \cref{z05}]Throughout the proof let $g,h\colon  [-\pi,\pi]\to \R$ be the functions which satisfy 
for all $s\in [-\pi,\pi]\setminus\{0\}$ that
\begin{align}
g(s)= \frac{1-\cos (s)}{s^2},\quad
h(s)=\left|\frac{e^{-\ima s}-1}{s}\right|,\quad
g(0)=\frac12, \quad\text{and}\quad h(0)=1.
\end{align}
The fact that
$\lim_{s\to0}\frac{1-\cos (s)}{s^2}= \frac{1}{2}$ and the fact that
$\left|\frac{e^{-\ima s}-1}{s}\right|=1$
show that $g,h\in C([-\pi,\pi],\R)$. The extreme value theorem and the fact  that  $\forall\, s\in [-\pi,\pi]\colon\min\{g(s),h(s)\}> 0$ then imply that
\begin{align}
0<\inf_{s\in [-\pi,\pi]}g(s)\leq\inf_{s\in [-\pi,\pi]\setminus \{0\}}\tfrac{1-\cos (s)}{s^2} \leq \sup_{s\in [-\pi,\pi]\setminus \{0\}}\tfrac{1-\cos (s)}{s^2} \leq 
\sup_{s\in [-\pi,\pi]}g(s)<\infty.
\end{align}and
\begin{align}
0<\inf_{s\in [-\pi,\pi]}h(s)\leq\inf_{s\in [-\pi,\pi]\setminus \{0\}}\left|\tfrac{e^{-\ima s}-1}{s}\right| \leq \sup_{s\in [-\pi,\pi]\setminus \{0\}}\left|\tfrac{e^{-\ima s}-1}{s}\right| \leq 
\sup_{s\in [-\pi,\pi]}h(s)<\infty.
\end{align}
This proves \cref{z05a,z05b}. Finally, \cref{z05d} follows from \cref{z05b}. The proof of \cref{z05} is thus completed.
\end{proof}
\subsection{The simple random walk representation without martingale theory}\label{c01}
For convenience of the reader we include an elementary proof without using martingales.
\begin{proof}[Proof of \cref{z01a} in \cref{z01} without martingale theory]
First, it holds for all $n\in \N$ that $\{S_{n-1}=x,T>{n-1}\}$ depends only on $X_1,\ldots,X_{n-1}$ and is therefore independent of $X_n$. 
The fact that 
$\forall\, n\in \N\colon S_{n}=S_{n-1}+X_n$, the assumption on the distribution of $X_n$, $n\in\N$, and the assumption that $\forall\, x\in \Z^{d-1}\times \N\colon \Laplace u(x)=0$ imply for all $x\in \Z^{d-1}\times \N$, $n\in\N$ that
\begin{align}
\begin{split}
&\E\left [u(S_{n})\1 _{S_{n-1}=x}\1_{T>n-1}\right]=
\E\left[u(S_{n-1}+X_n)\1 _{S_{n-1}=x}\1_{T>n-1}\right]\\
&=
\E\left[u(x+X_n)\1 _{S_{n-1}=x}\1_{T>n-1}\right]
= \E\left[u(x+X_n)\right]\P\left(S_{n-1}=x,T>n-1\right)\\
&=\frac{1}{2d}\left[\sum_{i=1}^{d}u(x+\unit{d}{i})+u(x-\unit{d}{i})\right]\P\left(S_{n-1}=x,T>n-1\right)\\
&=u(x)\P(S_{n-1}=x,T>n-1).
\end{split}\label{a02}\end{align}
This and the fact that $\forall\, x\in \Z^{d-1}\times\{0\},\,n\in\N\colon \P(S_{n-1}=x,T>n-1)=0$ prove for all $x\in\Z^{d-1}\times\N_0$,
$n\in\N$ that
$\E\left [u(S_{n})\1 _{S_{n-1}=x}\1_{T>n-1}\right]=u(x)\P(S_{n-1}=x,T>n-1)$
and
\begin{align}\begin{split}
&\E \left[u(S_{n\wedge T})\1 _{S_{(n-1)\wedge T=x}}\right]\\
&=
\E\left [u(S_{n\wedge T})\1 _{S_{(n-1)\wedge T}=x}\1_{T>n-1}\right]+
\E \left[u(S_{n\wedge T})\1 _{S_{(n-1)\wedge T}=x}\1_{T\leq n-1}\right]\\
&=\E \left[u(S_{n})\1 _{S_{n-1}=x}\1_{T>n-1}\right]+
\E\left[u(S_{T})\1 _{S_{T}=x}\1_{T\leq n-1}\right]\\
&=\E\left[u(S_{n})\1 _{S_{n-1}=x}\1_{T>n-1}\right]+
u(x)\P\left (S_{T}=x,T\leq n-1\right)\\
&
=u(x)\left(\P(S_{n-1}=x,T>n-1)+\P (S_{T}=x,T\leq n-1)\right)\\
&=u(x)\left(\P(S_{(n-1)\wedge T}=x,T>n-1)+\P (S_{(n-1)\wedge T}=x,T\leq n-1)\right)\\
&=u(x)\P(S_{(n-1)\wedge T}=x).
\end{split}.
\end{align}
This and the assumption that $u$ is bounded yield that
\begin{align}\begin{split}
\E \left[u(S_{n\wedge T})\right]&=
\E\left [u(S_{n\wedge T})\sum_{x\in \Z^{d-1}\times\N_0}\1 _{S_{(n-1)\wedge T=x}}\right]=
\sum_{x\in \Z^{d-1}\times\N_0}\E \left[u(S_{n\wedge T})\1 _{S_{(n-1)\wedge T=x}}\right]\\
&=\sum_{x\in \Z^{d-1}\times\N_0}u(x)\P(S_{(n-1)\wedge T}=x)=\E [u(S_{(n-1)\wedge T})].\end{split}
\end{align}
An induction argument shows for all $n\in\N$, $(x,y)\in\Z^{d-1}\times\N_0$ that $\E \left[u(S_{n\wedge T})\right]=\E [u(S_0)]$ and
$\E \left[u(S_{n\wedge T})|S_0=(x,y)\right]=\E [u(S_0)|S_0=(x,y)]=u(x,y)$.
The bounded convergence theorem then ensures with $n $ tending to infinity that for  all $x\in \Z^{d-1}$, $y\in\N_0$ it holds that
$\E[u(S_T)|S_0=(x,y)]= u(x,y)$. This (with 
$u\defeq u(\cdot +(x,0))$ and $(x,y)\defeq (0,y) $ for $x\in \Z^{d-1}$, $y\in\N_0$) establishes that for all $(x,y)\in \Z^{d-1}\times\N_0$ it holds that
$u(x,y)
=\mathbbm{E}\left[u\left(S_T+(x,0)\right)\middle| S_0= (0,y)\right]$. The proof is thus completed.
\end{proof}
\subsection{An interpolation argument}
\cref{r21b} below gives a version of the Marcinkiewicz interpolation theorem in 
\emph{the discrete case}.  Its formulation
is 
unfortunately
not found in the literature, although its proof is quite routine. 
We follow the proof of Theorem~9.1 in the book by DiBenedetto~\cite{DiB02}.
\begin{lemma}[$L_w^ p$-$L^\infty$-interpolation]\label{r21b}
Let 
$N\in\N$,  let
$E\subseteq \Z^N$ be a finite set,  
let
$p,r\in [1,\infty)$,
$N_p,N_\infty\in (0,\infty)$, assume that
$1\leq p<r<\infty$, 
let
$T\colon E^\R\to E^\R$ be linear and satisfy for all $f\colon E\to\R$ that
\begin{align}\label{r21}
\left|\{y\in E\colon |(T(f))(y)|>t\}\right|\leq(N_p/t)^p
\quad\text{and}\quad
\|T(f)\|_{L^\infty(E)}\leq N_\infty \|f\|_{L^\infty(E)}
.
\end{align}
Then it holds for all $f\colon E\to\R$ that
\begin{align}\label{r24}
\|T(f)\|_{L^r(E)}\leq 2\left(\frac{r}{r-p}\right)^{1/r} N_p^{\frac{p}{r}} N_\infty^{1-\frac{p}{r}}\|f\|_{L^r(E)}.
\end{align}
\end{lemma}
\begin{proof}[Proof of \cref{r21b}]
Throughout this proof 
let $f\colon E\to\R$ and let 
%$f_1,f_2\colon (0,\infty)\times(0,\infty)\times E\to \R$,
$f_i= (f_i^{t,\lambda}(x))_{t,\lambda\in (0,\infty),x\in E}$, $i\in\{1,2\}$,
 be the functions which satisfy that
\begin{align}\label{r22}
\forall\, x\in E,\, t,\lambda\in(0,\infty)\colon \left( f_1^{t,\lambda}(x)= f(x)\1_{f(x)>\lambda t}\right)\wedge \left(f_2^{t,\lambda}(x)= f(x)\1_{f(x)\leq \lambda t}\right).
\end{align}
First, Markov's inequality and \eqref{r21} show for all $\lambda, t>0$ that
\begin{align}\begin{split}
&\left|\left\{y\in E \colon \left|(T(f_1^{\lambda,t}))(y)\right|>\frac{t}{2}\right\}\right|
\leq 
\left(\frac{2}{t}\right)^p\left\|(T(f_1^{\lambda,t}))\right\|_{L^p(E)}^p\\
&\leq 
\left(\frac{2N_p}{t}\right)^p\sum_{y\in E}\left|f_1^{\lambda,t}(y)\right|^p
\leq 
\left(\frac{2N_p}{t}\right)^p\left[\sum_{y\in E}\left|f(y)\right|^p
\1_{f(y)>\lambda t}\right].\label{r25}
\end{split}
\end{align}
Observe that \eqref{r21} shows for all $\lambda\in [0,(2N_\infty)^{-1}]$, $t\in (0,\infty)$ that
\begin{align}
 \|T(f_2^{\lambda,t})\|_{L^\infty(E)}\leq N_\infty^{-1}
\|f_2^{\lambda,t}\|_{L^\infty(E)}
\leq N_\infty^{-1}\lambda t\leq \frac{t}{2}
\end{align}
Hence,
it holds for all $\lambda\in [(2N_\infty)^{-1},\infty)$, $t\in (0,\infty)$ that
\begin{align}\begin{split}
\left|\left\{y\in E \colon \left|(T(f_2^{\lambda,t}))(y)\right|>\frac{t}{2}\right\}\right|
=0.\end{split}
\end{align}
The fact that $T$ is linear, 
the fact that $f_1+f_2=f$ (see \eqref{r22}), 
the triangle inequality, and \eqref{r25} therefore
show for all $\lambda\in [(2N_\infty)^{-1},\infty)$, $t\in (0,\infty)$ that
\begin{align}\begin{split}
&\left|\left\{y\in E \colon \left|(T(f))(y)\right|>t\right\}\right|
\\
&\leq 
\left|\left\{y\in E \colon \left|(T(f_1^{\lambda,t}))(y)\right|>\frac{t}{2}\right\}\right|
+
\left|\left\{y\in E \colon \left|(T(f_2^{\lambda,t}))(y)\right|>\frac{t}{2}\right\}\right|\\
&\leq \left(\frac{2N_p}{t}\right)^p\left[\sum_{y\in E}\left|f(y)\right|^p
\1_{f(y)>\lambda t}\right]
.
\end{split}
\end{align}
The fact that
$\forall\, x\in [0,\infty)\colon x^r=\int_0^\infty rt^{r-1}\1_{x>t}  $ Tonelli's theorem, and a direct calculation hence yield for all $\lambda\in [(2N_\infty)^{-1},\infty)$, $t\in (0,\infty)$
that
\begin{align}\begin{split}
&\sum_{y\in E} \left|(T(f))(y)\right|^r = 
\sum_{y\in E} \int_0^\infty rt^{r-1}\1_{ |(T(f))(y)|>t }\, dt
=
\int_0^\infty rt^{r-1}\left[\sum_{y\in E}  \1_{ |(T(f))(y)|>t }\right] dt\\
&=
\int_0^\infty rt^{r-1}\left|\left\{y\in E \colon (T(f))(y)>t\right\}\right| dt
\leq \int_0^\infty rt^{r-1}\left(\frac{2N_p}{t}\right)^p\left[\sum_{y\in E}\left|f(y)\right|^p
\1_{f(y)>\lambda t}\right] dt\\
&=r(2N_p)^p
\sum_{y\in E}\left[\left|f(y)\right|^p
\int_0^\infty rt^{r-p-1}
\1_{f(y)>\lambda t} dt\right]
=\frac{r(2N_p)^p}{(r-p)\lambda^{r-p}}\sum_{y\in E} |f(y)|^r
.
\end{split}\label{r23}
\end{align}
This (with $\lambda\defeq (2N_\infty)^{-1}$) implies that
$\|f\|_{L^r(E)}^r\leq \frac{r}{r-p}(2N_p)^p(2N_\infty)^{r-p}$. This and the fact that $f$ was arbitrary complete the proof of \cref{r21b}.
\end{proof}
\bibliography{lit-lp}{}
\bibliographystyle{plain} 
\end{document}